\documentclass[a4paper,12pt]{amsart}
\usepackage[english]{babel}
\usepackage[T1]{fontenc}
\usepackage{array}
\usepackage{makecell}
\usepackage[a4paper,left=1.8cm,right=1.8cm,top=2.2cm,bottom=2cm]{geometry}
\usepackage[justification=centering]{caption}

\usepackage{enumerate}
\usepackage{tabularx} 
\usepackage{amsmath}  
\makeatletter
\newcommand{\git}{\mathrm{/\hskip-3pt/}}
\newcommand{\xleftrightarrow}[2][]{\ext@arrow 3359\leftrightarrowfill@{#1}{#2}}
\newcommand{\xdashrightarrow}[2][]{\ext@arrow 0359\rightarrowfill@@{#1}{#2}}

\newcommand{\xdashleftarrow}[2][]{\ext@arrow 3095\leftarrowfill@@{#1}{#2}}
\newcommand{\xdashleftrightarrow}[2][]{\ext@arrow 3359\leftrightarrowfill@@{#1}{#2}}
\def\rightarrowfill@@{\arrowfill@@\relax\relbar\rightarrow}
\def\leftarrowfill@@{\arrowfill@@\leftarrow\relbar\relax}
\def\leftrightarrowfill@@{\arrowfill@@\leftarrow\relbar\rightarrow}
\def\arrowfill@@#1#2#3#4{%
  $\m@th\thickmuskip0mu\medmuskip\thickmuskip\thinmuskip\thickmuskip
   \relax#4#1
   \xleaders\hbox{$#4#2$}\hfill
   #3$%
}
\makeatother
\usepackage{graphicx} 

\usepackage{amsmath,amscd,amsbsy,amssymb,latexsym,url,bm,amsthm,mathrsfs}
\usepackage{epsfig,graphicx,subfigure}
\usepackage{enumerate,balance,mathtools}
\usepackage{wrapfig}
\usepackage{mathrsfs,euscript}
\usepackage[usenames]{xcolor}
\usepackage[vlined,boxed,commentsnumbered,linesnumbered,ruled]{algorithm2e}
\usepackage{amsfonts}
\usepackage{amsthm}
\usepackage{fancyhdr}
\usepackage{picinpar}
\usepackage{listings}
\usepackage{layout}
\usepackage[all,cmtip]{xy}
\usepackage{indentfirst}
\usepackage{tikz-cd}
\usepackage{tikz}
\usepackage[all]{xy}
\usepackage{mathrsfs}
\usepackage[colorlinks=true, allcolors=blue]{hyperref}
\usepackage{setspace}

\usepackage{amsmath,amsthm,amssymb,amsxtra,calligra,mathrsfs}
\usepackage{graphicx}
\usepackage{tikz}
\usepackage{tikz-cd} 
\usepackage{xcolor}
\usepackage{upgreek}
\usepackage{comment}
\usepackage{graphicx}
\usepackage{mathtools}
\usepackage{extarrows}
\usepackage{faktor}
\usepackage{mathrsfs,euscript}
\usepackage{comment}
\usepackage{bookmark}
\usepackage{hyperref}
\usepackage{mathrsfs}
\usepackage{multirow}

\hypersetup{
    colorlinks=true, 
    citecolor=blue,
    linkcolor=magenta,
    pdfauthor={},
	bookmarksnumbered=true,
}

\usepackage[backend=biber, citestyle=alphabetic, bibstyle=trad-abbrv, style=alphabetic, url=false,doi=false,isbn=false,sorting=nyt,maxbibnames=9, giveninits, sortcites = true]{biblatex}
\renewbibmacro{in:}{}

\newcommand{\wh}[1]{\widehat{#1}}

\newcommand{\wt}[1]{\widetilde{#1}}

\newcommand{\mb}[1]{\mathbb{#1}}

\newcommand{\ove}[1]{\overline{#1}}

\newcommand{\mtc}[1]{\mathcal{#1}}
\newcommand{\mtf}[1]{\mathfrak{#1}}
\newcommand{\mts}[1]{\mathscr{#1}}

\DeclareMathOperator{\coeff}{coeff}
\DeclareMathOperator{\Fut}{Fut}

\DeclareMathOperator{\PGL}{PGL}

\DeclareMathOperator{\loc}{loc}

\DeclareMathOperator{\Spec}{Spec}

\DeclareMathOperator{\mult}{mult}

\DeclareMathOperator{\Pic}{Pic}

\DeclareMathOperator{\ord}{ord}

\DeclareMathOperator{\Def}{Def}

\DeclareMathOperator{\SL}{SL}
\DeclareMathOperator{\Sym}{Sym}

\DeclareMathOperator{\GIT}{GIT}

\DeclareMathOperator{\Aut}{Aut}

\DeclareMathOperator{\Ind}{Ind}

\DeclareMathOperator{\vol}{vol}

\DeclareMathOperator{\QG}{QG}

\DeclareMathOperator{\CM}{CM}
\DeclareMathOperator{\Bl}{Bl}

\newcommand{\bG}{\mathbb{G}}

\newcommand{\sslash}{\mathbin{\mkern-3mu/\mkern-6mu/\mkern-3mu}}

\newcommand{\sheafHom}{\mathscr{H}\text{\kern -3pt {\calligra\large om}}\,}

\newcommand{\bZ}{{\mathbb Z}}

\newcommand{\bQ}{{\mathbb Q}}
\newcommand{\bP}{{\mathbb P}}

\newcommand\cD{{\mathcal{D}}}

\newcommand\cL{\mathcal{L}}

\newcommand\cO{{\mathcal{O}}}

\newcommand\cM{{\mathcal{M}}}

\newtheorem{theorem}{Theorem}[section]
\newtheorem{lemma}[theorem]{Lemma}
\newtheorem*{convention}{Convention}
\newtheorem{corollary}[theorem]{Corollary}
\newtheorem{prop}[theorem]{Proposition}

\theoremstyle{definition}
\newtheorem{defn}[theorem]{Definition}
\newtheorem{remark}[theorem]{Remark}

\title{K-moduli of log del Pezzo pairs and variations of GIT}

\author{Jesus Martinez-Garcia}
\address{Departamento de Matem\'aticas, Universidad Aut\'onoma de Madrid, {Spain}} \email{jesus.martinezg@uam.es}
\address{{School of Mathematics, Statistics and Actuarial Science}, {University of Essex}, {UK}} \email{jesus.martinez-garcia@essex.ac.uk}

\author{Theodoros Stylianos Papazachariou}
\address{{Yau Mathematical Sciences Center}, {Tsinghua University}, {Haidian District, Beijing}, {China}} \email{tpapazachariou@mail.tsinghua.edu.cn}

\author{Junyan Zhao}
\address{Department of Mathematics, William E. Kirwan Hall, 4176 Campus Dr., University of Maryland, College Park, MD 20742-4015, USA}
\email{jzhao81@umd.edu}

\begin{document}
\begin{abstract}
We study the K-moduli of log del Pezzo pairs formed by a del Pezzo surface of degree $d$ and an anti-canonical divisor. These moduli spaces naturally depend on one parameter, providing a natural problem in variations of K-moduli spaces. For degrees 2, 3, 4, we establish an isomorphism between the K-moduli spaces and variations of Geometric Invariant Theory compactifications, which generalizes the isomorphisms in the absolute cases established by Odaka--Spotti--Sun and Mabuchi--Mukai.

\end{abstract}

\maketitle

\tableofcontents

\section{Introduction}

K-stability has successfully provided the construction of moduli spaces of Fano varieties and log Fano pairs. The general K-moduli theorem (cf. Theorem \ref{17}) is a combination of substantial work by a number of people (cf. \cite{ABHLX20,BHLLX21,BLX22,BX19,CP21,Jia20,LWX21,LXZ22,Xu20,XZ20, OSS16, SSY, XZ21}), which establishes that for fixed dimension $n$ and log-anticanonical volume $v$, the functor of K-semistable log Fano pairs is represented by a separated Artin stack, which admits a projective good moduli space (in the sense of \cite{Alp13});  the closed points of the moduli space are in bijection with isomorphic classes of  K-polystable log Fano pairs. 

Compact K-moduli spaces $\ove{M}^{K}(c)$ of K-polystable log Fano pairs $(X, cD)$ with a coefficient $c$ give rise to wall-crossing phenomena as $c$ varies, which is important as it allows us to relate many birational moduli spaces to each other by providing an explicit resolution of the rational maps between them. Perturbations of $c$ result, locally on $\ove{M}^{K}(c)$, into variations of geometric invariant theory (VGIT) quotients (cf. \cite{ADL24,Zho23}). In recent years, significant efforts have been made to produce specific examples of K-moduli wall crossings (cf. \cite{MG13, CMG16,ADL21,ADL22,GMG21,zha22, Pap22,zha23,zha23b,PSW23}). However, so far there are no completely understood examples where both the varieties and the boundary divisors have moduli, and we address that gap in this paper. It is also worth noting that recent years have witnessed significant progress in understanding the K-moduli of Fano threefolds. In particular, many of these K-moduli compactifications have been explicitly related to GIT or VGIT constructions, which greatly facilitates the study of their global properties (e.g., see \cite{ADL22,LX19,3fold-dim1-Kmoduli,LZ25,KLPZ26,Zha26, Casagrande-Druel})

The goal of this paper is to study log del Pezzo pairs of degree $d>1$, especially for $d=2,3,4$, where both surfaces and divisors can vary. We denote by ${\mtc{M}}^{K}_d(c)$ the irreducible component of the K-moduli stack of $\mathbb Q$-Gorenstein smoothable log pairs $(X,cD)$, where $X$ is a del Pezzo surface of degree $d$ and $D\in |-K_X|$, and denote by $\ove{M}^K_d(c)$ its good moduli space. Note that these pairs are natural examples of strongly asymptotically log del Pezzo surfaces \cite{ALF, ALF-ample-angles} (see \cite{MG25} for a short survey).


For degrees $d=3,4$, since an anticanonical divisor of a smooth degree $d$ del Pezzo surface $X\subseteq \bP^d$ is a hyperplane section of $X$, there is a natural way --- via VGIT --- to construct compactified moduli spaces of smooth log pairs $(X,D)$ (cf. Section \ref{subsection:VGIT}). These moduli spaces, denoted by $\ove{M}^{\GIT}_d(t)$ depend on a parameter $t$ with a wall-chamber decomposition (cf. \cite{GMG19, Pap22}).


Our first result establishes isomorphisms between the VGIT moduli spaces and the K-moduli spaces for the above log del Pezzo pairs, preserving both wall-crossing structures.

\begin{theorem}[{cf. Theorem \ref{13}}]\label{31}
	Let $c\in(0,1)$ be a rational number and $d=3,4$. Let $\mtc{M}^{\GIT}_d(t(c))$ be the GIT moduli stack of log del Pezzo pairs of degree $d$. Let $t(c)=\frac{9c}{8+c}$ if $d=3$, and $t(c)=\frac{6c}{5+c}$ if $d=4$. Then there is an isomorphism between Artin stacks $${\mtc{M}}^{K}_d(c)\simeq {\mtc{M}}^{\GIT}_d(t(c)),$$ which descends to an isomorphism ${\ove{M}}^{K}_d(c)\simeq {\ove{M}}^{\GIT}_d(t(c))$ between the corresponding good moduli spaces. In particular, these isomorphisms commute with the wall-crossing morphisms.
\end{theorem}

In \cite{GMG19,GMG21} and \cite{Pap22}, the authors work out the VGIT moduli spaces for $d=3,4$ respectively, and give a complete geometric characterization of the singularities of the pairs present in the VGIT compactification for each of the chambers. In fact, the authors prove the above isomorphisms between K-moduli spaces and VGIT moduli spaces when the coefficient $c$ is small (but not beyond the second chamber). Theorem \ref{31}, in conjunction with \cite{GMG19, GMG21, Pap22} thus gives a full geometric description of all the log pairs represented in the K-moduli space for all $c$.~\\


Del Pezzo surfaces of degree $2$ have a richer geometry than the higher degree analogues. To present our result on their moduli (Theorem \ref{46}), we first need to introduce some considerations. Let $Y$ be a del Pezzo surface of degree $2$ with canonical singularities. Its anti-canonical divisor $-K_Y$ is base-point free and induces a $2:1$ morphism from $Y$ to $\mathbb P^2$ branched along a quartic curve $C$, where curves $D\in |-K_Y|$ are mapped to lines $C'\subset \mathbb P^2$. Using the double cover, we  relate the K-stability of $(Y, cD)$ to that of $(\mathbb P^2, \frac{1}{2}C+cC')$ \cite[Example 4.2]{Der16}. We denote by $\mtc{M}^K_{\mb{P}^2}(c)$ the moduli stack of log Fano pairs with a $\mb{Q}$-Gorenstein smoothing to pairs $(\mathbb P^2, \frac{1}{2}C+cC')$, and its good moduli space by $\ove{M}^K_{\mb{P}^2}(c)$. We will prove that ${\ove{M}}^{K}_2(c)$ and $\ove{M}^K_{\mb{P}^2}(c)$ are isomorphic (cf. Theorem \ref{42}). We need to caution the reader that even if we denote the moduli space by $\ove{M}^K_{\mb{P}^2}(c)$, there are pairs $(X,\frac{1}{2}D_1+cD_2)$ in the moduli such that $X\simeq \mb{P}(1,1,4)$. 

Natural GIT compactifications can be considered for pairs $(C, C')$ in $\mathbb P^2$, where $C$ is a quartic curve and $C'$ is a line, and $a,b\in \mathbb Z_{>0}$ (cf. \cite{Laza-thesis, Laz09}):
$$\ove{M}^{\GIT}_{\mathbb P^2}(t)\coloneqq\big(\mathbb P H^0\left(\mathbb P^2, \mathcal O_{\mathbb P^2}(4)\right)\times\mathbb PH^0\left(\mathbb P^2, \mathcal O_{\mathbb P^2}(1)\right)\big)\sslash_{\mtc{O}(a,b)} \PGL(3),$$ 
where $t=\frac{b}{a}$. Similarly, let $\mtc{M}^{\GIT}_{\mb{P}^2}(t)$ be its GIT moduli stack. 

It is well known that smooth del Pezzo surfaces of degree $2$ can degenerate to a double cover $Y$ of $\mb{P}(1,1,4)$ branched along an octic hyperelliptic curve $(z^2-f_8(x,y)=0)$ (e.g. see \cite{OSS16}). The double cover is induced by the linear system $|-K_Y|$ and anti-canonical curves of $Y$ are mapped to degree $2$ curves in $\mb{P}(1,1,4)$ (in fact, the union of two lines, as there are no irreducible curves of degree $2$ in $\mathbb P(1,1,4)$). Let $\ove{M}^K_{\mb{P}(1,1,4)}(c)$ be the closure in the K-moduli space of the classes of pairs $(\mb{P}(1,1,4), \frac{1}{2}C+cC')$ where $C$ is an octic and $C'$ is a quadric, with reduced scheme structure. 
Similarly, as for $\mathbb P^2$, we obtain new natural GIT compactifications of pairs $(C,C')$ in $\mathbb P(1,1,4)$:
$$\ove{M}^{\GIT}_{\mathbb P(1,1,4)}(t)\coloneqq\left(\mathbb P H^0\left(\mathbb P^1, \mathcal O_{\mathbb P^1}(8)\right)\times\mathbb PH^0\left(\mathbb P^1, \mathcal O_{\mathbb P^1}(2)\right)\right)\sslash_{\mtc{O}(a,b)} \PGL(2)$$
where $t=\frac{b}{a}$. 
We now present our second result.

\begin{theorem}[{{cf. Corollary \ref{36}, Theorems \ref{19}, \ref{20}, \ref{42}, \ref{43}}}]
	\label{46} 
	Let $c\in(0,1)$ be a rational number, and $\ove{M}^K_2(c)$ be the K-moduli space of log del Pezzo pairs $(Y,cD)$ of degree $2$, where $D\in|-K_Y|$. If  $c>\frac{1}{4}$, then we have a canonical isomorphism
	$${\mtc{M}}^K_{\mb{P}^2}(c)\ \simeq \ \left[\wt{\mtc{U}}^{ss}_{t(c)}/\PGL(3)\right],$$
	where $t(c)=2c$. The latter quotient stack is obtained by taking the Kirwan blowup $\wt{\mtc{U}}^{ss}_{t(c)}$ of the VGIT moduli stack $\mtc{M}^{\GIT}_{\mb{P}^2}(t(c))$ at the point parameterizing the pair $(2C,L)$, where $C$ is a smooth conic and $L$ is a line intersecting $C$ transversely. This isomorphism descends to a canonical isomorphism of the good moduli spaces:
	$$\ove{M}^K_2(c)\ \simeq\  \ove{M}^K_{\mb{P}^2}(c)\ \simeq \ \wt{\mtc{U}}^{ss}_t\sslash_{\mtc{O}(a,b)}\PGL(3),$$
	where $t(c)=\frac{b}{a}$. For any $c\in (0,1)$, the following holds.
	\begin{enumerate}
		\item If $(X,\frac{1}{2}C+cC')$ is a K-semistable pair in $\mtc{M}^K_{\mb{P}^2}(c)$, then $X\simeq \mb{P}^2$ or $X\simeq \mb{P}(1,1,4)$.
		\item A pair $(\mathbb P^2,\frac{1}{2}C+cC')$ is K-(semi/poly)stable if and only if the pair $(C,C')$ is GIT$_{t(c)}$-(semi/poly)stable, where $t(c)=2c$.
		\item A pair $(\mb{P}(1,1,4),\frac{1}{2}C+cC')$ is K-(semi/poly)stable if and only if $(C,C')$ is GIT$_{t'(c)}$-(semi/poly)stable, where $t'(c)=\frac{12c}{1-c}$. Moreover, there is a canonical isomorphism
		$$\ove{M}^K_{\mb{P}(1,1,4)}(c)\ \simeq \ \ove{M}^{\GIT}_{\mathbb P(1,1,4)}(t'(c)).$$ 
	\end{enumerate}
\end{theorem}

Theorem \ref{31} and Theorem \ref{46} generalize the isomorphisms between the GIT moduli spaces and the K-moduli spaces established analytically in \cite{OSS16} and \cite{MM20} to the log pair case (and algebraically). Moreover, as a consequence, we determine all the walls for the K-moduli spaces when $d=2,3,4$ by GIT analysis of their VGIT counterparts. We also characterize geometrically all the K-polystable elements for each wall and chamber, by looking at their VGIT counterparts, since our Theorem implies that GIT$_t$-(semi/poly)stability coincides with $c(t)$ (semi/poly)stability. We refer the reader to Appendix \ref{app:deg2} for a detailed characterization for degree $2$ and to \cite{GMG19} and \cite{Pap22} for degrees $3$ and $4$, respectively.

Unfortunately, in the case of $d=2$, our canonical isomorphism of stacks in Theorem \ref{46} does not extend to  $0<c<\frac{1}{4}$. We refer the reader to Remark \ref{copout}, for some discussion on why this is. However, we do point out that Theorem \ref{46} allows us to obtain a full geometric characterization of K-semistable, K-polystable and K-stable log del Pezzo pairs, along with the results of Appendix \ref{app:deg2}.

For higher degrees $d\geqslant 5$, the surface does not have moduli, which renders finding a GIT compactification that the K-moduli is isomorphic to a bit pointless. Nevertheless, we find the walls and a geometric characterization of all K-polystable pairs (see \S\ref{sec:higher-degree} and Appendix \ref{app:deg5-walls} for a detailed description).

\begin{theorem}[{{cf. Theorem \ref{44}, \ref{45}, Wall crossings of K-moduli}}]
	\label{thm:walls}
	Let $c\in(0,1)$ be a rational number, $d\in\{2,...,9\}$ be an integer. Let $\ove{M}^K_{d}(c)$ be the K-moduli spaces of del Pezzo pairs $(X,cD)$, which admit a $\mb{Q}$-Gorenstein smoothing to a pair $(\Sigma_d,cC)$, where $\Sigma_d$ is a smooth del Pezzo surface of degree $d$, and $C\in|-K_{\Sigma_d}|$ is a smooth curve. We say a wall $c$ is an \emph{interior} wall if $\ove{M}^K_{d}(c\pm \varepsilon)\neq \emptyset$ for $0<\varepsilon\ll1$.
	\begin{enumerate}
		\item If $d=9$, then there are no (interior) walls and only one chamber. Moreover, $\ove{M}^K_9(c)$ is canonically isomorphic to the GIT compactification as cubic curves in $\mathbb P^2$ of the moduli of elliptic curves. 
		\item If $d=8$ and $\Sigma_8=\mb{P}^1\times\mb{P}^1$, then there is a unique (interior) wall at $c=\frac{1}{4}$ and two chambers.
		\item If $d=8$ and $\Sigma_8=\Bl_p\mb{P}^2$, then there are two walls, the first wall at $c=\frac{1}{5}$ and an interior wall at $\frac{1}{4}$ and two chambers.
		\item If $d=7$, then there are three walls: the first wall at $c=\frac{4}{25}$ and two interior walls $\frac{2}{9},\frac{2}{5}$ and three chambers.
		\item If $d=6$, then there are five (interior) walls $c=\frac{2}{11},\frac{1}{4},\frac{5}{14},\frac{2}{5},\frac{1}{2}$ and seven chambers.
		\item If $d=5$, then there are six (interior) walls $c=\frac{2}{17},\frac{4}{19},\frac{2}{7},\frac{8}{23},\frac{4}{9},\frac{4}{7}$ and eight chambers.
		\item If $d=4$, then there are five (interior) walls $c=\frac{1}{7},\frac{1}{4},\frac{1}{3},\frac{1}{2},\frac{5}{8}$ and seven chambers.
		\item If $d=3$, then there are five (interior) walls $c=\frac{2}{11},\frac{4}{13},\frac{2}{5},\frac{10}{19},\frac{2}{3}$ and seven chambers.
		\item If $d=2$, then there are eight (interior) walls $c=\frac{1}{13},\frac{1}{7},\frac{1}{5},\frac{1}{4},\frac{2}{5},\frac{1}{2},\frac{4}{7},\frac{7}{10}$ and ten chambers.
	\end{enumerate}

\end{theorem}


\subsection*{Structure of the paper}
In Section \ref{sec:preliminaries} we recover some of the basic definitions on K-stability and K-moduli, as well as the VGIT construction for log del Pezzo pairs. In Section \ref{sec:deg-3-4} we study the K-moduli of del Pezzo pairs of degrees $3$ and $4$, proving Theorem \ref{31}. In Section \ref{sec:deg2} we study the K-moduli of del Pezzo pairs of degree $2$, proving Theorem \ref{46}. In Section \ref{sec:higher-degree} we tackle degrees $5$--$9$, proving Theorem \ref{thm:walls}. Some technical results have been relegated to the appendices. In Appendix \ref{app:deg2}, we describe the VGIT quotients of quartic plane curves (and their degenerations) that are used in the proof of Theorem \ref{46} and in Appendix \ref{app:deg5-walls} we describe the strictly K-polystable pairs appearing in the walls for degrees $5$ and higher.

\subsection*{Acknowledgements}
We thank Dori Bejleri, Izzet Coskun, Ruadha{\'i} Dervan, Tiago Duarte Guerreiro, Patricio Gallardo, Yuchen Liu, Yuji Odaka and Cristiano Spotti for advice and useful comments. JMG was partially supported by EPSRC grant EP/V055399/1, by Grant RYC2024-049909-I funded by MICIU/AEI/10.13039/501100011033 and by ESF+, and by project  PID2025-174260NB-C21 funded by MICIU/AEI/10.13039/501100011033. TP would like to thank the Isaac Newton Institute for Mathematical Sciences, Cambridge, for support and hospitality during the programme ``New equivariant methods in algebraic and differential geometry, where work on this paper was undertaken''. This work was supported by EPSRC grant EP/R014604/1. TP was also funded by a postdoctoral fellowship associated to the Royal Society University Research Fellowship held by Ruadha{\'i} Dervan. An important part of this paper was written during a visit by TP and JZ to JMG supported by two University of Essex (UoE) School of Mathematics, Statistics and Actuarial Sciences Research and Impact Fund grants. We thank UoE for their support.

\section{Preliminaries}
\label{sec:preliminaries}
\begin{convention}\textup{
		Throughout this paper, we work over the field of complex numbers $\mb{C}$. By a surface, we mean a connected normal projective algebraic surface over $\mb{C}$. For notions and properties of singularities of surface pairs, we refer the reader to \cite[Chapter 2, Chapter 4]{KM98}.}
\end{convention}

\subsection{K-stability of log Fano pairs}

In this section, we follow conventions and notation from \cite{ADL24} and \cite{CA23}.
\begin{defn}
	Let $X$ be a normal projective variety, and $D$ be an effective $\mb{Q}$-divisor on $X$. Then the pair $(X,D)$ is called a \textup{log Fano pair} if $-(K_X+D)$ is an ample $\mb{Q}$-Cartier divisor and a \textup{log del Pezzo pair} if in addition $\dim(X)=2$. A normal projective variety $X$ is called a \textup{$\mb{Q}$-Fano variety} if $(X,0)$ is a Kawamata log terminal (\textup{klt}) log Fano pair. A ($\mathbb Q$-Fano) variety of dimension $2$ is a \textup{del Pezzo surface}.
\end{defn}

\begin{defn}
	Let $(X,D)$ be an $n$-dimensional log Fano pair, and $L$ be an ample line bundle on $X$ which is $\mb{Q}$-linear equivalent to $-k(K_X+D)$ for some positive rational number $k\in\mb{Q}$. Then a \textup{normal test configuration (abbreviated TC)} $(\mtc{X},\mtc{D};\mtc{L})$ of $(X,D;L)$ consists of 
	\begin{itemize}
		\item a normal projective variety $\mtc{X}$ with a flat projective morphism $\pi:\mtc{X}\rightarrow \mb{A}^1$;
		\item a line bundle $\mtc{L}$ ample over $\mb{A}^1$;
		\item a $\mb{G}_m$-action on the polarized variety $(\mtc{X},\mtc{L})$ such that $\pi$ is $\mb{G}_m$-equivariant, where $\mb{A}^1$ is equipped with the standard $\mb{G}_m$-action;
		\item a $\mb{G}_m$-equivariant isomorphism between the restriction $(\mtc{X}\setminus \mtc{X}_0;\mtc{L}|_{\mtc{X}\setminus \mtc{X}_0})$ and $(X;L)\times (\mb{A}^1\setminus\{0\})$;
		\item an effective $\mb{Q}$-divisor $\mtc{D}$ on $\mtc{X}$ such that $\mtc{D}$ is the Zariski closure of $D\times (\mb{A}^1\setminus\{0\})$ in $\mtc{X}$, under the identification between $\mtc{X}\setminus \mtc{X}_0$ and $X\times (\mb{A}^1\setminus\{0\})$.
	\end{itemize}
	A TC is called
	\begin{itemize}
		\item a \textup{product test configuration} if 
		\begin{equation}
			(\mtc{X},\mtc{D};\mtc{L})\simeq (X\times \mb{A}^1,D\times \mb{A}^1;p_1^{*}L\otimes \mtc{O}_{\mtc{X}}(l\mtc{X}_0))
			\label{eq:product-TC}
		\end{equation}
		for some $l\in \mb{Z}$;
		\item a \textup{trivial test configuration} if it is a product TC and the isomorphism \eqref{eq:product-TC} is $\mb{G}_m$-equivariant, where the $\mathbb G_m$-action on $X$ is trivial;
		\item a \textup{special test configuration} if $\mtc{L}\sim_{\mb{Q}}-k(K_{\mtc{X}/\mb{A}^1}+\mtc{D})$ and $(\mathcal X, \mathcal D + \mathcal X_0)$ is purely log terminal (plt). 
	\end{itemize}
	
	The \textup{generalized Futaki invariant} of a normal test configuration $(\mtc{X},\mtc{D};\mtc{L})/\mb{A}^1$ is $$\Fut(\mtc{X},\mtc{D};\mtc{L})\coloneqq\frac{1}{(-K_X-D)^n}\left(\frac{n}{n+1}\cdot\frac{\ove{\mtc{L}}^{n+1}}{k^{n+1}}+\frac{(\ove{\mtc{L}}^n.(K_{\ove{\mtc{X}}/\mb{P}^1}+\ove{\mtc{D}}))}{k^n}\right),$$ where $(\ove{\mtc{X}},\ove{\mtc{D}};\ove{\mtc{L}})$ is the natural compactification of $(\mtc{X},\mtc{D};\mtc{L})$ over $\mb{P}^1$. 
\end{defn}

Now we define  K-stability for  log Fano pairs.

\begin{defn}
	A log Fano pair $(X,D)$ is called 
	\begin{enumerate}[(1)]
		\item \textup{K-semistable} if  $\Fut(\mtc{X},\mtc{D};\mtc{L})\geq0$ for any normal test configuration $(\mtc{X},\mtc{D};\mtc{L})/\mb{A}^1$ and any $k\in\mb{Q}$ such that $L$ is Cartier;
		\item \textup{K-polystable} if it is K-semistable, and $\Fut(\mtc{X},\mtc{D};\mtc{L})=0$ for some TC $(\mtc{X},\mtc{D};\mtc{L})$ if and only if it is a product TC.
		\item \textup{K-stable} if it is K-semistable, and $\Fut(\mtc{X},\mtc{D};\mtc{L})=0$ for some TC $(\mtc{X},\mtc{D};\mtc{L})$ if and only if it is a trivial TC.
	\end{enumerate}
\end{defn}

There is an equivalent formulation of K-(semi)stability in terms of birational geometry, by means of the following invariants:
\begin{defn}
	Let $(X,D)$ be an $n$-dimensional log Fano pair, and $E$ a prime divisor on a normal projective variety $Y$, where $\pi:Y\rightarrow X$ is a birational morphism. Then the \textup{log discrepancy} of $(X,D)$ with respect to $E$ is $$A_{(X,D)}(E)\coloneqq1+\coeff_{E}(K_Y-\pi^{*}(K_X+D)).$$ 
	The \textup{expected vanishing order} of $(X,D)$ with respect to $E$ (also referred to as the \emph{$S$-invariant}) is $$S_{(X,D)}(E)\coloneqq \frac{1}{(-K_X-D)^n}\int_0^{\tau}\vol(\pi^*(-K_X-D)-xE)dx,$$ 
    where $\tau$ is the pseudo-effective threshold
    $$\tau=\tau(\pi^*(-K_X-D); E)=\sup \{x \, \colon \, \pi^*(-K_X-D)-xE\text{ is big}\}.$$
	The \textup{$\beta$-invariant} of $(X,D)$ with respect to $E$ is then 
	$$\beta_{(X,D)}(E)\coloneqq A_{(X,D)}(E)-S_{(X,D)}(E).$$
\end{defn}

\begin{theorem}[cf. {\textup{\cite{Fuj19,Li17,BX19}}}]
	The following assertions hold:
	\begin{itemize}
		\item $(X,D)$ is K-stable if and only if $\beta_{(X,D)}(E)>0$ for every prime divisor $E$ over $X$;
		\item $(X,D)$ is K-semistable if and only if $\beta_{(X,D)}(E)\geq0$ for every prime divisor $E$ over $X$.
	\end{itemize}
\end{theorem}

\begin{theorem}[cf. {\textup{\cite{Fuj18,LL19,Liu18}}}]
	\label{16}
	Let $(X,D)$ be an $n$-dimensional K-semistable log Fano pair. Then for any $x\in X$, we have $$(-K_X-D)^n\leq \left(1+\frac{1}{n}\right)^n\widehat{\vol}(x,X,D),$$
	where $\widehat{\vol}(x,X,D)$ is the \textup{normalized volume} of $(X, D)$ at $x$.
\end{theorem}

The following two results enable us to deduce the (semi)stability of the central fiber in a test configuration from invariant conditions and (semi)stability of general fibers.

\begin{lemma}[cf. \cite{Kem78}]\label{3}
	Let $(X,L)$ be a polarized projective variety with an action by a reductive group $G$. Let $\sigma:\mb{G}_m\rightarrow G$ be a 1-parameter subgroup, and $x\in X$ a GIT semistable point such that its Hilbert-Mumford weight $\mu^{L}(x,\sigma)=0$. Then the limit $x_0=\lim_{t\to 0}\sigma(t)\cdot x$ is also GIT semistable.
\end{lemma}

\begin{lemma}[cf. \cite{LWX21}]\label{4}
	Let $(X,D)$ be a log Fano pair, and $(\mtc{X},\mtc{D};\mtc{L})\rightarrow \mb{A}^1$ is a normal TC. If $(X,D)$ is K-semistable and $\Fut(\mtc{X},\mtc{D};\mtc{L})=0$, then $(\mtc{X},\mtc{D};\mtc{L})$ is a special TC, and the central fiber $(\mtc{X}_0,\mtc{D}_0)$ is also K-semistable.
\end{lemma}

\begin{defn}[{cf. \cite[\S 1.3]{CA23}}]
	Let $(X,D)$ be a log del Pezzo (surface) pair with klt singularities. If the maximal torus of $\Aut^0(X,D)$ is of dimension $1$ (i.e. isomorphic to $\mb{G}_m$), then we call $(X,D)$ a \textup{complexity-one $\mb{T}$-pair}. A $\mb{G}_m$-equivariant divisor $F$ over a complexity-one log del Pezzo $\mb{T}$-pair $(X,D)$ is called \textup{vertical} if a maximal $\mb{G}_m$-orbit in $F$ has dimension $1$; otherwise it is called \textup{horizontal}.
\end{defn}

\begin{theorem}[{{cf. \cite[Theorem 1.31]{CA23}}}]\label{30}
	Let $(X,D)$ be a log del Pezzo $\mb{T}$-pair of complexity-one, with $\mb{G}_m$-action $\lambda$. The pair $(X,D)$ is K-polystable if and only if the following conditions hold:
	\begin{enumerate}[(1)]
		\item $\beta_{(X,D)}(F)>0$ for each vertical $\mb{G}_m$-equivariant prime divisor $F$ on $X$,
		\item $\beta_{(X,D)}(F)=0$ for each horizontal $\mb{G}_m$-equivariant prime divisor $F$ on $X$, and
		\item $\Fut_{(X,D)}(\overline{\lambda \cdot (X,D)})=0$ where $\overline{\lambda \cdot (X,D)}$ is the TC induced by taking the closure of the $\lambda$-orbit of $(X,D)$. 
	\end{enumerate}
\end{theorem}

\begin{lemma}
	\label{lemma:refined-complexity-one}
	Let $(X, D)$ be a log del Pezzo $\mathbb T$-pair of complexity-one with $\mathbb G_m$-action $\lambda$, with no horizontal divisors such that for each vertical divisor $F$, $\beta(F)>0$. Then $(X, D)$ is K-polystable if and only if for any $\lambda$-fixed point $p\in F$, where $F$ is any vertical divisor, $\beta_{(X, D)}(E)=0$ where $E$ is the exceptional divisor of the plt blowup of $(X, D)$ at $p$. In fact, it is enough to check this condition for one fixed point of each vertical divisor $F$.
\end{lemma}
\begin{proof}
	We apply Theorem \ref{30}. By hypothesis, conditions (1) and (2) are satisfied. By \cite[Lemma 3.4]{Xu21} we have that $\Fut_{(X,D)}(\overline{\lambda \cdot (X,D)})=\beta(E)$. Note that given $q\in F$, a vertical divisor, then the $\lambda$-invariant points in $F$ are $\lim_{t\rightarrow 0}\lambda(t)\cdot q=p_1$ and $\lim_{t\rightarrow \infty}\lambda(t)\cdot q = p_2$. Let $E_1, E_2$ be the exceptional divisors of their plt-blow ups. By construction, we have $\beta(E_1)+\beta(E_2)=0$, so it is enough to check $\beta(E_1)=0$. Note vertical divisors cover the whole set of fixed points (since otherwise we would have a horizontal divisor).
\end{proof}

\subsection{K-moduli spaces and CM line bundles}
In this part, we collect some results on the K-moduli spaces of log Fano pairs, following \cite{ADL24} and \cite{GMG19}.

\begin{defn}
	Let $f:(\mtc{X},\mtc{D})\rightarrow B$ be a proper flat morphism with normal geometrically connected fibers of pure dimension $n$ to a reduced scheme $B$, where $\mtc{D}$ is an effective $\mb{Q}$-divisor on $\mtc{X}$ which does not contain any fiber of $f$. Then $f$ is a \textup{$\mb{Q}$-Gorenstein flat family of log Fano pairs} if $-(K_{\mtc{X}/B}+\mtc{D})$ is $\mb{Q}$-Cartier and ample over $B$.
\end{defn}

\begin{defn}
	Let $0<c<\frac{1}{r}$ be a rational number and $(X,cD)$ be a log Fano surface pair such that $D\sim -rK_X$. A $\mb{Q}$-Gorenstein flat family of log Fano pairs $f:(\mtc{X},c\mtc{D})\rightarrow C$ over a pointed smooth curve $(0\in C)$ is called a \textup{$\mb{Q}$-Gorenstein smoothing} of $(X,D)$ if  
	\begin{enumerate}[(1)]
		\item the divisors $\mtc{D}$ and $K_{\mtc{X}/C}$ are both $\mb{Q}$-Cartier, $f$-ample, and $\mtc{D}\sim_{\mb{Q},f}-rK_{\mtc{\mtc{X}/C}}$;
		\item both $f$ and $f|_{\mtc{D}}$ are smooth over $C\setminus\{0\}$, and
		\item $(\mtc{X}_0,c\mtc{D}_0)\simeq (X,cD)$.
	\end{enumerate}
\end{defn}

Now let $f:\mtc{X}\rightarrow S$ be a proper flat family of connected schemes with $S_2$ fibers of pure dimension $n$, and $\mtc{L}$ be a line bundle on $\mtc{X}$ which is ample over $S$. It follows from \cite{KM76} that for sufficiently divisible $k\gg 0$, we have a decomposition $$\det f_{!}(\mtc{L}^k)=\lambda_{n+1}^{\binom{k}{n+1}}\otimes\lambda_{n}^{\binom{k}{n}}\otimes\cdots\otimes\lambda_0^{\binom{k}{0}}$$ for some line bundles $\lambda_i=\lambda_i(\mtc{X},\mtc{L})$ on $S$. Since $f$ is flat, then each fiber $(\mathcal X_t, \mathcal L_t)$, $t\in S$, has the same Hilbert polynomial for sufficiently divisible $m\gg 0$
$$\chi(\mtc{X}_t,\mtc{L}^m_t)=a_0m^n+a_1m^{n-1}+\cdots+a_{n}$$
by Riemann-Roch's Theorem.






Let $\mtc{D}=\sum_{i=1}^l \mtc{D}_i$ be a relative Mumford $\mb{Z}$-divisor \cite[Definition 1]{Kol18} on $\mtc{X}$ over $S$. Assume that $\mtc{D}_i$ is flat over $S$ for each $i$. Let $\widetilde a_0$ be the leading coefficient of the Hilbert polynomial $\chi(\mtc{D}_t,(\mtc{L}|_{\mtc{D}})^m_t)$ where $m\gg 0$. As above, we have line bundles $\widetilde{\lambda}_i=\widetilde{\lambda}_i(\mtc{D},(\mtc{L}|_{\mtc{D}})^m)$ on $S$ such that
$$\det f_{!}(\mtc{L}|_{\mtc{D}}^k)={\widetilde{\lambda}}_{n}^{\binom{k}{n}}\otimes{\widetilde{\lambda}}_{n-1}^{\binom{k}{n-1}}\otimes\cdots\otimes\widetilde{\lambda}_0^{\binom{k}{0}}.$$

\begin{defn}[cf. {\cite[Definition 2.2]{GMG21}, \cite[Definition 2.18]{ADL24}}]
	The \textup{CM} $\mb{Q}$-line bundle of a $\mb{Q}$-Gorenstein smoothable family $f:(\mtc{X},c\mtc{D})\rightarrow S$ over a reduced base $S$ is defined as
	$$\Lambda_{CM,f,c\mtc{D}}\coloneqq k^{-n}\left(\lambda_{n+1}^{\otimes\left(n(n+1)+\frac{2a_1-c\widetilde a_0}{a_0}\right)}\otimes\lambda_n^{\otimes(-2(n+1))}\otimes\widetilde \lambda_n^{\otimes c(n+1)}\right)$$
	where $\mtc{L}=-k(K_{\mtc{X}/S}+c\mtc{D})$ is a Cartier divisor on $\mtc{X}$ ample over $S$, and $k\in\mb{Z}$ is some positive integer.
\end{defn}

\begin{theorem}[{\cite[Theorem 2.6]{GMG21}, cf. \cite{PT08}}]
	\label{theorem:weight-DF}
	Let $(f:(\mathcal X, \mathcal D)\rightarrow \mb{A}^1, \mathcal L)$ be a test configuration of an $n$-dimensional $(-k(K_X+cD))$-polarized pair $(X, D)$. Then
	$$w(\Lambda_{\CM, f,c\mathcal D})=(n+1)!\Fut(\mathcal X, c\mathcal D, -k(K_{\mtc{X}/\mb{A}^1}+c\mtc{D})),$$
	where $w(\Lambda_{\CM, f,c\mathcal D})$ is the total weight of $\Lambda_{\CM, f,c\mathcal D}$ under the $\mathbb C^*$-action of the test configuration.
\end{theorem}

\begin{theorem}[{{K-moduli Theorem}, cf. \cite{ABHLX20,BHLLX21,BLX22,BX19,CP21,Jia20,LWX21,LXZ22,Xu20,XZ20, OSS16, SSY, XZ21}}]\label{17} 
	Let $r\in\mb{Q}_{\geq 1}$ and $c\in(0,1/r)$ be a rational number, and $\chi$ be the Hilbert polynomial of an anti-canonically polarized smooth Fano variety. Consider the moduli pseudo-functor sending a reduced base $S$ to
	
	\[
	{\mtc{M}}^K_c(S)\coloneqq\left\{(\mtc{X},\mtc{D})/S\left| \begin{array}{l}(\mtc{X},c\mtc{D})/S\textrm{ is a $\mb{Q}$-Gorenstein smoothable log Fano family,}\\ \mtc{D}\sim_{S,\mathbb{Q}}-rK_{\mtc{X}/S},~\textrm{each fiber $(\mtc{X}_s,c\mtc{D}_s)$ is K-semistable,}\\ \textrm{and $\chi(\mtc{X}_s,\mtc{O}_{\mtc{X}_s}(-mK_{\mtc{X}_s}))=\chi(m)$ for $m$ sufficiently divisible.}\end{array}\right.\right\}.
	\]
	Then there is a reduced Artin stack ${\mtc{M}}^K(c)$ of finite type over $\mb{C}$ representing this moduli pseudo-functor. The $\mb{C}$-points of ${\mtc{M}}^K(c)$ parameterize K-semistable $\mb{Q}$-Gorenstein smoothable log Fano pairs $(X,cD)$ with Hilbert polynomial $\chi(X,\mtc{O}_X(-mK_X))=\chi(m)$ for sufficiently divisible $m\gg 0$ and $D\sim_{\mb{Q}}-rK_X$. Moreover, the stack ${\mtc{M}}^K(c)$ admits a good moduli space $\ove{M}^K(c)$, which is a reduced projective scheme of finite type over $\mb{C}$, whose $\mb{C}$-points parameterize K-polystable log Fano pairs. Furthermore, the CM $\mb{Q}$-line bundle on $\mtc{M}^K(c)$ descends to an ample $\mb{Q}$-line bundle on $\ove{M}^K(c)$.
\end{theorem}

\begin{lemma}
\label{lem:Fut-isotrivial-Kss-degeneration}
    Let $(X,cD)$ and $(X',cD')$ be two K-semistable log Fano pairs such that their images in $\ove{M}^K(c)$ are the same and there is an isotrivial one-parameter family $(\mathcal X, \mathcal D)$ over $\mathbb A^1$ whose general fiber is $(X, D)$ and whose central fiber is $(X', D')$. Then $\Fut(\mtc{X},c\mtc{D})=0$.
\end{lemma}
\begin{proof}
    Let $f:(\ove{\mtc X}, \ove{\mtc D})\rightarrow\mathbb P^1$ be the trivial completion of $(\mtc X, \mtc D)$ at infinity. By hypothesis, the map $(\mathcal X, \mathcal D)\rightarrow \ove{M}^K(c)$ induced by $\Lambda_{CM, f, c\mathcal D}$ is constant, so the degree of $\Lambda_{CM, f, c\mathcal D}$ is $0$. Thus Theorem \ref{theorem:weight-DF} implies $\Fut(\mtc{X},c\mtc{D})=0$.
\end{proof}
An analogue argument for GIT quotients gives the following:
\begin{lemma}
\label{lem:HM-isotrivial-ss-degeneration}
    Let $p, p'$ be two semistable points of an $L$-polarized projective $G$-scheme $X$ where $G$ is a reductive group and $L$ is a $G$-linearized ample line bundle on $X$ and assume there is a one-parameter subgroup $\sigma$ of $G$ such that $\lim_{t\rightarrow 0}\sigma(t)\cdot p=p'$. Then the Hilbert-Mumford weight $\mu^L(p, \sigma)=0$.
\end{lemma}

The following theorem illustrates how the moduli spaces $\ove{M}^K(c)$ depend on the coefficient $c$.

\begin{theorem}[{cf. \cite[Theorem 1.2]{ADL24}}] There are rational numbers $$0=c_0<c_1<c_2<\cdots<c_n=1,$$ such that for every $0\leq j<n$, the K-moduli stacks ${\mtc{M}}^K(c)$ are independent of the choice of $c\in(c_j,c_{j+1})$. Moreover, for every $0\leq j<n$ and $0<\epsilon\ll1$, one has open immersions $${\mtc{M}}^K(c_j-\epsilon)\hookrightarrow {\mtc{M}}^K(c_j)\hookleftarrow {\mtc{M}}^K(c_j+\epsilon),$$ which descend to projective birational morphisms $$\ove{M}^K(c_j-\epsilon)\rightarrow \ove{M}^K(c_j)\leftarrow \ove{M}^K(c_j+\epsilon).$$
	
\end{theorem}

The values $c_i$ are called \textup{walls} and the intervals $(c_i, c_{i+1})$ are called \textup{chambers}. Suppose a pair $(X, cD)$ is K-polystable if and only if $0<c<c_i$ and let  $(X_0, c_iD_0)$ be its image under the map $\ove{M}^K(c_i-\epsilon)\rightarrow \ove{M}^K(c_i)$ at $c=c_i$. We say that $(X_0, D_0)$ is \textup{the polystable replacement of $(X, D)$ at $c=c_i$}.

\begin{theorem}\label{12} \textup{(cf. \cite[Theorem 2.22]{ADL24})}
	Let $f:(\mtc{X},c\mtc{D})\rightarrow S$ be a $\mb{Q}$-Gorenstein smoothable family of log Fano pairs over a normal projective base $S$. Let $G$ be a reductive group acting on $\mtc{X}$ and $S$ such that $\mtc{D}$ is $G$-invariant and $f$ is $G$-equivariant. Moreover, assume the following conditions are satisfied:
	\begin{enumerate}[(i)]
		\item for any $s\in S$, if $\Aut(\mtc{X}_s,\mtc{D}_s)$ is finite, then the stabilizer $G_s$ is also finite;
		\item if for $s,s'\in S$, we have that $(\mtc{X}_s,\mtc{D}_s)\simeq (\mtc{X}_{s'},\mtc{D}_{s'})$, then $s'\in G\cdot s$;
		\item $\Lambda_{CM,f,c\mtc{D}}$ is an ample $\mb{Q}$-line bundle on $S$.
	\end{enumerate}
	Then $s\in S$ is GIT-(poly/semi)stable with respect to the $G$-linearized $\mb{Q}$-line bundle $\Lambda_{CM,f,c\mtc{D}}$ if $(\mtc{X}_s,c\mtc{D}_s)$ is a K-(poly/semi)stable pair.
\end{theorem}

\subsection{Variations of GIT quotients}

For the general theory on Variations of GIT quotients, we refer the reader to \cite{DH98,Tha96} and the beautiful survey \cite{Laz13}. Here we state only the following structure theorem.

\begin{lemma}\textup{(cf. \cite[Lemma 3.10]{Laz13})}\label{8}
	Let $(X,L)$ be a polarized projective variety, and $G$ a reductive group acting on $(X,L)$. Let $L_0$ be a $G$-linearized line bundle on $X$. For any rational number $0<\epsilon\ll1$, we define $L_{\pm \epsilon}\coloneqq L\otimes L_0^{\otimes (\pm \epsilon)}$. Let $X^{ss}(0)$ and $X^{ss}(\pm)$ denote the semistable loci with respect to $L$ and $L_{\pm \epsilon}$, respectively. Then
	\begin{enumerate}[(i)]
		\item there are open immersions $X^{ss}(\pm)\subseteq X^{ss}(0)$, and
		\item for any $x\in X^{ss}(0)\setminus X^{ss}(\pm)$, there is a 1-parameter subgroup $\sigma$ of $G$ such that $$\mu^{L}(x,\sigma)=0,\quad \textup{and}\quad \mu^{L_{+\epsilon}}(x,\sigma)\mu^{L_{-\epsilon}}(x,\sigma)<0 \text{ for }0<\epsilon\ll 1,$$
		where $\mu$ is the Hilbert-Mumford weight. Note, the second condition implies that $\mu^{L_t}(x, \sigma)$ changes sign when crossing the `wall' at $t=0$.
	\end{enumerate}
\end{lemma}

\subsection{Overview of VGIT of hypersurfaces and complete intersections in $\mb{P}^3$ and $\mb{P}^4$}\label{subsection:VGIT}

In \cite{GMG18,GMG-code, GMGZ18}, compactifications of log Fano pairs $(X,cD)$ where $X$ is a Fano hypersurface of degree $d$ in $\mathbb P^n$ and $D$ is a hyperplane section were introduced, using VGIT of pairs $(X,H)$ where $H$ is a hyperplane so that $D=X|_H$. This construction was extended in \cite{Pap22} to the case where $X$ is a complete intersection of $k$ hypersurfaces of the same degree $d$ and $D=X\cap H$ is still a hyperplane section (with the previous case being $k=1$). Note that in both cases $(n+1-kd)D\sim -K_X$ and in the case of Fano index $1$ (i.e. $kd=n$), $D\sim -K_X$. Furthermore, \cite{GMG19} classified the points for each compactification for the case where $X\subset \mathbb P^3$ is a cubic surface (a del Pezzo surface of degree $3$) and \cite{Pap22} for the case where $X\subset \mathbb P^4$ is a complete intersection of two quadrics (a del Pezzo surface of degree $4$). Let us briefly recall the constructions and main results for these two specific cases, as we will need them later.

The scheme $S_1\coloneqq\mb{P}H^0(\mb{P}^3,\mtc{O}_{\mb{P}^3}(3))\simeq \mb{P}^{19}$ parameterizes embedded cubic surfaces in $\mb{P}^3$ (including all del Pezzo surfaces of degree $3$ with canonical singularities \cite[III.3]{Kol96}), and the scheme $S_2\coloneqq\mb{P}H^0(\mb{P}^3,\mtc{O}_{\mb{P}^3}(1))\simeq \mb{P}^{3}$ parameterizes embedded hyperplanes in $\mb{P}^3$. Similarly, let $V = H^0(\mb{P}^4,\mtc{O}_{\mb{P}^4}(2))$ be the vector space of global sections of  $\mtc{O}_{\mb{P}^4}(2)$. The scheme $S_1'\coloneqq\operatorname{Gr}(2,V) \simeq \operatorname{Gr}(2,15)$  parameterizes embedded complete intersections of two quadrics in $\mathbb{P}^4$ (including all del Pezzo surfaces of degree $4$ with canonical singularities \cite[III.3]{Kol96}), with Pl{\"u}cker embedding $\operatorname{Gr}(2,V) \rightarrow\bigwedge^2 \mb{P}V$. The scheme  $S_2'\coloneqq\mb{P}H^0(\mb{P}^4,\mtc{O}_{\mb{P}^4}(1))\simeq \mb{P}^{4}$ parameterizes embedded hyperplanes in $\mb{P}^4$.

The groups $\PGL(4)$ and $\PGL(5)$ act naturally on the product spaces $S_1\times S_2$ and $S_1'\times S_2'$, respectively, parameterizing pairs $(X,H)$ formed by a (possibly non-reduced or reducible) cubic surface and a hyperplane, or a (possibly non-reduced or reducible) complete intersection of two quadrics and a hyperplane, respectively. Let $\pi_1$, $\pi_1'$, $\pi_2$ and $\pi_2'$ be the projections to $S_1$, $S_1'$, $S_2$, $S_2'$, respectively. For any rational number $t>0$, we set $$\mtc{L}_t\coloneqq\mtc{O}(1,t)\coloneqq\pi_1^{*}\mtc{O}_{S_1}(a)\otimes\pi_2^{*}\mtc{O}_{S_2}(b)$$ 
and
$$\mtc{L}_t'\coloneqq \mtc{O}(1,t)\coloneqq (\pi_1')^{*}\mtc{O}_{S_1'}(a)\otimes(\pi_2')^{*}\mtc{O}_{S_2'}(b),$$
where $t=\frac{b}{a}$ and we can define GIT$_t$-(semi/poly)stability in the obvious way. Note that $\mathcal L_t$ and $\mathcal L_t'$ are only defined up to scaling. However, since we only use them to define a GIT quotient, the specific choice of $(a,b)$ will be irrelevant (we assume they are both positive and sufficiently large). For each rational $t>0$, $\mtc{L}_t$ and  $\mtc{L}_t'$ are ample. For each value $t>0$, define the GIT moduli stacks
$${\mtc{M}}_3^{\GIT}(t)\coloneqq\left[{(S_1\times S_2)^{ss}}/_{\mtc{L}_t}{\PGL(4)}\right], \qquad {\mtc{M}}_4^{\GIT}(t)\coloneqq\left[{(S_1'\times S_2')^{ss}}/_{\mtc{L}_t'}{\PGL(5)}\right]$$
and let their moduli spaces be $\ove{M}_3^{\GIT}(t)$ and ${\ove{M}}_4^{\GIT}(t)$, respectively.

By \cite[Theorem 1.1]{GMG18} and \cite[Lemma 3.20]{Pap22}, there are no GIT$_t$-semistable pairs $(X,H)$ for $t>1$. Furthermore, by \cite[Theorem 1.3]{GMG18} and \cite[Theorem 3.23]{Pap22}, if $(X,H)$ is semistable, then $D=X\cap H\sim -K_X$ is a hypersurface in $X$. In particular, we note $\ove{M}_d^{\GIT}(t)$ gives a GIT compactification of the moduli of log smooth pairs $(X,D)$. For $t\in (0,1)$, there is a wall-chamber decomposition \cite{Tha96, DH98}. A \textup{chamber} is defined as the largest subinterval in which $\ove{M}_d^{\GIT}(t)$ does not vary, while a \textup{wall} is the intersection of the closure of two chambers. If $t$ is a wall, there are contractions:
$$\ove{M}_d^{\GIT}(t-\epsilon)\longrightarrow\ove{M}_d^{\GIT}(t)\longleftarrow \ove{M}_d^{\GIT}(t+\epsilon)$$
for $0<\epsilon \ll 1$ and $d=3$, $4$.

We have a full description of the walls and chambers:
\begin{theorem}[{\cite[Lemma 1]{GMG19}, \cite[\S 6.3, Theorem 6.20]{Pap22}}]
	The walls for $\ove{M}_3^{\GIT}(t)$ are:
	$$t_0=0,\quad t_1=\frac{1}{5},\quad t_2=\frac{1}{3},\quad t_3=\frac{3}{7},\quad t_4=\frac{5}{9},\quad t_5=\frac{9}{13},\quad t_6=1.$$
	The walls for $\ove{M}_4^{\GIT}(t)$ are: 
	$$t_0=0,\quad t_1=\frac{1}{6},\quad t_2=\frac{2}{7},\quad t_3=\frac{3}{8},\quad t_4=\frac{6}{11},\quad t_5=\frac{2}{3},\quad t_6=1.$$
\end{theorem}
Moreover, in \cite[Theorems 2 and 3]{GMG19} and \cite[Theorem 6.20, Theorem 6.21]{Pap22} the authors gave a precise description of all pairs $(X,D)$ represented in each $\ove{M}_3^{\GIT}(t)$ and $\ove{M}_4^{\GIT}(t)$, respectively, for each $t\in (0,1)$. The full statements are rather long, so we will omit them here.

Let $U_3\subseteq S_1\times S_2$, $U_4\subseteq S_1'\times S_2'$ be the open subsets representing pairs $(X,H)$ such that $H$ is not contained in $X$ and that $X\cap H\neq X\cap H'$ for any hyperplane $H'\neq H$. Furthermore, $U_4$ also represents pairs $(X,H)$, such that $X$ is not isomorphic to $\{x_4x_{3}=x_4x_{2}=0\}$ and $\{x_4^2=x_4x_{3}=0\}$, i.e., $X$ is a complete intersection. It follows that the complement of $U_d$ is of codimension $\geq 2$ for $d= 3$, $4$, and it only contains pairs $(X,H)$ which are GIT-unstable for any $0<t<1$ (c.f. \cite[Lemmas 3.3, 3.5, 3.6]{GMG21} and \cite[Lemmas 3.3, 3.5, 3.6, 7.2, 7.3, 7.4, 7.8]{Pap22}). Let $f:(\mtc{X},\mtc{D}=\mtc{X}\cap \mtc{H})\rightarrow U_d$ be the universal family over $U_d$. For any $0<c<1$, let $\Lambda_{\CM, c}\coloneqq\Lambda_{\CM,f,c\mtc{D}}$ be the CM $\mb{Q}$-line bundle of this family.

\begin{prop}\textup{(cf. \cite[Theorem 3.8]{GMG21})}\label{7}
	For $d=3$ and any rational number $0<c<1$, we have $$\Lambda_{\CM,c}\simeq \mtc{O}(8+c,9c)$$ as $\mb{Q}$-line bundles. In particular, note that $\Lambda_{\CM,c}$ is ample for all $c\in (0,1)\cap \mathbb Q$ and $\Lambda_{\CM,c}$ is proportional to $\mtc{L}_{t(c)}$ 
	where $t(c)=\frac{9c}{8+c}$.
\end{prop}

\begin{prop}\label{18}\textup{(cf. \cite[Lemma 7.5, Proposition 7.9]{Pap22})}
	For $d=4$ and any rational number $0<c<1$, we have that
	$$\Lambda_{\CM,c}\simeq \mtc{O}(2(5+c),12c)$$ as $\mb{Q}$-line bundles. In particular, we have that $\Lambda_{\CM,c}$ is ample for all $c\in (0,1)\cap \mathbb Q$ and $\Lambda_{\CM,c}$ is proportional to $\mtc{L}_{t(c)}'$ 
	where $t(c)=\frac{6c}{5+c}$.
\end{prop}

For $d=4$ the author also shows that there is an isomorphism of K-moduli and GIT moduli stacks up until and including the second GIT chamber \cite[Theorem 8.2]{Pap22}.

\section{Isomorphisms of VGIT and K-moduli for degrees $3$ and $4$}
\label{sec:deg-3-4}
In this section, we will first give a bound of singularities of the surfaces which will appear in the K-moduli spaces. This will be widely used later. Then we prove Theorem \ref{31}.

\begin{theorem}\label{2}
	Let $0<c<1$ be a positive number. Let $(X_0,cD_0)$ be a K-semistable log del Pezzo pair that admits a $\mb{Q}$-Gorenstein smoothing to $(X,cD)$, where $X$ is a smooth del Pezzo surface of degree $l\geq 3$ and $D\in |-K_X|$ is smooth. Then $X_0$ is Gorenstein. 
\end{theorem}

\begin{proof}
	Let $x\in X_0$ be a point, and $n\coloneqq\Ind(x,K_{X_0})\geq1$ be the Gorenstein index at $x$. If $x\notin D_0$, then the assumption $D_0+K_{X_0}\sim 0$, in both cases, implies that $n\leq 1$ (since then locally around $x$, $K_{X_0}$ is trivial and thus Gorenstein). Now let us assume that $x\in D_0$. Since $x$ is a klt surface singularity (as $(X_0,cD_0)$ is K-semistable), then $x$ is a quotient singularity \cite[Corollary 5.21]{KM98}. As $x\in X_0$ is $\mathbb Q$-Gorenstein smoothable and a quotient singularity, then $x$ is an index $n$ point cyclic quotient singularity of type $\frac{1}{dn^2}(1,dna-1)$ \cite[Proposition 3.10]{KSB88}. It follows from Theorem \ref{16} that for $X$ a del Pezzo surface of degree $l$ $$l(1-c)^2=(-K_{X_0}-cD_0)^2\leq \frac{9}{4}\widehat{\vol}(x,X_0,cD_0).$$
	Let $(\widetilde{x},\widetilde{X_0},\widetilde{D_0})\rightarrow(x,X_0,D_0)$ be a smooth covering of degree $dn^2$ over $x$. By \cite[Theorem 2.7 (3)]{LX19}, one has that 
	\begin{equation}\label{38}
		dn^2\cdot\frac{4}{9}\cdot l(1-c)^2\leq dn^2\widehat{\vol}(x,X_0,cD_0)=\widehat{\vol}(\widetilde{x},\widetilde{X_0},c\widetilde{D_0})\leq(2-c\ord_{\widetilde{x}}\widetilde{D_0})^2<4,
	\end{equation} 
	where the second inequality follows from computing the discrepancy on the smooth blowup of $\widetilde x$.
	
	Since $X_0$ is a del Pezzo surface of degree $l\geqslant 3$, if $\ord_{\widetilde{x}}\widetilde{D_0}\geq 2$, then we have
	$$dn^2\cdot \frac{4}{3}(1-c)^2\leq dn^2\cdot \frac{4l}{9}(1-c)^2\leq 4(1-c)^2,$$
	and thus $3\geq dn^2\geq n^2$. This forces $n=1$ and $X_0$ is Gorenstein.
	
	Thus, assume that $\ord_{\widetilde{x}}\widetilde{D_0}=1$. Let $u,v$ be local coordinates near $\widetilde{x}\in\widetilde{X_0}$ such that the cyclic group action on $\widetilde X_0$ defining the quotient singularity is scaling on each coordinate, and let $u^iv^j$ be a monomial appearing in the local equation of $\widetilde{D_0}$ such that $i+j=\ord_{\widetilde{x}}\widetilde{D_0}=1$. The relation $D_0+K_{X_0}\sim 0$ 
	implies that $$i+(dna-1)j\equiv dna \mod dn^2,$$
	as around $x$, the total weights of $u$, $v$ and $K_{X_0}$ are 1, $dna-1$ and $-dna$, respectively and $D_0\sim-K_{X_0}$ Since either $(i,j)=(0,1)$ or $(i,j)=(1,0)$, then $dn$ has to divide $1$, which implies $d=n=1$, and hence $X_0$ is Gorenstein.
	
\end{proof}

\begin{corollary}
	\label{cor:Gorenstein-are-in-GIT}
	Let $0\leq c\leq 1$ be a positive number. 
	\begin{enumerate}
		\item Let $(X_0,cD_0)$ be a K-semistable log del Pezzo pair that admits a $\mb{Q}$-Gorenstein smoothing to $(X,cD)$, where $X$ is a cubic surface and $D\in |-K_X|$. Then $X_0$ can be embedded into $\mb{P}^3$ by $|-K_{X_0}|$ as a cubic surface.
		\item Let $(X_0,cD_0)$ be a K-semistable log del Pezzo pair that admits a $\mb{Q}$-Gorenstein smoothing to $(X,cD)$, where $X$ is a complete intersection of two quadrics in $\mathbb{P}^4$ and $D\in |-K_X|$. Then $X_0$ can be embedded into $\mb{P}^4$ by $|-K_{X_0}|$ as a complete intersection of two quadrics.
	\end{enumerate}
\end{corollary}

\begin{proof}
	This Corollary is an immediate consequence, in both cases, of Theorem \ref{2} and \cite[p. 117 and Corollary 1.7]{Fuj90}.
\end{proof}


Recall that $U_d$ is the open subset parameterizing pairs $(X,H)$ such that $H$ is a hypersurface not contained in $X$ and that $X\cap H\neq X\cap H'$ for any hyperplane $H'\neq H$. Furthermore, $U_4$ also represents pairs $(X,H)$, such that $X$ is not isomorphic to $\{x_4x_{3}=x_4x_{2}=0\}$ and $\{x_4^2=x_4x_{3}=0\}$, i.e., $X$ is a complete intersection. 
Let $U^{K}_{d}(c)$ be the open subset of $U_d$ which consists of pairs $(X,H)$ such that $(X,cD)$ is K-semistable, where $D=X\cap H$, and $U_d^{\GIT}(c)\subseteq U_d$ the GIT$_{t(c)}$-semistable locus, where $t(c)=\frac{9c}{8+c}$ for $d =3$ and $t(c)=\frac{6c}{5+c}$ for $d=4$.

\begin{prop}\label{6}
	Let $d$ be $3$ or $4$. For any $0<c<1$, the K-moduli stack $\mtc{M}^K_d(c)$ is isomorphic to the quotient stack $[U^K_d(c)/\PGL(d+1)]$, via which the open immersions $$U^K_d(c-\epsilon)\hookrightarrow U^K_d(c)\hookleftarrow U^K_d(c+\epsilon)$$ descend to wall-crossing morphisms $$\mtc{M}_d^K(c-\epsilon)\hookrightarrow \mtc{M}_d^K(c)\hookleftarrow \mtc{M}_d^K(c+\epsilon).$$
\end{prop}

\begin{proof}
We follow the argument of \cite[Proposition 5.9]{ADL21}. By definition, every pair
parametrized by $U^K_d(c)$ is K-semistable, so the universal property of the K-moduli
stack yields a natural $\PGL(d+1)$-equivariant morphism $U^K_d(c)\to\cM^K_d(c)$. Passing
to the quotient, we obtain
\[
  \psi\colon\ \big[U^K_d(c)/\PGL(d+1)\big]\ \longrightarrow\ \cM^K_d(c).
\]

To construct an inverse, consider the universal family
$\pi_d\colon(\mtc{X}_d,c\cD_d)\to\cM^K_d(c)$. By Corollary
\ref{cor:Gorenstein-are-in-GIT}, every pair parametrized by $\cM^K_d(c)$ is of the form
$(X,cD)$ with $D=X\cap H$ for a hyperplane $H$, where $X$ is a cubic surface in $\bP^3$
if $d=3$, or a complete intersection of two quadrics in $\bP^4$ if $d=4$. In particular
each fiber is a Gorenstein del Pezzo surface of degree $d\ge 3$, so
$\cL_d\coloneqq -K_{\mtc{X}_d/\cM^K_d(c)}$ is a relatively very ample line bundle. Since
$h^0(X,-K_X)=d+1$ and $h^1(X,-K_X)=0$ for every fiber, cohomology and base change show
that $(\pi_d)_*\cL_d$ is locally free of rank $d+1$, and $\cL_d$ induces a closed
embedding
\[
  (\mtc{X}_d,\cD_d)\ \hookrightarrow\ \bP\big((\pi_d)_*\cL_d\big)
\]
over $\cM^K_d(c)$, realizing each fiber as a degree-$d$ del Pezzo surface together with
its anticanonical section.

The projectivized frame bundle of $(\pi_d)_*\cL_d$ is a $\PGL(d+1)$-torsor
$\mtc{P}_d\to\cM^K_d(c)$, and consequently $\cM^K_d(c)\simeq\big[\mtc{P}_d/\PGL(d+1)\big]$.
A point of $\mtc{P}_d$ over $(X,cD)$ is a projective frame on $\bP\big((\pi_d)_*\cL_d\big)$,
i.e.\ an identification of the ambient $\bP^d$ with the standard one; applying it to
$(X,D)$ produces an embedded pair in the standard $\bP^d$, that is, a point of
$U^K_d(c)$ (it lies in the K-semistable locus since every fiber of $\pi_d$ is
K-semistable). This defines a $\PGL(d+1)$-equivariant morphism $\mtc{P}_d\to U^K_d(c)$,
which descends to
\[
  \cM^K_d(c)\ \longrightarrow\ \big[U^K_d(c)/\PGL(d+1)\big].
\]
By construction this morphism is inverse to $\psi$, and hence $\psi$ is an isomorphism.
\end{proof}




We denote by $$0=c_0<c_1<\cdots<c_l=1$$ a sequence of numbers such that for each $c_i$, either $c=c_i$ is a wall for K-moduli stacks $\mtc{M}^K_d(c)$, or $t=t(c_i)$ is a wall for the VGIT moduli stacks $\mtc{M}^{\GIT}_d(t(c))$. The following Proposition summarizes the main results of \cite{GMG21} and \cite{Pap22}.

\begin{prop}[{cf. \cite[Proposition 4.7]{GMG21}, \cite[Theorem 8.2]{Pap22}}]\label{5}
	Let $d\in \{3,4\}$. For $c\in (0,c_1)$, the two open subsets $U^K_d(c)$ and $U^{\GIT}_d(c)$ of $U$ are equal.
\end{prop}

The next two propositions involve an induction process, leading to the results that $U^K(c)=U^{\GIT}(c)$ for any $0<c<1$, both for $d=3$ and $d=4$. This is an approach first used in \cite{ADL21} for degree $9$.

\begin{prop}\label{11}
	Let $d\in \{3,4\}$. Suppose that, for any $c\in(0,c_i)$, we have $U^K_d(c)=U^{\GIT}_d(c)$. Then $U^K_d(c_i)=U^{\GIT}_d(c_i)$.
\end{prop}

\begin{proof}
	We first show that $U^K_d(c_i)\subseteq U^{\GIT}_d(c_i)$. Let $(X,D)\in U^K_d(c_i)$ be any K-semistable point. Let $(X_0,D_0)\in \ove{M}^{K}_d(c_i)$ be the K-polystable pair, and $\sigma$ a 1-parameter subgroup such that
    $$\lim_{t\to 0}\sigma(t)\cdot [(X,D)]=[(X_0,D_0)]$$
    in $\mtc{M}^K_d(c_i)$. Note that since $-K_X$ is very ample, any special test configuration is naturally embedded into $\mathbb P^d\times \mathbb A^1$, so that the test configuration is induced by a 1-parameter subgroup. Now, by the surjectivity of the wall-crossing morphism $\ove{M}^K_d(c_i-\epsilon)\rightarrow \ove{M}^K_d(c_i)$, we can choose a K-semistable pair $(X',D')\in U^K_d(c_i-\epsilon)$ and a 1-parameter subgroup $\sigma'$ such that $$\lim_{t\to 0}\sigma'(t)\cdot [(X',D')]=[(X_0,D_0)].$$ Let $(\mtc{X},c_i\mtc{D})$ and $(\mtc{X}',c_i\mtc{D}')$ be the two test configurations corresponding to $\sigma$ and $\sigma'$ respectively. Since $(X_0,c_iD_0)$ is K-polystable and $(X',c_iD')$ is K-semistable, then the generalized Futaki invariant $\Fut(\mtc{X}',c_i\mtc{D}')=0$ by Lemma \ref{lem:Fut-isotrivial-Kss-degeneration}. As $\Fut(\mtc{X}',c_i\mtc{D}')$ is proportional to the GIT weight of the CM $\mb{Q}$-line bundle $\Lambda_{U,c_i}$ (Theorem \ref{theorem:weight-DF}), then we have $\mu^{\mtc{L}_{t(c_i)}}([(X',D')],\sigma')=0$ by Proposition \ref{7}. By our hypothesis, the pair $(X',D')$ is contained in $U^K_d(c_i-\epsilon)=U^{\GIT}_d(c_i-\epsilon)\subseteq U_d^{\GIT}(c_i)$. Now it follows from Lemma \ref{3} that $[(X_0,D_0)]$ is GIT$_{t(c_i)}$-semistable. By openness of GIT semistability, we conclude that $[(X,D)]\in U^{\GIT}_d(c_i)$, and hence $U_d^K(c_i)\subseteq U_d^{\GIT}(c_i)$.
	
	Conversely, for any $(X,D)\in U^{\GIT}_d(c_i)$, similarly we can find GIT$_{t(c_i)}$-polystable pair $(X_0,D_0)\in U_d^{\GIT}(c_i)$, GIT$_{t(c_i)-\epsilon}$-semistable pair $(X',D')\in U_d^{\GIT}(c_i-\epsilon)$, and two 1-parameter subgroups $\sigma$ and $\sigma'$ such that $$\lim_{t\to 0}\sigma(t)\cdot [(X,D)]=[(X_0,D_0)],\quad \lim_{t\to 0}\sigma'(t)\cdot [(X',D')]=[(X_0,D_0)],$$ and that $$\mu^{\mtc{L}_{t(c_i)}}([(X,D)],\sigma)=\mu^{\mtc{L}_{t(c_i)}}([(X',D')],\sigma')=0,$$ 
	by Lemma \ref{lem:HM-isotrivial-ss-degeneration}.
    
	Let $(\mtc{X}',c_i\mtc{D}';\mtc{L})$ be the test configuration of $(X',c_iD')$ corresponding to $\sigma'$. Then, $\Fut(\mtc{X}',c_i\mtc{D}';\mtc{L})=0$ by Proposition \ref{7} and Theorem \ref{theorem:weight-DF}. Since $(X',(c_i-\epsilon)D')$ is K-semistable by our hypothesis, then $\mtc{X}'$ is normal (\cite[Section 8.2]{LX14}, the first paragraph). Now it follows from Lemma \ref{4} that $(X_0,c_iD_0)$ is K-semistable, and hence $(X,D)$ is K-semistable by openness of K-semistability.
\end{proof}

\begin{prop}\label{10}
	Let $d\in \{3,4\}$. Suppose that, for any $c\in(0,c_i]$ we have $U^K_d(c)=U^{\GIT}_d(c)$. Then $U^K_d(c)=U^{\GIT}_d(c)$ for any $c\in(c_i,c_{i+1})$.
\end{prop}

\begin{proof}
	Notice that it suffices to prove the equality for a fixed $c=c_i+\epsilon$, $0<\epsilon \ll c_{i+1}-c_i$. We first show that $U_d^K(c_i+\epsilon)\subseteq U_d^{\GIT}(c_i+\epsilon)$. Suppose otherwise, choose a pair $(X,D)\in U_d^K(c_i+\epsilon)\setminus U^{\GIT}(c_i+\epsilon)$. By our assumption and Proposition \ref{6}, $(X,D)$ is GIT$_{c_i}$-semistable, and hence using Lemma \ref{8}, we can find a 1-parameter subgroup $\sigma$ of $\PGL(d+1)$ such that $$\mu^{\mtc{L}_{t(c_i)}}([(X,D)],\sigma)=0,\quad \textup{and}\quad \mu^{\mtc{L}_{t(c_i+\epsilon)}}([(X,D)],\sigma)<0.$$ Then the limit $\lim_{t\to0}\sigma(t)\cdot [(X,D)]$ is GIT$_{t(c_i)}$-semistable by Lemma \ref{3}, and thus represented by some $(X_0,D_0)\in U_d^{\GIT}(c_i)$. Let $(\mtc{X},\mtc{D};\mtc{L})$ be the test configuration of $(X,D,-K_X)$ corresponding to $\sigma$. It follows from Proposition \ref{7} and Theorem \ref{theorem:weight-DF} that $\Fut(\mtc{X},(c_i+\epsilon)\mtc{D};\mtc{L})<0$, and hence $(X,(c_i+\epsilon)D)$ is K-unstable, which is a contradiction. Thus $U_d^K(c_i+\epsilon)\subseteq U_d^{\GIT}(c_i+\epsilon)$. In fact, note moreover that if $(X,(c_i+\epsilon)D)$ is K-polystable, then it is GIT$_{t(c_i+\epsilon)}$-polystable (Lemma \ref{9}).
	
	Now consider the commutative diagram $$\xymatrix@C=0.5em{
		U^K_d(c_i+\epsilon) \ar@{^(->}[d]^\varphi \ar[r] & \left[U^K_d(c_i+\epsilon)/\PGL_{d+1} \right] \ar@{^(->}[d]^{\phi} \ar[r] & U^K_d(c_i+\epsilon)\git\PGL_{d+1}=\ove{M}_d^{K}(c_i+\epsilon) \ar[d]^\psi \\
		U_d^{\GIT}(c_i+\epsilon) \ar[r] & \left[U_d^{\GIT}(c_i+\epsilon)/\PGL_{d+1} \right] \ar[r] & U_d^{\GIT}(c_i+\epsilon)\git\PGL_{d+1}=\ove{M}_d^{\GIT}(t(c_i+\epsilon)).   }$$
    Let us clarify how these maps are defined. The map $\varphi\colon U_d^K(c_i+\epsilon)\hookrightarrow U_d^{\GIT}(c_i+\epsilon)$ is the inclusion proven above. The stack construction then induces map $\phi$. In particular $\varphi$ and $\phi$ are natural inclusions. For the existence of the map $\psi$, notice that Lemma \ref{9} below shows that if $(X,(c_i+\epsilon)D)$ is K-polystable, then it is GIT$_{t(c_i+\epsilon)}$-polystable, so the map $\phi$ induces the map $\psi$ at the level of moduli spaces.
	
	As $\phi$ is descended from the open immersion $\varphi$, then $\phi$ is a monomorphism of stacks, and thus it is separated \cite[\href{https://stacks.math.columbia.edu/tag/06MY}{Tag 06MY}]{stacks-project} and representable by \cite[\href{https://stacks.math.columbia.edu/tag/04YY}{Tag 04YY}]{stacks-project} (as it is an injection and the stabilisers are the same). By Lemma \ref{9} and Proposition \ref{6}, we see that $\phi$ sends closed points to closed points (which correspond to closed points in the moduli spaces), which implies that $\psi$ is quasi-finite, as it is injective (because $\phi$ is injective).
	As $\ove{M}_d^{\GIT}(t(c_i+\epsilon))$ and $\ove{M}_d^{K}(c_i+\epsilon)$ are separated, $\psi$ is separated. As $\ove{M}_d^{\GIT}(t(c_i+\epsilon))$ and $\ove{M}_d^{K}(c_i+\epsilon)$ are proper and $\psi$ is separated, then $\psi$ is proper, which together with being quasi-finite means $\psi$ is finite. Consequently $\phi$ is finite \cite[Proposition 6.4]{Alp13}. Note the moduli construction implies that the left and right square above are fiber products. Since finite maps are preserved by base change, 
	the map $\varphi$ is finite (and dominant, since both domain and target include log smooth pairs), and thus surjective. We conclude that $U^K_d(c_i+\epsilon)=U_d^{\GIT}(c_i+\epsilon)$. 
\end{proof}

\begin{lemma}\label{9}
	With the same notation and assumption as in Proposition \ref{10}, if $(X,(c_i+\epsilon)D)\in U_d^K(c_i+\epsilon)$ is K-polystable, then $(X,D)$ is GIT$_{t(c_i+\epsilon)}$-polystable. 
\end{lemma}

\begin{proof}
	It is shown in the proof of Proposition \ref{10} that $(X,D)$ is GIT$_{t(c_i+\epsilon)}$-semistable. Suppose it is not GIT$_{t(c_i+\epsilon)}$-polystable, then there is a 1-parameter subgroup $\sigma$ of $\PGL(d+1)$ such that $$(X_0,D_0)\coloneqq\lim_{t\to 0}\sigma(t)\cdot(X,D)$$ is GIT$_{t(c_i+\epsilon)}$-polystable. Thus $\mu^{\mtc{L}_{t(c_i+\epsilon)}}([(X,D)],\sigma)=0$ and $\Fut(\mtc{X},(c_i+\epsilon)\mtc{D};\mtc{L})=0$ by Proposition \ref{7} and lemmas \ref{lem:Fut-isotrivial-Kss-degeneration}, \ref{lem:HM-isotrivial-ss-degeneration}, where $(\mtc{X},\mtc{D};\mtc{L})$ is the test configuration of the polarized pair $(X,D,\mtc{O}_X(1))$ corresponding to $\sigma$. Notice that $(X_0,D_0)\in U_d^{\GIT}(c_i+\epsilon)\subseteq U_d^{\GIT}(c_i)=U_d^K(c_i)$, in particular $(X_0,c_iD_0)$ is klt, and hence $(\mtc{X},(c_i+\epsilon)\mtc{D};\mtc{L})$ is a special test configuration (since $\epsilon$ is small). Since $(X,(c_i+\epsilon)D)$ is K-polystable, then $(X,D)\simeq (X_0,D_0)$, and hence they are in the same $\PGL(d+1)$ orbit. As a consequence, $(X,D)$ is GIT$_{t(c_i+\epsilon)}$-polystable. 
\end{proof}

\begin{theorem}\label{13}
	Let $c\in(0,1)$ be a rational number. Then there is an isomorphism between Artin stacks $${\mtc{M}}_{d}^{K}(c)\simeq {\mtc{M}}_{d}^{\GIT}(t(c))$$ with $t(c)=\frac{9c}{8+c}$ when $d= 3$ and $t(c)=\frac{6c}{5+c}$ when $d=4$. Moreover, such isomorphisms commute with wall crossing morphisms.
\end{theorem}

\begin{proof}
	This is now an immediate consequence of Proposition \ref{5}, Proposition \ref{10},  Proposition \ref{11}, Proposition \ref{7} and Proposition \ref{18} using an induction argument.
\end{proof}

\begin{remark}
	It should be noted that Proposition \ref{5}, and the results of \cite{GMG21, Pap22} are necessary in order to run the inductive argument which allows us to prove Theorem \ref{13}. This is due to the fact that we require that $U_d^{K}(c)\subseteq U^{\GIT}_d(c)$ for $c\in (0,c_1)$, which cannot be proven without the explicit methods of \cite{GMG21} for degree $3$ and \cite{Pap22} for degree $4$. In particular, this method is likely to work in many other cases (e.g. Fano complete intersections with hyperplane sections), if one can characterize the first chamber as a global quotient.
\end{remark}

\section{VGIT and K-moduli spaces of del Pezzo pairs of degree $2$}
\label{sec:deg2}

A smooth del Pezzo surface $Y$ of degree $2$ admits a natural double cover $Y\to\mb{P}^2$ defined by the linear series $|-K_Y|$, branched along a quartic curve $C_4\subseteq\mb{P}^2$. Under this cover a divisor $D\in|-K_Y|$ maps to a line $L\subseteq\mb{P}^2$, giving a crepant double cover
\[
  (Y,cD)\ \longrightarrow\ \big(\mb{P}^2,\tfrac{1}{2}C_4+cL\big).
\]
We will study the K-moduli of pairs $(X,\tfrac{1}{2}C+cC')$ admitting a $\mb{Q}$-Gorenstein smoothing to $(\mb{P}^2,\tfrac{1}{2}C_4+cL)$ for all rational numbers $c\in(0,1)$, and prove that the resulting K-moduli spaces are isomorphic to those of the log del Pezzo pairs $(Y,cD)$ of degree $2$ with $D\in|-K_Y|$ (cf.\ Theorem~\ref{42}).

We begin by making the K-moduli problem precise, in particular the notion of a family of such pairs and hence of the associated stacks. Defining a family of pairs is subtle, and considerable effort has gone into addressing this subtlety; see \cite[Chapter~8]{Kol23}. As noted above, a general pair of interest $(\mb{P}^2,\tfrac{1}{2}C_4+cL)$ has a boundary divisor with two components, so we require every family of pairs to satisfy the generalized Koll\'ar condition (cf.\ \cite[Definition~3.1]{Pan25}).

Let $\mtc{M}^K_2(c)$ (resp.\ $\ove{M}^K_2(c)$) denote the K-moduli stack (resp.\ space) with reduced scheme structure  of smoothable log del Pezzo pairs of degree~$2$. Let $\mtc{M}^K_{\mb{P}^2}(c)$ be the K-moduli stack of $(\tfrac{1}{2},c)$-marked pairs $(X,\frac{1}{2}C+cC')$ admitting a smoothing to $(\mb{P}^2,\tfrac{1}{2}C_4+cL)$ (see \cite[Definition~8.4]{Kol23} and \cite[Remark~7.20]{Xu25}), and let $\ove{M}^K_{\mb{P}^2}(c)$ be the corresponding good moduli space. In fact, here one can impose a stronger condition in defining the functor $\cM^{K}_2(c)$: one can require that for any family $(\mts{X},\frac{1}{2}\mts{C}+c\mts{C}')\rightarrow T$ parametrized by this functor, the divisors $\mts{C}$ and $\mts{C}'$ are $\bQ$-Cartier and flat over $T$ with $S_1$ fibers. 

\smallskip
We first prove an analogue of Theorem \ref{2} in this case. 

\begin{lemma}\label{lemma:gor-deg2}
	Let $(X,\frac{1}{2}C+cC')$ be a K-semistable pair which admits a $\mb{Q}$-Gorenstein smoothing to $(\mb{P}^2,\frac{1}{2}C_4+cL)$. Then the Gorenstein index of $X$ is at most $2$.
\end{lemma}

\begin{proof}
	As $(X,\frac{1}{2}C+cC')$ is K-semistable, then $X$ has at worst klt singularities. It follows from \cite[Proposition 6.2]{Hac04} that any singularity $x\in X$ is of the form $\frac{1}{n^2}(1,na-1)$, where $(a,n)=1$ and $n\neq3$ is the Gorenstein index of $X$. Thus, it suffices to show that the Gorenstein index is strictly less than $4$. Fix an arbitrary point $x\in X$. If $x\notin C'$, then by the condition that $3C'+K_X\sim 0$, we see that $K_X$ is Cartier near $x$, and thus $n=1$. Now, we may assume that $x\in C'$. The same argument in the Proof of Theorem \ref{2} shows that $n\leq 3$ as desired.
\end{proof}

\begin{corollary}\label{36}
	Let $(X,\frac{1}{2}C+cC')$ be a K-semistable pair in $\mtc{M}^K_{\mb{P}^2}(c)$, then either $X\simeq \mb{P}^2$ or $X\simeq \mb{P}(1,1,4)$.
\end{corollary}

\begin{proof}
Since $X$ has klt singularities, it is known that it is a \emph{Manetti surface} (i.e. $X$ is a deformation of $\mathbb P(a^2, b^2, c^2)$ where $a^2+b^2+c^2=3abc$, cf. \cite[Section 8.2]{Hac04}) of Gorenstein index at most $2$, hence $X$ is either $\mb{P}^2$ or $\mb{P}(1,1,4)$. Indeed, since $K_X^2=9$, by \cite{Fuj90}, we know that the only Gorenstein del Pezzo surface is $\mathbb P^2$ and by Lemma \ref{lemma:gor-deg2}, the Gorenstein index of $X$ is at most $2$, so if $X\neq\mathbb P^2$, then $X\simeq \mathbb P(1,1,4)$, the only del Pezzo with Gorenstein index $2$ \cite[Table 6]{Nak07}.
\end{proof}

\subsection{K-moduli spaces of log del Pezzo pairs of degree $2$}

In this section, we prove Theorem \ref{42}. As a consequence, the study of degree $2$ log del Pezzo pairs can be reduced to the study of K-polystable plane curve pairs.

\begin{theorem}\label{42}
	There is a natural isomorphism $\phi(c):\ove{M}^K_{\mb{P}^2}(c)\stackrel{\sim}{\longrightarrow}\ove{M}^K_2(c)$.
\end{theorem}

\begin{remark}
It is worth noting that the two corresponding K-moduli stacks are not isomorphic. Indeed, not every K-semistable del Pezzo surface of degree 2 arises as a double cover of a degeneration of $\bP^2$. For example, let $Y$ be the anticanonical model of the blowup of $\bP^2$ at eight general points lying on a smooth conic. Then $Y$ has a unique singularity of type $\frac{1}{4}(1,1)$. The surface $Y$ is K-semistable, and its K-polystable degeneration is a double cover of $\bP(1,1,4)$. However, $Y$ itself is not a double cover of any degeneration of $\bP^2$. To see this, note that every double-cover structure induces a nontrivial deck transformation, whereas $\Aut(Y)$ is trivial. 
\end{remark}

\begin{lemma}\label{39}
	Let $(Y,cD)$ be a degree $2$ K-semistable log del Pezzo pair. Then the $\mb{Q}$-Gorenstein deformation of $(Y,cD)$ is unobstructed. In particular, the moduli stack $\mtc{M}^K_2(c)$ is a smooth connected component of K-moduli stack of K-semistable log Fano surface pairs of volume $2(1-c)^2$, and the K-moduli space $\ove{M}^K_2(c)$ is normal.
\end{lemma}

\begin{proof}
	The proof is identical to \cite[Proof of Theorem 1.4]{zha22}.
\end{proof}


\begin{lemma}\label{lem:4.6}
  Let $(X,\frac{1}{2}C+cC')$ be a K-semistable pair in $\cM^K_{\bP^2}(c)$. Then the $\bQ$-Gorenstein deformation of $(X,\frac{1}{2}C+cC')$ is unobstructed.
\end{lemma}

\begin{proof}
  Let $\pi\colon\mtf{X}\rightarrow X$ be the canonical covering stack, and let
  $\mtf{C}\coloneqq\pi^*C$ and $\mtf{C}'\coloneqq\pi^*C'$. By Corollary \ref{36}, $X\simeq\mb{P}^2$ or $X\simeq\mb{P}(1,1,4)$. In the first case $\mtf{X}=X$. In the second, the local class group at the singular point is $\mb{Z}/4\mb{Z}$, generated by $\mtc{O}(1)$, and $K_X=\mtc{O}(-6)$ has class $2$, of order $2$. The relations $3C+4K_X\sim0$ and $3C'+K_X\sim0$ give $[C]=0$ and
$[C']=[K_X]$, so both lie in the subgroup trivialised by the index one cover. Hence $\mtf{C}$, $\mtf{C}'$ and $K_{\mtf{X}}$ are Cartier on $\mtf{X}$. By \cite[Proposition~3.7]{Hac04} (see also \cite[Theorem 5.3.6]{AH11}), $\bQ$-Gorenstein deformations of
  $(X,C,C')$ are identified with deformations of the triple
  $(\mtf{X},\mtf{C},\mtf{C}')$, and since the two divisors deform
  independently once a deformation of the ambient stack is fixed, we have
  \[
    \Def^{\QG}(X,C,C')\ =\ \Def(\mtf{X},\mtf{C},\mtf{C}')
    \ =\ \Def(\mtf{X},\mtf{C})\times_{\Def(\mtf{X})}\Def(\mtf{X},\mtf{C}').
  \] We claim that $\Def^{\QG}(X)=\Def(\mtf{X})$ is smooth. If $X=\bP^2$ this is
  clear, as $\bP^2$ is rigid. If $X=\bP(1,1,4)$, then $X$ is a del Pezzo
  surface whose unique singularity $p$ is the T-singularity
  $\frac{1}{4}(1,1)$. By \cite[Proposition~3.1]{HP10}, there are no
  local-to-global obstructions, i.e. the restriction map
  \[
    \Def^{\QG}(X)\longrightarrow \Def^{\QG}_{\loc}(X,p)
  \]
  is smooth. Moreover, $\Def^{\QG}_{\loc}(X,p)$ is itself smooth: the index one
  cover of $\frac{1}{4}(1,1)$ is the hypersurface singularity
  $A_1=\frac{1}{2}(1,1)$, so local $\bQ$-Gorenstein deformations of $p$ are
  the $\mu_2$-equivariant deformations of $A_1$, which are unobstructed.
  Hence $\Def^{\QG}(X)$ is smooth in both cases.

  We next show that the forgetful map
  $\Def(\mtf{X},\mtf{C})\rightarrow\Def(\mtf{X})$ is smooth; the argument for
  $\mtf{C}'$ is identical. Write $L\coloneqq\cO_{\mtf{X}}(\mtf{C})$, which is a
  line bundle on $\mtf{X}$ by Corollary \ref{36} since $\mtf{X}$ is smooth. Since $\pi$ is cohomologically
  affine and $\pi_*\cO_{\mtf{X}}=\cO_X$, we have
  $H^i(\mtf{X},\cO_{\mtf{X}})=H^i(X,\cO_X)=0$ for $i=1,2$, so $L$ extends
  uniquely to every infinitesimal deformation of $\mtf{X}$. The obstruction to
  extending the section of $L$ cutting out $\mtf{C}$ lies in
  $H^1(\mtf{X},L)$. Again since $\pi$ is cohomologically affine,
  $H^1(\mtf{X},L)=H^1(X,\pi_*L)$, where $\pi_*L$ is the rank one reflexive
  sheaf $\cO_X(\deg C)$; explicitly, $\pi_*L$ is $\cO_{\bP^2}(4)$ or
  $\cO_{\bP(1,1,4)}(8)$ for $C$, and $\cO_{\bP^2}(1)$ or $\cO_{\bP(1,1,4)}(2)$
  for $C'$. The intermediate cohomology of the reflexive sheaves $\cO(d)$ on
  a weighted projective space vanishes for every $d\in\bZ$, and hence $H^1(\mtf{X},L)=0$, and $\Def(\mtf{X},\mtf{C})\rightarrow\Def(\mtf{X})$ is smooth.

  Therefore $\Def(\mtf{X},\mtf{C},\mtf{C}')$, being the fiber product of two
  smooth morphisms of deformation functors over a smooth base, is smooth, and
  the $\bQ$-Gorenstein deformations of $(X,\frac{1}{2}C+cC')$ are
  unobstructed.
\end{proof}

\begin{lemma}\label{37}
	The K-moduli stack $\mtc{M}^K_{\mb{P}^2}(c)$ is isomorphic to $[\mtc{H}/\PGL(6)]$ where $\mtc{H}$ is a locally closed subscheme of some Hilbert scheme of pairs in $\mb{P}^5$.
\end{lemma}

\begin{proof}
	Let $\mb{H}$ be the locally closed subscheme of the Hilbert scheme parameterizing triples $(X;C,C')$ such that $X\subseteq \mb{P}^5$ has Hilbert polynomial $\chi(\mb{P}^2,\mtc{O}(2m))$ (i.e. the one embedding $\mathbb P^2$ via Veronese embedding), $3C+4K_X\sim0$ and $3C'+K_X\sim 0$. Let $\mtc{H}$ be the open subscheme of $\mb{H}$ parameterizing those triples such that $(X,\frac{1}{2}C+cC')$ is K-semistable. Then we know from Corollary \ref{36} that any $X$ in $\mtc{H}$ is either $\mb{P}^2$ or $\mb{P}(1,1,4)$. Moreover, the $\mb{Q}$-Gorenstein deformations of $\mb{P}^2$ and $\mb{P}(1,1,4)$ are unobstructed by \cite[Proposition 3.1]{HP10}, and the class groups of them are torsion free by \cite[Proposition 3.14]{Kol13}. Therefore, $\mtc{O}_{\mb{P}^5}(1)|_X$ is the unique Weil divisor class which is $\mb{Q}$-linearly equivalent to $-\frac{2}{3}K_X$. Hence, the stack $\mtc{H}$ is smooth for the same argument as Lemma \ref{lem:4.6}, and thus, the quotient stack $[\mtc{H}/\PGL(6)]$ is smooth. By our construction of $\mtc{H}$, there is a morphism $[\mtc{H}/\PGL(6)]\rightarrow \mtc{M}^K_{\mb{P}^2}(c)$, which is separated. As $\mtc H \rightarrow \mtc{M}^K_{\mb{P}^2}(c)$ is a $\mathrm{PGL}(6)$-torsor then $[\mtc{H}/\PGL(6)]\simeq \mtc{M}^K_{\mb{P}^2}(c)$
\end{proof}

\begin{lemma}\label{35}
	There is a natural morphism $\psi(c)\colon \mtc{G}(c)\rightarrow \mtc{M}^K_2(c)$ from some $\mu_2$-gerbe $\mtc{G}(c)$ of $\mtc{M}^K_{\mb{P}^2}(c)$ to $\mtc{M}^K_2(c)$.
\end{lemma}

\begin{proof}  
	Consider the universal family $(\mts{X},\frac{1}{2}\mts{C}+c\mts{C}')\rightarrow \mtc{M}^K_{\bP^2}(c)$. Using the isomorphism in Lemma \ref{37}, one obtains via  the pull-back from $\mb{P}^5$ a line bundle $\mtc{O}_{\mts{X}}(1)$ on $\mts{X}$. We have that $\mtc{O}_{\mts{X}}(\mts{C})\otimes \mtc{O}_{\mts{X}}(-2)$ is trivial on each fiber. Thus, it is the pull-back of some line bundle $\mtf{L}$ on $\mtc{M}^K_{\mb{P}^2}(c)$. Let $\psi(c):\mtc{G}(c)\rightarrow \mtc{M}^K_{\mb{P}^2}(c)$ be the $\mu_2$-gerbe obtained by taking the second root stack of $\mtf{L}$ \cite[Appendix B.1]{AGV08}. Then there is a line bundle $\mtf{L}'$ on $\mtc{G}(c)$ such that $\mtf{L}'^{\otimes2}$ is the pull-back of $\mtf{L}$. Let $\pi_{\mtc{G}(c)}:(\mts{X}_{\mtc{G}(c)};\mts{C}_{\mtc{G}(c)},\mts{C}'_{\mtc{G}(c)})\rightarrow \mtc{G}(c)$ be the pull-back of the universal family to $\mtc{G}(c)$. Then the line bundle $\mtf{N}_{\mtc{G}(c)}\coloneqq\mtc{O}_{\mts{X}_{\mtc{G}(c)}}(1)\otimes \pi_{\mtc{G}(c)}^{*}\mtf{L}'$ satisfies $\mtf{N}_{\mtc{G}(c)}^{\otimes2}\simeq \mtc{O}_{\mts{X}_{\mtc{G}(c)}}(\mts{C}_{\mtc{G}(c)})$.
	
	Now we take the double cover of $\mts{X}_{\mtc{G}(c)}$ branched along $\mts{C}_{\mtc{G}(c)}$, denoted by $$\mts{Y}_{\mtc{G}(c)}\coloneqq\Spec_{\mts{X}_{\mtc{G}(c)}}\left(\mtc{O}_{\mts{X}_{\mtc{G}(c)}}\oplus\mtf{N}_{\mtc{G}(c)}^{*}\right).$$ This double cover is also a double cover fiberwise, as $\mtf{N}_{\mtc{G}(c)}$ is a line bundle. Let $\mts{D}_{\mtc{G}}$ be the pull-back of $\mts{C}'_{\mtc{G}(c)}$ to $\mts{Y}_{\mtc{G}(c)}$. Then, every fiber $(\mts{X}_t,\frac{1}{2}\mts{C}_t+c\mts{C}'_t)$ is the $\mu_2$-quotient of $(\mts{Y}_g,c\mts{D}_g)$, where $g\in|\mtc{G}(c)|$ is the unique point over $t\in |\mtc{M}^K_{\mb{P}^2}(c)|$. Therefore, the morphism $$(\mts{Y}_{\mtc{G}(c)},c\mts{D}_{\mtc{G}(c)})\longrightarrow \mtc{G}(c)$$ is a family of K-semistable log Fano pairs, since K-semistability is preserved in finite covers \cite{Der16}, where a general fiber is a log del Pezzo pair of degree $2$. By the universal property of K-moduli stacks, we obtain a morphism $\mtc{G}(c)\rightarrow \mtc{M}^K_2(c)$.
\end{proof}

\begin{lemma}
	The composition morphism $$\mtc{G}(c)\ \stackrel{\psi(c)}{\longrightarrow} \ \mtc{M}^K_{\mb{P}^2}(c)\ \longrightarrow \ \ove{M}^K_{\mb{P}^2}(c)$$ is the good moduli space of $\mtc{G}(c)$. 
\end{lemma}

\begin{proof}
	By the definition of a good moduli space \cite[Definition 4.1]{Alp13}, we need to prove that $(\psi(c))_{*}\mtc{O}_{\mtc{G}(c)}=\mtc{O}_{\mtc{M}^K_{\mb{P}^2}(c)}$ and that $\psi(c)$ is cohomologically affine \cite[Definition 3.1]{Alp13}. Recall that $\mtc{G}(c)\rightarrow \mtc{M}^K_{\mb{P}^2}(c)$ is the $\mu_2$-gerbe, so $\mtc{G}(c)=\mtc{M}^K_{\mb{P}^2}(c)\times_{B\mb{G}_m}B\mb{G}_m$, where $f:B\mb{G}_m\rightarrow B\mb{G}_m$ is given by taking the second power, and the morphism $\mtc{M}^K_{\mb{P}^2}(c)\rightarrow B\mb{G}_m$ is the classifying morphism of $\mtf{L}$ introduced in the proof of Lemma \ref{35}. Thus to show that $\phi(c)$ is cohomologically affine, it suffices to prove $f$ is \cite[Lemma 6.4.16]{Alper-book}, which is known to hold true. Indeed, by weight decomposition, a quasi-coherent sheaf $V$ over $B\bG_m$ corresponds to a family $(V_i)_{i\in \mb{Z}}$ of vector spaces. Moreover, $W\coloneqq f_{*}V$ corresponds to $(W_i\coloneqq V_{2i})_{i\in \mb{Z}}$ . Since the correspondence $V\mapsto V_i$ is exact for every $i$, then $f_{*}$ is an exact functor. On the other hand, since $\psi(c)$ is a $\mu_2$-gerbe, we have that $(\psi(c))_{*}\mtc{O}_{\mtc{G}(c)}=\mtc{O}_{\mtc{M}^K_{\mb{P}^2}(c)}$. This completes the proof.
\end{proof}

The morphism $\mtc{G}(c)\rightarrow \mtc{M}^K_2(c)$ constructed in Lemma \ref{35} descends to a morphism $$\phi(c)\ :\ \ove{M}^K_{\mb{P}^2}(c)\ \longrightarrow\  \ove{M}^K_2(c).$$ 
In particular, $\phi(c)(X, \frac{1}{2}C + cC')=(Y, cD)$ where $Y$ is the double cover of $X$ branched at $C$ and $D=\phi(C)^{-1}(C)$.

\begin{prop}\label{41}
	The morphism $\phi(c)$ is bijective.
\end{prop}

\begin{proof}
	Since $\ove{M}^K_2(c)$ is irreducible (as it is smooth and connected by Lemma \ref{39}) and the morphism $\phi(c)$ is dominant (as log smooth pairs are dense in both domain and target), then $\phi(c)$ is surjective by the properness of the K-moduli. Suppose now $(X_1,\frac{1}{2}C_1+cC_1')$ and $(X_2,\frac{1}{2}C_2+cC_2')$ are two K-polystable pairs such that their double covers $(Y_1,cD_1)$ and $(Y_2,cD_2)$ are isomorphic. Then $X_1$ and $X_2$ are isomorphic to either $\mb{P}^2$ or $\mb{P}(1,1,4)$ by Corollary \ref{36}. Note $(X_i, \frac{1}{2}C_i)$ is canonical if and only if $Y_i$ is canonical. If $X_i$ is isomorphic to $\mathbb P^2$, then $Y_i$ is given by $w^2-f_4(x,y,z)=0$ in $\mathbb P(1,1,1,2)$ and the double cover $\pi_i\colon Y_i\rightarrow X_i\simeq \mathbb P^2$ gives us that $K_{Y_i}=\pi_i^*(K_{X_i}+\frac{1}{2}\mathcal O_{X_i}(4))$, so $K_{Y_i}$ is Cartier. Moreover, since $Y_i$ is K-semistable, then it is klt, and hence $Y_i$ has Du Val singularities. If $X_i$ is isomorphic to $\mb{P}(1,1,4)$, then $C_i$ does not contain the singularity of $X_i$ since the pair is K-semistable (by Lemma \ref{lemma:gor-deg2}), and thus $Y_i$ has exactly two non-ADE singularities. By the assumption that $(Y_1,cD_1)\simeq (Y_2,cD_2)$, we know that $X_1\simeq X_2$. Now, we assume that $(Y_1,cD_1)=(Y_2,cD_2)=(Y,cD)$ and denote the double covers by $\pi_1:Y\rightarrow X_1$ and $\pi_2:Y\rightarrow X_2$.

	Assume first that $X_i\simeq \mb{P}(1,1,4)$, and denote by $p_i$ the singularity of $X_i$. Denote by $L_i\coloneqq\mtc{O}_{X_i}(1)$ the class of a line on $X_i\simeq \mb{P}(1,1,4)$. Then the local Picard group $\Pic(p_i\in X_i)\simeq \mb{Z}/4\mb{Z}$ with $L_i$ a generator. Since $Y$ is a Klt $\mb{Q}$-Fano variety, it is simply connected \cite{Takayama}, so $Y$ has no non-trivial torsion line bundle.  As $\pi_i$ is \'{e}tale near $p_i$, then $$\pi^*_{1}L_1-\pi^*_{2}L_2$$ is torsion, and thus trivial. It follows from \cite[Definition 2.50]{KM98} that $$(\pi_i)_{*}\mtc{O}_{Y}(\pi^{*}_i(2L_i))\simeq \mtc{O}_{X_i}(2L_i)\oplus\mtc{O}_{X_i}(-2L_i),$$ and thus $$H^0(X_1,\mtc{O}_{X_1}(2L_1))\simeq H^0(Y,\pi_1^{*}\mtc{O}(2L_1))\simeq H^0(Y,\pi_2^{*}\mtc{O}(2L_2))\simeq H^0(X_2,\mtc{O}_{X_2}(2L_2)).$$ Similarly, we have that $$(\pi_i)_{*}\mtc{O}_{Y}(\pi^{*}\mtc{O}_{X_i}(4L_i))\simeq \mtc{O}_{X_i}(4L_i)\oplus\mtc{O}_{X_i},$$ and thus $$H^0(X_1,\mtc{O}_{X_1}(4L_1))\oplus\mb{C}\simeq H^0(Y,\pi_1^{*}\mtc{O}(4L_1))\simeq H^0(Y,\pi_2^{*}\mtc{O}(4L_2))\simeq H^0(X_2,\mtc{O}_{X_2}(4L_2))\oplus\mb{C}.$$
	Choosing a basis $(s_1,s_2,s_3)$ of $H^0(X_i,2L_i)$ and an element $s_4$ in $H^0(Y,\pi_i^{*}\mtc{O}(4L_i))\setminus H^0(X_i,\mtc{O}_{X_i}(4L_i))$, one obtains a morphism $$(s_1:s_2:s_3:s_4):Y\longrightarrow \mb{P}(1,1,1,2),$$ which is isomorphism to $\pi_i$ once we take the image. It follows that the two double covers $\pi_1$ and $\pi_2$ differ by an automorphism of $\mb{P}(1,1,1,2)$, and hence they are isomorphic. Therefore, we conclude that $(X_1,\frac{1}{2}C_1+cC_1')\simeq (X_2,\frac{1}{2}C_2+cC_2')$.
	
	Similarly, if $X_i\simeq \mb{P}^2$, then the same argument shows that $(X_1,\frac{1}{2}C_1+cC_1')\simeq (X_2,\frac{1}{2}C_2+cC_2')$. Therefore, the morphism $\phi(c)$ is injective, concluding the proof.
\end{proof}

Having proven the above statements, we can deduce the main theorem in this section.

\vspace{0.3cm}
\noindent\textbf{Proof of Theorem~\ref{42}.}
	Using Lemma \ref{39} and Proposition \ref{41}, we can apply Zariski's main theorem to the morphism $\phi(c)$, and conclude that $\phi(c)$ is an isomorphism.
\qed

\subsection{VGIT of plane curves and K-stability of pairs on $\mb{P}^2$}
\label{sec:vgit of plane curves}

Consider the $\PGL(3)$-action on $H_4\times H_1\coloneqq\mb{P}H^0(\mb{P}^2,\mtc{O}_{\mb{P}^2}(4))\times \mb{P}H^0(\mb{P}^2,\mtc{O}_{\mb{P}^2}(1))\simeq \mathbb P^{14}\times \mathbb P^2$ induced by the natural action of $\PGL(3)$ on $\mb{P}^2$. Take a polarization $\mtc{O}(a,b)$ with $a,b>0$ and set $t=b/a\in(0,2)$, where $t=2$ is the final wall. We call a pair $(C_4,L)\in H_4\times H_1$ \emph{GIT$_t$-(semi/poly)stable} if it is a (semi/poly)stable point under the $\PGL(3)$-action with respect to $\mtc{O}(a,b)$.

From the algorithm described in \cite{GMG18, Pap22, Pap23_thesis} and the computational package \cite{Pap_code} we obtain the following walls and chambers (cf. \cite[\S 4]{Laza-thesis} for an independent solution of the GIT problem): 
\begingroup
\setlength{\tabcolsep}{4pt} 
\begin{center}
	\begin{tabular}{ c |c c c c c c c c c c c c c c c}
		& $t_0$ & &$t_1$ & & $t_2$ & &$t_3$ & &$t_4$ & & $t_5$ & & $t_6$\\ 
		walls & 0 & &$\frac{1}{2}$ & & $\frac{4}{5}$ & &$1$ & &$\frac{8}{7}$ & & $\frac{7}{5}$ & & $2$\\ 
		chambers & &$(0,\frac{1}{2})$ & & $(\frac{1}{2}, \frac{4}{5})$& & $(\frac{4}{5}, 1)$ & &$(1, \frac{8}{7})$ & &$(\frac{8}{7}, \frac{7}{5})$ & & $(\frac{7}{5},2)$ 
	\end{tabular}
\end{center}
\endgroup
We thus obtain $11$ non-isomorphic quotients $\ove{M}^{\GIT}_{\mb{P}^2}(t_i)$, which are characterized by Theorems \ref{main thm in p4 vgit} and \ref{main thm in vgit quartics_polystable}.

In the study of the K-moduli of del Pezzo surfaces of degree $2$ and the GIT of plane quartic curves (cf. \cite{OSS16,ADL24}), we know that the double smooth conic is special, in the sense that it appears in the GIT moduli, but its counterpart does not appear in the K-moduli. Indeed, note that if it did, then the double cover of $\mathbb P^2$ branched at the double conic should appear on $\overline{M}^K_{2}(0)$ via $\phi(0)$. However, such double cover is not normal, so it cannot be K-semistable. 

We will prove the following result in this section.

\begin{theorem}\label{19}
	Let $c\in(0,1)$ be a rational number, $t=2c$, and $(C_4,C_1)\in H_4\times H_1$ such that $C_4$ is not a double smooth conic. Then $(C_4,C_1)$ is GIT$_t$-(semi/poly)stable if and only if $(\mb{P}^2,\frac{1}{2}C_4+cC_1)$ is K-(semi/poly)stable. 
\end{theorem}

\begin{lemma}\label{33}
	Let $c\in (0,1)$ be a rational number, and $$\pi:\big(\mb{P}^2\times H_4\times H_1,\frac{1}{2}\mtc{C}_4+c\mtc{C}_1\big)\longrightarrow H_4\times H_1$$ be the universal family of pairs. Then the CM line bundle $\Lambda_{\CM,c}$ is proportional to the polarization $\mtc{O}_{H_4\times H_1}(1,t)$ for $t=2c$ up to a positive constant.
\end{lemma}

\begin{proof}
	By \cite{CP21}, the CM $\mb{Q}$-line bundle with respect to $\pi$ is given by $$-\pi_{*}\big(-K_{\mb{P}^2\times H_4\times H_1/H_4\times H_1}-1/2\mtc{C}_4-c\mtc{C}_1\big)^3\ =\ \frac{3}{2}(1-c)^2\mtc{O}_{H_4\times H_1}(1,2c).$$
\end{proof}

We verify that for all GIT$_t$-polystable pairs $(C_4,L)$ given in Table \ref{tab:C^* invariant quartics}, the pair $(\mb{P}^2,\frac{1}{2}C_0+cL_0)$ is $c$-K-polystable for $t = 2c$. We will explicitly demonstrate this only for wall $t = \frac{7}{5}$, as the computations for other walls follow in the same way.

\begin{prop}\label{prop:deg2-walls-demonstration} Let $0<c<1$ be a rational number. Let $(C_0=\{F=0\},L_0=\{H=0\})$ be given by the equations of row 8 of Table \ref{tab:C^* invariant quartics}, and $C=\{x_0x_2^3-x_1^3x_2+f_4(x_0,x_1)=0\}$.
	\begin{enumerate}
		\item The pair $(\mb{P}^2,\frac{1}{2}C_0+cL_0)$ is K-polystable if and only if $c=\frac{7}{10}$.
		\item The pair $(\mb{P}^2,\frac{1}{2}C+cL_0)$ is K-semistable if $0\leq c\leq \frac{7}{10}$, and K-unstable for $\frac{7}{10}<c<1$.
	\end{enumerate}
\end{prop}

\begin{proof} 
	\begin{enumerate}
		\item Note $(\mb{P}^2,\frac{1}{2}C_0+cL_0)$ is a $\mb{T}$-pair of complexity-one with the $\mb{G}_m$-action $\lambda$ with weights $(5,1,-4)$. Since the horizontal divisors are those whose points have finite orbits, it is easy to deduce that there are no horizontal divisors on $(\mathbb P^2, C_0)$ by checking the only fixed points are the intersection of coordinate axis. The vertical divisors must be toric (and thus rational). One checks that $x_0x_2^2-ax_1^3$ for $a\in \mathbb P^1$ covers $\mathbb P^2$ (with $a=0$ and $a=\infty$ giving the coordinate axes). Then one readily checks that their $\beta$-invariant is positive when $c=\frac{7}{10}$.
		
		$\Fut(\overline{\lambda\cdot (\mathbb P^2, \frac{1}{2}C_0+L_0)})=0$ is equivalent to showing $\beta(E)=0$, where $E$ is the exceptional divisor under the $(3,1,0)$-weighted blowup  $\pi:X\rightarrow\mb{P}^2$ at $p=(0:0:1)$ (Lemma \ref{lemma:refined-complexity-one}). 
		Thus $m=\mult_{p, (3,1)}L=3$, $E^2=\frac{-1}{1\cdot 3}$, $L\cdot E=\frac{m}{1\cdot 3}=1$, $\widetilde L^2=L^2+m^2E^2=-2$ and
		$$K_X-\pi^*(K_{\mathbb P^2}+\frac{1}{2}C_0+cL) = -\frac{1}{2}\tilde{C}_0+-c\tilde{L}+\bigg(\frac{3}{2}-3c\bigg)E.$$
		Hence, $$A_{(\mb{P}^2,\frac{1}{2}C_0+cL_0)}(E)=\frac{5}{2}-3c.$$
		Furthermore, consider $\pi^*(-K_{\mathbb P^2}-\frac{1}{2}C_0-cL)-tE\sim (1-c)\tilde{L}+ \big(3(1-c)-t\big)E$. This shows that the pseudoeffective threshold is $\tau = 3(1-c)$. Moreover, $\operatorname{vol}(\pi^*(-K_{\mathbb P^2}-\frac{1}{2}C_0-cL))=(1-c)^2$, and we have
		$$P(u)\sim_{\mathbb{R}}\begin{cases}
			(1-c)\tilde{L}+ \big(3(1-c)-t\big)E, & 0\leq t\leq 1-c\\
			\big(\frac{3}{2}(1-c)-\frac{1}{2}t\big)\tilde{L}+ \big(3(1-c)-t\big)E, & 1-c\leq t\leq 3(1-c),
		\end{cases}$$
		where $P(u)$ is the positive part of the Zariski decomposition of the divisor $\pi^*(K_{\mathbb P^2}+\frac{1}{2}C_0+cL)-tE$. Hence, 
		\begin{align*}
				S_{(\mb{P}^2,\frac{1}{2}C_0+cL_0)}&(E)=\frac{1}{(1-c)^2}\bigg(\int_{0}^{\tau}\operatorname{vol}(\pi^*(-K_{\mathbb P^2}-\frac{1}{2}C_0-cL)-tE)dt\\
				&= \frac{1}{(1-c)^2}\bigg(\int_0^{1-c}\left((1-c)^2-\frac{t^3}{3}\right)dt+\int_{1-c}^{3(1-c)}\frac{1}{6}(3(1-c)-t)^2dt\bigg)\\
				&=\frac{1}{(1-c)^2}\bigg(\frac{8}{9}(1-c)^2+\frac{4}{9}(1-c)^2\bigg)\\
				&=\frac{4(1-c)}{3}.
		\end{align*}
		Thus $\beta_{(\mb{P}^2,\frac{1}{2}C_0+cL_0)}(E)=0$ if and only if $c=\frac{7}{10}$. 
		
		\item The computation in (1) shows that if $(\mb{P}^2,\frac{1}{2}C+cL_0)$ is K-semistable, then we have that $$\beta_{\frac{1}{2}C+cL_0}(E)=\frac{7-10c}{6}\geq 0,$$ hence $0\leq c\leq \frac{7}{10}$. By openness of K-semistability we have that $(\mb{P}^2,\frac{1}{2}C+\frac{7}{10}L_0)$ is K-semistable. Moreover, $(\mb{P}^2,\frac{1}{2}C)$ is K-semistable \cite{OSS16, Der16} and by interpolation (c.f. \cite[Proposition 2.13]{ADL24}), we have that $(\mb{P}^2,\frac{1}{2}C+cL_0)$ is K-semistable for any $0\leq c\leq \frac{7}{10}$.
	\end{enumerate}
\end{proof}

By the same computation, we get the following tables on the K-moduli walls, where the equation of the line $L$ is $x_0=0$, and the corresponding polystable replacements on the walls.

			

			
			
			
			


\begin{center}
	\renewcommand*{\arraystretch}{1.2}
	\begin{table}[ht]
		\centering
		\begin{tabular}{ |c  |c|c |c|c|c|}
			\hline
			Wall & Curve $C$ & Line $L$ & \makecell{Weighted blow\\-up weights} & $\beta$-invariant & \makecell{Row in \\Table \ref{tab:C^* invariant quartics}}    \\ \hline 
			
			$\frac{1}{4} $  &  $x_2(x_1^3-x_2x_0^2)$ & $x_0$ & $(3,2,0)$  &   $(1-4c)/3$ & 3 
			\\ \hline
			$\frac{1}{4} $  &  $x_0x_1x_2(x_1-x_2)$ & $x_0$ & $(0,1,1)$  &   $(1-4c)/6$ & 4 
			\\ \hline
			
			$\frac{2}{5} $  &  $x_1(x_1^3+x_0x_2^2)$  & $x_0$ & $(3,1,0)$  &   $(2-5c)/3$  & 5
			\\ \hline     
			
			$\frac{1}{2}$  &  $x_2x_1(x_0x_2-x_1^2)$&$x_0$   & $(2,1,0)$  &   $(1-2c)/2$ & 6
			\\ \hline  
			
			$\frac{4}{7} $  &  $x_0x_2^3-x_1^4$ & $x_0$ & $(4,1,0)$  &   $(4-7c)/3$ & 7
			\\ \hline  
			
			$\frac{7}{10} $  &  $x_0x_2^3-x_2x_1^3$ & $x_0$ & $(1,3,0)$  &   $(7-10c)/6$ & 8 
			\\ \hline

		\end{tabular}
		\caption{K-moduli walls and polystable replacements obtained for each pair $(\mathbb P^2, \frac{1}{2}C, L)$ with the method of Proposition \ref{prop:deg2-walls-demonstration}.}
		\label{Kwall}
	\end{table}
\end{center}
\begin{remark}\label{GIT_t polystability gives K-ps}
		Notice, that up to projective equivalence, the curves and lines in Table \ref{Kwall} are identical to the corresponding entries in Table \ref{tab:C^* invariant quartics}, except from rows $1$ and $2$, and the description in Theorem \ref{main thm in vgit quartics_polystable}. In particular, this shows that for each of the GIT$_t$-polystable pairs $(C_4,L)$ of Theorem \ref{main thm in vgit quartics_polystable}, the log Fano pairs $(\mb{P}^2,\frac{1}{2}C_4+cL)$ are K-semistable, with $t = 2c$.
\end{remark}

We next verify that for each of the strictly GIT$_t$-semistable pairs $(C_4,L)$ of Theorem \ref{strictly semistable orbits} the log Fano pairs $(\mb{P}^2,\frac{1}{2}C_4+cL)$ are K-semistable, with $t = 2c$. Let $\tilde{C}' = \{f_4(y,z) + xz(f_2(y,z) +xz) \}$ such that the coefficient $a$ of $y$ in $f_4$ and the coefficient $b$ of $y$ in $f_2$ satisfy $b = 2a$. The pair $(\tilde{C}', \tilde{L})$ is strictly GIT$_t$-semistable for all $0<t<2$ by Theorem \ref{strictly semistable orbits}, and degenerates via one-parameter subgroups to the pair $(2Q, L)$.

\begin{prop}\label{GITt ss implies c-K-ss}
	Let $(C_4,L)$ be a strictly GIT$_t$-semistable pair, where $C_4\not\simeq \tilde{C}'$. Then, the pair $(\mb{P}^2,\frac{1}{2}C_4+cL)$ is strictly K-semistable for $t = 2c$.
\end{prop}
\begin{proof}
	Let $(C_4,L)$ be a strictly GIT$_t$-semistable pair, with $C\neq \tilde{C}$. Since the pair is strictly GIT semistable, there exists a one-parameter subgroup $\lambda$ such that the limit $\lim_{t\to 0} \lambda(t)\cdot (C_4,L) = (C_0,L_0)$, where $(C_0,L_0)$ is the strictly GIT$_t$-polystable pair detailed in Theorem \ref{main thm in vgit quartics_polystable} such that $C_0 \neq 2Q$, from the above discussion. For each such pair, this one-parameter subgroup induces a family $f\colon (\mathcal{C}, \mathcal{L})\rightarrow B$, over a curve $B$, such that the fibers $(\mathcal{C}_t, \mathcal{L}_t)$ are isomorphic to $(C_4,L)$ for all $t\neq 0$, and $(\mathcal{C}_0, \mathcal{L}_0) \simeq (C_0,L_0)$. From this construction, and similarly to Lemma \ref{33} we obtain a map $\pi \colon (\mb{P}^2,\frac{1}{2}\mathcal{C}+c\mathcal{L})\rightarrow B$ which is naturally a test configuration of the pair $(\mb{P}^2,\frac{1}{2}C_4+cC_1)$ with central fiber $(\mb{P}^2,\frac{1}{2}C_0+cL_0)$. Hence, we have constructed a test configuration $g\colon \mathcal{X} \rightarrow B$ where the central fiber is a log Fano pair, which is K-polystable by the above discussion and Remark \ref{GIT_t polystability gives K-ps}, and the general fiber $\mathcal{X}_t\simeq (\mb{P}^2,\frac{1}{2}C_4+cC_1) $ is not isomorphic to $\mathcal{X}_0$. By \cite[Corollary 1.13]{CA23} and Lemma \ref{33} the general fiber $\mathcal{X}_t \simeq (\mb{P}^2,\frac{1}{2}C_4+cC_1)$ is strictly $c$-K-semistable, for $t = 2c$.
\end{proof}

\noindent\textbf{Proof of Theorem~\ref{19}.}
	It follows from Lemma \ref{33} and Theorem \ref{12} that if the surface pair $(\mb{P}^2,\frac{1}{2}C_4+cC_1)$ is K-(semi/poly)stable, then $(C_4,C_1)$ is GIT$_t$-(semi/poly)stable. For the converse, we need to compare the explicit stability conditions in Table \ref{tab:C^* invariant quartics} and Table \ref{Kwall}. For the case when $c=0$, it is known from \cite[Section 5]{ADL24} that a pair $(\mb{P}^2,\frac{1}{2}C_4)$ with $C_4$ not a double smooth conic is K-(semi/poly)stable if and only if $C_4$ is GIT-(semi/poly)stable, i.e. $(C_4,C_1)$ is GIT$_0$-(semi/poly)stable. This is covered in the above discussion, Remark \ref{GIT_t polystability gives K-ps} and Proposition \ref{GITt ss implies c-K-ss}.
	
	
\qed

\subsection{VGIT of binomials on $\mb{P}^1$ and K-stability of pairs on $\mb{P}(1,1,4)$}

Consider the $\SL(2)$-action on $H_8'\times H'_2\coloneqq\mb{P}H^0(\mb{P}^1,\mtc{O}_{\mb{P}^1}(8))\times \mb{P}H^0(\mb{P}^1,\mtc{O}_{\mb{P}^1}(2))$ induced by the natural action of $\SL(2)$ on $\mb{P}^1$. Take a polarization $\mtc{O}(a,b)$ with $a,b>0$ and set $t=b/a\in(0,+\infty)$. We call a pair $(f_8,f_2)\in H_8'\times H'_2$ \emph{GIT$_t$-(semi/poly)stable} if it is a (semi/poly)stable point under the $\SL(2)$-action with respect to $\mtc{O}(a,b)$. Let $$\mtc{M}^{\GIT}_{\mb{P}^1}(t)\coloneqq\left[(H_8'\times H_2')\sslash_t\PGL(2)\right]$$ be the VGIT moduli stack, and $\ove{M}^{\GIT}_{\mb{P}^1}(t)$ be its good moduli space --- referred to as $\ove{M}^{\GIT}_{\mathbb P(1,1,4)}(t)$ in the introduction.

\begin{prop}\label{prop:GIT-stability-(8,2)}
	A pair $(f_8,f_2)$ is GIT$_t$-(semi)stable if and only if $$\mult_p(f_8)+t\mult_p(f_2)<(\leq)4+t$$ for any point $p\in\mb{P}^1$.
\end{prop}

\begin{proof}
	This follows from the classical GIT-stability of $\SL(2)$-action on $\mb{P}H^0(\mb{P}^1,\mtc{O}_{\mb{P}^1}(d))$ (cf. \cite{MFK94}) and linearity of Hilbert-Mumford invariant (cf. \cite[Proposition 2.15]{Ben14}).
\end{proof}

Immediately, we can obtain the following result.

\begin{corollary}
	There are six walls $t=0,1,2,3,4,+\infty$, and five chambers for the VGIT moduli spaces. 
\end{corollary}

Let $\ove{M}^K_{\mb{P}(1,1,4)}(c)$ be the closure in the K-moduli space of the classes of pairs $(\mb{P}(1,1,4), \frac{1}{2}C+cC')$ where $C$ is an octic and $C'$ is a quadric, with reduced scheme structure; and $\mtc{M}^K_{\mb{P}(1,1,4)}(c)$ be the preimage of the $\ove{M}^K_{\mb{P}(1,1,4)}(c)$ under the good moduli space morphism. We aim to prove the following theorem in the section.

\begin{theorem}\label{20}
	Let $c\in(0,1)$ be a rational number, and $t=t(c)=\frac{12c}{1-c}$. Let $F_8=\{z^2-f_8(x,y)=0\}$ and $F_2=\{f_2(x,y)=0\}$ be curves on $\mb{P}(1,1,4)$ of degree $8$ and $2$ respectively. Then the pair $(\mb{P}(1,1,4),\frac{1}{2}F_8+cF_2)$ is K-(poly/semi)stable if and only if $(f_8,f_2)$ is GIT$_t$-(poly/semi)stable. Moreover, there is an isomorphism $$\ove{M}^{K}_{\mb{P}(1,1,4)}(c)\ \simeq\  \ove{M}^{\GIT}_{\mb{P}^1}(t).$$
\end{theorem}

We first need to construct the universal family. Let $F_8$ be a hyperelliptic curve in $\mb{P}(1,1,4)$ of degree $8$ given by the equation $z^2-f_8(x,y)=0$. We identify
$F_8$ to a point in the vector space $V=H^0(\mb{P}^1,\mtc{O}_{\mb{P}^1}(8))$. Since there is a canonical isomorphism $H^0(\mb{P}(1,1,4),\mtc{O}(2))\simeq H^0(\mb{P}^1,\mtc{O}(2))$, we will not distinguish $\mb{P}H^0(\mb{P}(1,1,4),\mtc{O}(2))$ with $H_2'$. Notice there is a $\mb{G}_m$-action on $(V\setminus \{0\})\times H_2'$ by acting on the first factor of weight $2$. Then $H'_8$ is naturally isomorphic to the coarse moduli space of the DM stack $[(V\setminus\{0\})/\mb{G}_m]$. There is a natural $SL(2)$-action on $V$ induced by the usual $SL(2)$-action on $H^0(\mb{P}^1,\mtc{O}(1))$. Since this $\SL(2)$-action commutes with the $\mb{G}_m$-action,
it descends to an $\SL(2)$-action on $(H'_8\times H'_2,\mtc{O}(a,b))$.

Let \[\textstyle \pi\ :\ \left(\mb{P}(1,1,4)\times V\times H_2',\frac{1}{2}\mtc{F}_8+c\mtc{F}_2\right) \longrightarrow V\times H'_2\] be the universal family of pairs over $V\times H'_2$. Here the fiber of $\pi$ over each point $(f_8,f_2)$ of $V\times H'_2$ is $(\mb{P}(1,1,4),\frac{1}{2}F_8+cF_2)$, where $F_8=\{z^2-f_8(x,y)=0\}$ and $F_2=\{f_2(x,y)=0\}$. Then the $\mb{G}_m$-action on $V\times H'_2$ naturally lifts to the universal family via 
$$\textstyle \left( (x:y:z),\frac{1}{2}(z^2-f_8=0)+c(f_2=0) \right) \ \longmapsto \ \left((x:y:tz),\frac{1}{2}(z^2-t^2f_8=0)+c(f_2=0)\right)$$ for any $t\in \mb{G}_m$. After taking the $\mb{G}_m$-quotient, we obtain a $\mb{Q}$-Gorenstein
family of log Fano pairs over $[(V\setminus\{0\})/\mb{G}_m]\times H'_2$, and the CM $\mb{Q}$-line bundle $\lambda_{\CM,\pi,c}$ on $V\times H'_2$ also descends to a $\mb{Q}$-line bundle on $H'_8\times H'_2$, denoted by $\Lambda_{\CM,c}$.

\begin{prop}\label{32}
The CM $\mathbb{Q}$-line bundle of the family over $
\bigl[(V\setminus\{0\})/\mathbb{G}_m\bigr]\times H'_2$ descends to a $\mathbb{Q}$-line bundle $\Lambda_{\mathrm{CM},c}$ on
$H'_8\times H'_2$, which is proportional to $\mathcal{O}(1-c,12c)$.
\end{prop}

\begin{proof}
	Write $\Lambda_{\CM,c}=\mtc{O}(a(c),b(c))$. First, notice that the CM $\mb{Q}$-line bundle with respect to $\pi$ over $V\times H'_2$ is given by $$-\pi_{*}\big(-K_{\mb{P}(1,1,4)\times V\times H'_2/V\times H'_2}-1/2\mtc{F}_8-c\mtc{F}_2\big)^3\ =\ p_2^{*}\mtc{O}_{H'_2}(3c(1-c)^2),$$ where $p_2:V\times H'_2\rightarrow H'_2$ is the second projection. Since $\mb{G}_m$ acts trivially on $H_2'$, then $b(c)=3c(1-c)^2$. Now we compute the degree $a(c)$. As $\Pic(V)$ is trivial, then the degree $a(c)$ equals to the $\mb{G}_m$-weight of the fiber $\lambda_{\CM,\pi,c}\otimes \mb{C}((0,f_2))$ for any $f_2\in H'_2$. It follows from \cite[Proposition 2.19]{ADL24} and \cite[Lemma 3.4]{Xu21} that $$a(c)\ =\ -3(1-c)^2\cdot\frac{1}{2}\cdot \beta_{(\mb{P}(1,1,4),1/2(2Q)+cF_2)}(Q)\ =\ \frac{3(1-c)^3}{12},$$ where $Q=\{z=0\}$ is the conic at infinity. Thus, we conclude that $\Lambda_{\CM,c}$ is proportional to $\mtc{O}(1-c,12c)$.
\end{proof}

\noindent\textbf{Proof of Theorem \ref{20}.}
	By Corollary \ref{36}, every K-semistable degeneration of  $(\mb{P}(1,1,4),\frac{1}{2}F_8+cF_2)$ is a pair on $\mb{P}(1,1,4)$. It follows from Proposition \ref{32} and Theorem \ref{12} that a pair $(F_8,F_2)$ is GIT$_t$-(semi/poly)stable if $(\mb{P}(1,1,4),\frac{1}{2}F_8+cF_2)$ is K-(semi/poly)stable, where $t=\frac{12c}{1-c}$. Suppose now $(F_8,F_2)$ is a GIT$_t$-semistable pair. We can take a family of pairs $(F_{8,b},F_{2,b})_{b\in B}$ over a smooth pointed curve $0\in B$ such that $(F_{8,0},F_{2,0})\simeq (F_8,F_2)$ and that $(F_{8,b},F_{2,b})$ is a smooth octic $F_{8,b}$ together with a union of two distinct lines $F_{2,b}$ intersecting $F_{8,b}$ transversely. For a general fibre, where $F_{8,b}$ is a smooth octic, notice that the pair $(\mb{P}(1,1,4),\frac{1}{2}F_{8,b})$ is K-semistable by \cite[Theorem 5.10]{OSS16} as it corresponds to a K-semistable degree $2$ del Pezzo surface. Furthermore, we have $F_{2,b}\sim_{\mathbb Q} -\left(K_{\mathbb P(1,1,4)}+\frac{1}{2}F_{8,b}\right)$ and for $F_{2,b}$ general the pair $(\mb{P}(1,1,4),\frac{1}{2}F_{8,b}+ F_{2,b})$ is log canonical. In particular, by interpolation of K-stability (c.f. \cite[Proposition 2.13]{ADL24}), the log pair $(\mb{P}(1,1,4),\frac{1}{2}F_{8,b}+ cF_{2,b})$ is K-semistable for all $0< c< 1$.

    By the properness of K-moduli spaces, one can find a limit $(F'_{8,0},F'_{2,0})$ as $b\to 0$ such that $(\mb{P}(1,1,4),\frac{1}{2}F'_{8,0}+cF'_{2,0})$ is K-polystable, after a finite base change. Then $(F'_{8,0},F'_{2,0})$ is GIT$_t$-polystable, and by the separatedness of GIT moduli spaces, we deduce that $(F_8,F_2)$ specially degenerates to $g\cdot (F'_{8,0},F'_{2,0})$ for some $g\in \PGL(2)$. By the openness of K-semistability, one has that $(\mb{P}(1,1,4),\frac{1}{2}F_{8}+cF_{2})$ is K-semistable. If moreover $(F_8,F_2)$ is a GIT$_t$-polystable, then we must have that $(F_8,F_2)=g\cdot (F'_{8,0},F'_{2,0})$ for some $g\in \PGL(2)$, and hence $(\mb{P}(1,1,4),\frac{1}{2}F_{8}+cF_{2})$ is K-polystable.
	
	From the equivalence of K-stability and VGIT stability, we obtain a morphism $$ \psi(c):
	\mtc{M}^{K}_{\mb{P}(1,1,4)}(c)\longrightarrow \mtc{M}^{\GIT}_{\mb{P}^1}(t),$$ which descends to a morphism $\ove{\psi(c)}\colon  \ove{M}^{K}_{\mb{P}(1,1,4)}(c)\rightarrow \ove{M}^{\GIT}_{\mb{P}^1}(t)$ between good moduli spaces. By the equivalence of K-polystability and GIT(t)-polystability proved above, $\ove{\psi(c)}$ is bijective on closed points. Moreover, since both moduli spaces are proper, and the GIT quotient is normal, $\ove{\psi(c)}$ is a proper quasi-finite birational morphism. Hence, by Zariski’s main theorem, equivalently by the fact that a finite birational morphism onto a normal variety is an isomorphism, $\ove{\psi(c)}$ is an isomorphism.
    \qed

\begin{corollary}
	For the K-moduli spaces of degree $2$ log del Pezzo pairs $(X,cD)$ with $D\in|-K_X|$, the walls are given by $$c\in\left\{0,\frac{1}{13},\frac{1}{7},\frac{1}{5},\frac{1}{4},\frac{2}{5},\frac{1}{2},\frac{4}{7},\frac{7}{10},1\right\}.$$  
\end{corollary}

\subsection{VGIT on the Kirwan blowups}

\begin{theorem}\label{40}
	Let $(X,\frac{1}{2}D_1+cD_2)$ be a K-polystable pair in $\ove{M}^K_{\bP^2}(c)$. Then 
	\begin{enumerate}
		\item[\textup{(1)}] either $(X,\frac{1}{2}D_1+cD_2)\simeq(\mb{P}^2,\frac{1}{2}C_4+cC_1)$ where $(C_4,C_1)$ is a GIT$_{2c}$-polystable plane curve pair of degree $4$ and $1$ respectively, where $C_4$ is not projectively isomorphic to $[2Q]$;
		\item[\textup{(2)}] or $(X,\frac{1}{2}D_1+cD_2)\simeq(\mb{P}(1,1,4),\frac{1}{2}F_8+cF_2)$
		where $(f_8,f_2)$ is GIT$_{\frac{12c}{1-c}}$-polystable. 
	\end{enumerate} 
	Conversely, any such pair $(\mb{P}^2,\frac{1}{2}C_4+cC_1)$ or $(\mb{P}(1,1,4),\frac{1}{2}F_8+cF_2)$ is K-polystable.
\end{theorem}

\begin{proof}
	We start with $c$-K-polystability implies GIT$_t$-polystability. Let $(X,1/2D_1+cD_2)$ be a K-polystable pair. Then, by Corollary \ref{36}, we have that $X\simeq\mb{P}^2$ or $X\simeq \mb{P}(1,1,4)$. If $X=\mb{P}^2$, then $D_1$ and $D_2$ are plane curves of degree $4$ and $1$ respectively. Notice that $D_1$ is not projectively isomorphic to $2Q$ since otherwise the $\beta$-invariant of the pair with respect to the divisor $Q$ is negative. It follows from Theorem \ref{19} that $(D_1,D_2)$ is GIT$_{2c}$-polystable. If $X\simeq\mb{P}(1,1,4)$, then the GIT$_{\frac{12c}{1-c}}$-polystability follows from Theorem \ref{20}.
	
	For the converse, let $(C_4, C_1)$ be a GIT$_t$-polystable pair such that $C_4$ is not projectively isomorphic to $2Q$. Then by Theorem \ref{19} we conclude that the pair $(\mb{P}^2,\frac{1}{2}C_4+cC_1)$ is $t/2$-K-polystable. Similarly, let $(F_8,F_2)$ be a GIT$_t$-polystable pair. Then, we conclude from Theorem \ref{20} that the pair $(\mb{P}(1,1,4),\frac{1}{2}F_8+cF_2)$ is $c$-K-polystable, where $c= \frac{t}{t+12}$, as required.
\end{proof}

Now fix a smooth plane conic $Q=\{q\coloneqq zx-y^2=0\}\subset\bP^2$, and let
$[2Q]\in H_4$ denote the corresponding non-reduced plane quartic. Let
$\ell_0\subset\bP^2$ be a line meeting $Q$ transversely, and set
$h\coloneqq(2Q,\ell_0)\in H=H_4\times H_1$. Write $V\coloneqq H^0(\bP^1,\mathcal
O_{\bP^1}(1))$ and $\mathbf V\coloneqq H^0(\bP^2,\mathcal O_{\bP^2}(1))\simeq\Sym^2V$,
so that $H^0(\bP^2,\mathcal O_{\bP^2}(k))=\Sym^k\mathbf V$. The conic $Q$ is the image of the second Veronese embedding
$\bP^1\hookrightarrow\bP^2=\bP(\mathbf V)$, so restriction to $Q$ and the
identification $Q\simeq\bP^1$ give
\[
  H^0(\bP^2,\mathcal O_{\bP^2}(4))\ \longrightarrow\ H^0(Q,\mathcal O_Q(4))\ \simeq \ 
  H^0(\bP^1,\mathcal O_{\bP^1}(8)).
\]
Multiplication by $q$ embeds $\Sym^2\mathbf V\hookrightarrow\Sym^4\mathbf V$, and the
above restriction induces a natural isomorphism
\[
  \Sym^4\mathbf V\,/\,q\cdot\Sym^2\mathbf V\ \simeq\ H^0(\bP^1,\mathcal O_{\bP^1}(8))\ \simeq\ \Sym^8V.
\]
Since $\Sym^2\mathbf V\simeq\Sym^4V\oplus\langle q\rangle$ as $\SL(2)$-representations,
this refines to a (canonical, multiplicity-free) decomposition
\begin{equation}\label{eq:decomposition}
    H^0(\bP^2,\mathcal O_{\bP^2}(4))\ =\ \Sym^4\mathbf V
  \ =\ \langle q^2\rangle\ \oplus\ q\cdot\Sym^4V\ \oplus\ \Sym^8V,
\end{equation}
where $\Sym^4V\subset\Sym^2\mathbf V$ is the
$\SL(2)$-complement of $\langle q\rangle$. Accordingly, every quartic sufficiently
close to $2Q$ can be written uniquely as
\[
  q^2+q\,g_2+g_4=0,\qquad g_2\in\Sym^4V,\ \ g_4\in\Sym^8V,
\]
with $q g_2$ and $g_4$ representing its components in $q\cdot\Sym^4V$ and
$\Sym^8V$, respectively.

\begin{prop}\label{luna slice}
Keep the notation as above. Let $Z\coloneqq G\cdot h$
be the $G$-orbit of $h$, $\ove Z$ its closure in $H$, and $\mathcal N_{Z/H}$ the normal bundle of
$Z$ in $H$. Then the following hold:
\begin{enumerate}
  \item[\textup{(i)}] there is a natural $G_h$-equivariant isomorphism
  \[
    \mathcal N_{Z/H}\big|_h\ \simeq\ H^0(\bP^1,\mathcal O_{\bP^1}(8))\ \simeq\ \Sym^8 V;
  \]
  \item[\textup{(ii)}] under the identification \eqref{eq:decomposition}, a Luna slice at
  $h$ to the orbit $Z$ is given, after possibly shrinking around $h$, by the locally closed subset
  \[
    W\ \coloneqq\ \Big\{\big(\mathbb V(q^2+f_4),\,\ell_0\big)\in H\ :\ f_4\in\Sym^8 V\Big\}.
  \]
\end{enumerate}
\end{prop}

\begin{proof}
By Theorem~\ref{main thm in vgit quartics_polystable}, the point $h$ is
$\mathrm{GIT}_t$-polystable for every $t\in(0,2)$. Hence its orbit $Z=G\cdot h$ is closed
in $H^{ss}$, and the stabilizer $G_h$ is reductive; in fact $G_h\cong\bG_m\rtimes\mu_2$.
In particular the Luna \'etale slice theorem applies at $h$.

Since $G$ acts transitively on smooth conics and, after fixing $Q$, the group
$\operatorname{Stab}_G([2Q])\cong\operatorname{PGL}(2)$ acts with a dense orbit on the lines
$H_1\cong\bP^2$, the orbit $Z=G\cdot h$ is dense in $i(H_2)\times H_1$, where
$i\colon H_2\hookrightarrow H_4$, $i(f_2)=f_2^2$, is the squaring embedding. Then
\[
  \ove Z\ \simeq\ i(H_2)\times H_1\ \simeq\ \bP^5\times\bP^2 .
\] Since
$H=\bP(\Sym^4\mathbf V)\times\bP(\mathbf V)$, one has
\[
  T_hH\ =\ \big(\Sym^4\mathbf V/\langle q^2\rangle\big)\ \oplus\ \big(\mathbf V/\langle\ell_0\rangle\big).
\]
As $Z$ is the $G$-orbit of $h$, its tangent space $T_hZ$ is the image of the
infinitesimal action $\mathfrak g\to T_hH$. By the description of $\ove Z$ above, $T_hZ$
fills the entire second factor $\mathbf V/\langle\ell_0\rangle$ and
spans $q\cdot\Sym^2\mathbf V/\langle q^2\rangle$ in the first factor. Hence
\[
  \mathcal N_{Z/H}\big|_h\ =\ T_hH/T_hZ\ \simeq\ \Sym^4\mathbf V\big/\big(q\cdot\Sym^2\mathbf V\big).
\]
By the decomposition \eqref{eq:decomposition},
\[
  \Sym^4\mathbf V\ =\ \langle q^2\rangle\ \oplus\ q\cdot\Sym^2\mathbf V/\langle q^2\rangle\ \oplus\ \Sym^8 V,
\]
the quotient is canonically $\Sym^8V$, and restriction to $Q\cong\bP^1$ identifies it with
$H^0(\bP^1,\mathcal O_{\bP^1}(8))$ as in \cite[eq.~(5.5)]{ADL24}. All maps are
$G_h$-equivariant, proving~(i).

Since $G_h$ is reductive, the short exact sequence of $G_h$-representations
\[
  0\ \longrightarrow\ T_hZ\ \longrightarrow\ T_hH\ \longrightarrow\ \mathcal N_{Z/H}\big|_h\ \longrightarrow\ 0
\]
splits $G_h$-equivariantly. Under \eqref{eq:decomposition} the summand $\Sym^8V$ provides a
$G_h$-stable complement to the orbit directions; we take it as the splitting. As $H$ is a product of projective spaces,
the affine chart of $\bP(\Sym^4\mathbf V)$ centered at $[q^2]$ realizes this complement as
the locally closed immersion
\[
  f_4\ \longmapsto\ \big(\mathbb V(q^2+f_4),\,\ell_0\big),\qquad f_4\in\Sym^8 V,
\]
which is $G_h$-equivariant with differential at $h$ equal to the inclusion
$\Sym^8V\hookrightarrow T_hH$. Its image $W$ is therefore $G_h$-invariant and transverse
to $Z$ at $h$, so $G\times_{G_h}W\to H$ is \'etale at $h$; that is, $W$ is a Luna slice at $h$. Finally, since $h \in H^{\mathrm{ss}}$ and
$H^{\mathrm{ss}} \subset H$ is open and $G$-invariant, after replacing $W$ by a sufficiently small $G_h$-invariant open neighborhood of $h$, we may assume that $W \subset H^{\mathrm{ss}}$. This proves~(ii).
\end{proof}

\subsection{K-moduli spaces as global VGIT quotients}

In the end of this section, we will prove that the K-moduli can be realized as a global GIT quotient; see Theorem \ref{43}.
\smallskip

As before, we fix a smooth plane conic $Q=\{q\coloneqq zx-y^2=0\}\subset\bP^2$, and let
$[2Q]\in H_4$ denote the corresponding non-reduced plane quartic. Let
$\ell_0\subset\bP^2$ be a line meeting $Q$ transversely, and set
$h\coloneqq(2Q,\ell_0)\in H=H_4\times H_1$. Let $Z\coloneqq G\cdot h$
be the $G$-orbit of $h$, $\ove Z$ its closure in $H$

Let $U_t\subseteq H_4\times H_1$ be the GIT$_t$-semistable locus under the $\PGL(3)$-action, and $(\mb{P}^2_{U_t},\frac{1}{2}\mtc{C}_4+c\mtc{C}_1)\rightarrow U_t$ be the universal family.

\begin{prop}\label{prop:weighted-blowup}
By performing the stacky weighted blowup of weight $2$ \textup{(see \cite[Definition 5.10]{ADL24})}
$\widetilde{\mathcal U}_t \to U_t
$ along $Z$, with exceptional divisor $\mathcal E$, followed by a contraction morphism, one obtains a universal family
\[
\big(\mathcal Z,\tfrac{1}{2}\mathcal C_{Z,4}+c\mathcal C_{Z,1}\big) \ \longrightarrow \ \widetilde{\mathcal U}_t
\]
satisfying the following properties.
\begin{enumerate}
    \item[\textup{(1)}]
    Over the open locus $\widetilde{\mathcal U}_t\setminus \mathcal E$, the family is isomorphic to
    \[
    \big(\mathbb P^2_{U_t},\tfrac{1}{2}\mathcal C_4+c\mathcal C_1\big)\big|_{U_t\setminus Z}
    \ \longrightarrow \ U_t\setminus Z.
    \]

    \item[\textup{(2)}]
    The fibers over $\mathcal E$ are of the form
    $(\mathbb P(1,1,4),\tfrac{1}{2}F_8+cF_2)$, where
    $F_8=\{z^2-f_8(x,y)=0\}$ and $F_2=\{f_2(x,y)=0\}$ for some
    $f_i\in H^0(\mathbb P^1,\mathcal O_{\mathbb P^1}(i))$ such that $f_2$ has two distinct roots.
\item[\textup{(3)}] Let $\mtc{E}_W$ be the exceptional divisor of the induced stacky weighted blowup
\[\widetilde {\mtc{W}}_t\ \coloneqq\  \widetilde{ \mtc{U}}_t\times_{U_t}W \ \longrightarrow \ W,\]
where $W$ is the Luna slice in Proposition \ref{luna slice}, and set $f_{2,0}\coloneqq\ell_0|_Q\in(H'_2)^\circ$. Then
\[\mtc{E}_W \ \simeq\ [(V\setminus\{0\})/\mathbb G_m],\] where $\mathbb G_m$ acts with weight $2$ on $V=H^0(\mathbb P^1,\mathcal O_{\mathbb P^1}(8))$. Moreover, the restriction of $\bigl(\mathcal Z,\frac12\mathcal C_{\mathcal Z,4}
+c\mathcal C_{\mathcal Z,1}\bigr)\to\widetilde {\mtc{U}}_t$
to $\mtc{E}_W$ is isomorphic to the pullback of the universal family over $[(V\setminus\{0\})/\mathbb G_m]\times H'_2$ via the morphism 
\[\mtc{E}_W\ \to\ [(V\setminus\{0\})/\mathbb G_m]\times H'_2,\qquad [f_8]\ \mapsto\ ([f_8],[f_{2,0}]).\]
\end{enumerate}
\end{prop}
\begin{proof}
Recall from Proposition~\ref{luna slice}(i) that the fiber of the normal bundle of $Z=G\cdot h$ in $U_t$ at $h$ is naturally isomorphic to
$H^0(\mathbb P^1,\mathcal O_{\mathbb P^1}(8))$ equipped with the induced $\SL(2)$-action.

Let $\widetilde{\mathcal U}_t\to U_t$ be the stacky weighted blowup of weight $2$ along the orbit $Z$, and let $\widetilde U_t\to U_t$ be the corresponding scheme-theoretic blowup. Both morphisms are $\SL(3)$-equivariant. Denote by $\mathcal E$ and $E$ the corresponding exceptional divisors. Let
\[
\big(\mathbb P^2_{\widetilde U_t},\tfrac12\mts{C}_4+c\mts{C}_1\big)\ \longrightarrow \ \widetilde U_t
\quad\text{and}\quad
(\mathbb P^2_{\widetilde{\mathcal U}_t},\tfrac12\mathcal C_4+c\mathcal C_1)\ \longrightarrow \  \widetilde{\mathcal U}_t
\]
be the pull-backs of the universal family. Let
\[
\psi\colon
(\mathcal X,\tfrac12\mathcal C_{\mathcal X,4}+c\mathcal C_{\mathcal X,1})
\longrightarrow
(\mathbb P^2_{\widetilde{\mathcal U}_t},\tfrac12\mathcal C_4+c\mathcal C_1)
\]
be the blowup of the universal family along the conic component of the fibers over $\mathcal E$, and let $\mathcal G$ denote its exceptional divisor. This blowup is $\SL(3)$-equivariant and yields a flat family over $\widetilde{\mathcal U}_t$. Its fibers over points of $\widetilde{\mathcal U}_t\setminus\mathcal E$ remain isomorphic to $\mathbb P^2$, while the fibers over $\mathcal E$ are isomorphic to $\mathbb P^2\cup\mathbb F_4$.

Let $\mathcal H$ be the strict transform of $\mathbb P^2_{\mathcal E}\subset \mathbb P^2_{\widetilde{\mathcal U}_t}$. Then the fiber of $\mathcal H$ over a point of $\mathcal E$ is $\mathbb P^2$, whereas the fiber of $\mathcal G$ is $\mathbb F_4$. We claim that $\mathcal H$ is contractible over $\widetilde{\mathcal U}_t$. Let $\mathcal Q_0$ be the pull-back to $\mathcal X$ of a conic in $\mathbb P^2$. Then $\mathcal Q_0-\mathcal G$ is relatively positive over $\widetilde{\mathcal U}_t\setminus\mathcal E$ and trivial on $\mathcal H$. Moreover, its restriction to the $\mathbb F_4$-component of a fiber over $\mathcal E$ is linearly equivalent to $4f+e$, where $f$ and $e$ denote the fiber class and negative section of $\mathbb F_4$, respectively. Hence
\[
(\mathcal Q_0-\mathcal G)\cdot f=1,
\qquad
(\mathcal Q_0-\mathcal G)\cdot e=0.
\]
It follows that $\mathcal Q_0-\mathcal G$ is relatively nef over $\widetilde{\mathcal U}_t$. Since it is trivial precisely on $\mathcal H$, there exists $a>0$ such that
\[
a(\mathcal Q_0-\mathcal G)-K_{\mathcal X}
\]
is relatively ample over $\widetilde{\mathcal U}_t$. By the relative base-point free theorem, $\mathcal Q_0-\mathcal G$ is relatively semiample and induces a contraction. After replacing $\mathcal Q_0-\mathcal G$ by a sufficiently divisible multiple, we obtain a contraction
\[
\mathcal X \longrightarrow
\mathcal Z
\coloneqq
\operatorname{Proj}_{\widetilde{\mathcal U}_t}
\bigoplus_{m\ge0}
\pi_*\mathcal O_{\mathcal X}\bigl(m(\mathcal Q_0-\mathcal G)\bigr).
\]
Since $\mathcal Q_0-\mathcal G$ is relatively semiample, cohomology and base change imply that the formation of this contraction commutes with base change. Therefore, for every $s\in\widetilde{\mathcal U}_t$, the fiber $\mathcal Z_s$ is obtained from $\mathcal X_s$ by the contraction induced by $(\mathcal Q_0-\mathcal G)|_{\mathcal X_s}$.

By construction, the fibers of $\mathcal Z$ over $\widetilde{\mathcal U}_t\setminus\mathcal E$ are isomorphic to $\mathbb P^2$, while those over $\mathcal E$ are isomorphic to $\mathbb P(1,1,4)$. Since the line bundle $\mathcal Q_0-\mathcal G$ carries a natural $\SL(3)$-linearization, the entire construction is $\SL(3)$-equivariant and therefore descends to the quotient. This proves~(1).

To determine the limiting boundary divisors, it is enough to pull
back the construction to the Luna slice $W$. Over the stacky
weighted blowup of $W$, the normal parameter has weight $2$.
Thus, locally along the exceptional divisor, we may write the
quartic equation as $q^2+t^2\widetilde f_4=0$, where $\widetilde f_4$ belongs to the summand
\[
    \Sym^8 V\ 
    \subseteq\ H^0(\mathbb P^2,\mathcal O_{\mathbb P^2}(4))
\]
in the decomposition~\eqref{eq:decomposition}. Consider the standard degeneration
\[
    \mathcal X
    \coloneqq
    \{xz-y^2=tw\}\ 
    \subseteq\ 
    \mathbb P(1,1,1,2)_{x,y,z,w}
    \times\mathbb A^1_t.
\]
For $t\neq 0$, the fiber $\mathcal X_t$ is isomorphic to
$\mathbb P^2$. Under this isomorphism, setting $w=\frac{q}{t}$ transforms the quartic equation into
\[
    w^2+\widetilde f_4(x,y,z)=0.
\]
The line boundary is given by $\ell_0(x,y,z)=0$. The central fiber $\mathcal X_0$ is isomorphic to
$\mathbb P(1,1,4)_{u,v,s}$ via
\[
    x=u^2,\qquad y=uv,\qquad z=v^2,\qquad w=s.
\]
Therefore, the two boundary divisors on the central fiber are
\[
    F_8
    =
    \{s^2+f_8(u,v)=0\},
    \qquad
    F_2
    =
    \{f_{2,0}(u,v)=0\},
\]
where
\[
    f_8(u,v)
    \coloneqq 
    \widetilde f_4(u^2,uv,v^2),\qquad
    f_{2,0}(u,v)
    \coloneqq 
    \ell_0(u^2,uv,v^2)
    =
    \ell_0|_Q.
\]
Since $\ell_0$ meets $Q$ transversely, the binary quadratic
$f_{2,0}$ has two distinct roots. Equivalently, $\operatorname{mult}_p(f_{2,0})<2$ for every $p\in\mathbb P^1$. After applying the $\PGL(3)$-action, the line $\ell_0$ varies through the open orbit of lines transverse to $Q$. This proves~(2).

It remains to identify the family over the exceptional divisor of
the Luna slice. Since the normal space to $Z$ in $U_t$ at $h$ is
naturally identified with $V=H^0(\mathbb P^1,\mathcal O_{\mathbb P^1}(8))$, and the stacky blowup assigns weight $2$ to this normal direction, its exceptional divisor is
\[
    \mtc{E}_W\ \simeq\ [(V\setminus\{0\})/\mathbb G_m],
\]
where $\mathbb G_m$ acts on $V$ with weight $2$. On the Luna slice, the line is fixed equal to $\ell_0$. Hence the
degree-two boundary on the central $\mathbb P(1,1,4)$ is fixed and
equal to
\[
    f_{2,0}=\ell_0|_Q\in(H'_2)^\circ,
\]
whereas the degree-eight boundary is parametrized by
$f_8\in V\setminus\{0\}$. Consequently, the restriction of the
family to $\mtc{E}_W$ is the pullback of the universal family along
\[
    E_W
    \rightarrow
    [(V\setminus\{0\})/\mathbb G_m]\times H'_2,\qquad
    [f_8]
    \mapsto([f_8],[f_{2,0}]),
\] which proves~(3).
\end{proof}
\begin{remark}
On the level of coarse parameter spaces, the $\PGL(3)$-saturation
of the exceptional locus is parametrized by
\[
    [H'_8\times(H'_2)^\circ/\PGL(2)].
\]
Here $(H'_2)^\circ$ is the open orbit of square-free binary
quadratics. At the level of the family and of the stacky exceptional divisor, one must retain the weight-$2$ quotient $[(V\setminus\{0\})/\mathbb G_m]$ appearing in Proposition \ref{prop:weighted-blowup}.
\end{remark}

Let $\mtc{I}_t\subseteq\mtc{O}_{U_t}$ be the ideal sheaf with
$\wt{U}_t=\Bl_{\mtc{I}_t}U_t$, and let $\ove{\mtc{I}}_t\subseteq\mtc{O}_{H_4\times H_1}$ be
its $\SL(3)$-equivariant extension, whose cosupport is $\ove{Z}$. Let
\[
  \pi_H\colon \wt{H}=\Bl_{\ove{\mtc{I}}_t}(H_4\times H_1)\longrightarrow H_4\times H_1
\]
be the blowup and $\ove{E}$ its exceptional divisor. For any sufficiently divisible
$k\gg1$, the line bundle
\[
  L_{k,t}\coloneqq \pi_H^{*}\mtc{O}(k,kt)\otimes\mtc{O}(-\ove{E})
\]
is an $\SL(3)$-linearized ample $\mathbb{Q}$-line bundle on $\wt{H}$. By \cite{Kir85}, the
$\mathrm{GIT}_t$-semistable locus of $(\wt{H},L_{k,t})$ is independent of the choice of
$k\gg1$, and is contained in $\wt{U}_t=\pi_H^{-1}(U_t)$; denote it by $\wt{H}_t^{ss}$, and
let $\wt{H}_t^{ps}\subseteq\wt{H}_t^{ss}$ be the $\mathrm{GIT}_t$-polystable locus.
Similarly write $U_t^{ps}$ for the $\mathrm{GIT}_t$-polystable locus in $H_4\times H_1$.
Finally, set
\[
  \wt{\mtc{U}}_t^{ss}\coloneqq\wt{\mtc{U}}_t\times_{\wt{U}_t}\wt{H}_t^{ss},
  \qquad
  \wt{\mtc{U}}_t^{ps}\coloneqq\wt{\mtc{U}}_t\times_{\wt{U}_t}\wt{H}_t^{ps}.
\]

\begin{theorem}\label{43}
For any rational number $c\in(1/4,1)$ and $t=2c$, there is an isomorphism
\[
\psi_t\colon
\big[\wt{\mtc{U}}^{ss}_t/\PGL(3)\big]
\xrightarrow{\ \sim\ }
\mtc{M}^K_{\mb{P}^2}(c).
\]
\end{theorem}

\begin{proof}
Let $E_{W,t}^{ps}$ denote the $\mathrm{GIT}_t$-polystable locus of the exceptional divisor $E_W$ of the blowup $\wt{W}\to W$ along the Luna slice $W$. By \cite{Kir85},
\[
\wt{U}_t^{ps}
=
\pi_H^{-1}\big(U_t^{ps}\setminus Z\big)
\cup
G\cdot E_{W,t}^{ps}.
\] By Theorem~\ref{40}, every fiber of
\[
\big(\mtc{Z},\tfrac12\mtc{C}_{Z,4}+c\,\mtc{C}_{Z,1}\big)
\longrightarrow
\wt{\mtc{U}}_t
\]
over $\wt{U}_t^{ps}$ is $c$-K-polystable. Since K-semistability is an open condition, every fiber over $\wt{\mtc{U}}_t^{ss}$ is K-semistable. Therefore, the universal property of the K-moduli stack induces a morphism
\[
\psi_t\colon
\big[\wt{\mtc{U}}^{ss}_t/\PGL(3)\big]
\longrightarrow
\mtc{M}^K_{\mb{P}^2}(c).
\] 
It remains to identify the GIT stability condition on the
exceptional divisor with the K-stability condition for pairs on $\mathbb P(1,1,4)$. Set $\tau\coloneqq\frac{12c}{1-c}$, which is the VGIT parameter appearing in Theorem \ref{20}. Since $c>1/4$, we have $\tau>4$. If $(f_8,f_2)$ is $\operatorname{GIT}_{\tau}$-semistable and $f_2$ has a double root at a point $p\in\mathbb P^1$, then Proposition~\ref{prop:GIT-stability-(8,2)} gives
\[
    \operatorname{mult}_p(f_8)+2\tau\leq 4+\tau.
\]
This is impossible when $\tau>4$. Hence every such $f_2$ is square-free.

After applying an element of $\PGL(2)$, we may therefore assume that $f_2=uv$. At the two roots $0,\infty\in\mathbb P^1$, the criterion in Proposition \ref{prop:GIT-stability-(8,2)} becomes $\operatorname{mult}_0(f_8)\leq 4$ and $\operatorname{mult}_\infty(f_8)\leq 4$. At every other point it becomes $\operatorname{mult}_p(f_8)\leq 4+\tau$, which is automatic since $\deg(f_8)=8$ and $\tau>4$. Thus the $\operatorname{GIT}_{\tau}$-semistability condition is precisely the semistability condition for the action of
\[
    \operatorname{Stab}_{\PGL(2)}([uv])
    \simeq \mathbb G_m\rtimes\mu_2
\]
on $H^0(\mathbb P^1,\mathcal O_{\mathbb P^1}(8))$. The analogous statement for polystability follows from the same numerical criterion. By Proposition \ref{prop:weighted-blowup}, this is exactly the GIT stability condition induced on the exceptional divisor of the Kirwan blow-up.

Consequently, the GIT-polystable fibers of the family over
$\widetilde U_t^{\mathrm{ss}}$ are exactly the K-polystable pairs classified in Theorem \ref{40}. The remainder of the proof is the same as the proof of \cite[Theorem~5.15]{ADL24}: the additional degree-one boundary is carried equivariantly throughout the construction, and Proposition \ref{prop:weighted-blowup} identifies its restriction on the exceptional slice with the fixed square-free quadratic $f_{2,0}$. Hence the closed-point, separatedness, quasi-finiteness, and representability arguments of loc.~cit. apply without change. Alper's criterion \cite[Proposition~6.4]{Alp13} therefore shows that
\[
    \psi_t:
    [\widetilde {\mathcal{U}_t}^{\mathrm{ss}}/\PGL(3)]
    \longrightarrow
    \mathcal M_{\mathbb P^2}^{K}(c)
\]
is an isomorphism.
\end{proof}

\begin{remark}
\label{copout}
    There are two reasons for Theorem \ref{43} not to extend to $c\in (0,\frac{1}{4}]$. Firstly, Proposition \ref{prop:GIT-stability-(8,2)} characterises GIT$_t$ (semi-)stability of pairs $(\mathbb P(1,1,4), \frac{1}{2}C+cC')$ in terms of the multiplicity of the defining (binary form) equations of $C$ and $C'$ at each $p\in \mathbb P^1$. In particular, one can use readily show that if $C$ is general and $C'$ is a a double line passing through the singular point of $\mathbb P(1,1,4)$, then the log pair is GIT$_t$-stable and, by Theorem \ref{20}, also K-stable. These pairs would have to appear `over' pairs $(\mathbb P^2, 2Q+cL)$ where $L$ is tangent to $Q$ by construction. But one may check by applying the Hilbert-Mumford criterion that such pairs are GIT$_t$-unstable for all $t>0$, thus they are not appearing in $\mtc{M}^{\GIT}_{\mb{P}^2}(t(c))$ and thus they will not appear in the Kirwan blow-up. 

The second reason is the mirror of the first one: again by Proposition \ref{prop:GIT-stability-(8,2)}  and Theorem \ref{20}, we may conclude that pairs $(\mathbb P(1,1,4), C+cC')$ where $C$ and $C'$ are defined by binary forms $f_8$ and $f_2$, respectively, such that $f_8$ has a root of multiplicity larger than $4 + c$ away from the two roots of $f_2$ are K-unstable. There are such examples for $c<\frac{1}{4}$ appearing on the exceptional divisor of the Kirwan blow-up.

We believe an explicit global quotient description of the K-moduli for $d=2$ should be possible for all $c$, but the construction would have to be different, starting by taking a different polarization on the blow-up than the natural one given by the Kirwan blow-up (i.e. use a different polarisation not close to the pullback of the polarisation used before blowing up).
\end{remark}
\medskip

\section{The walls of K-moduli spaces for degree $d\geq 5$} \label{sec:higher-degree}

For the purpose of completeness, we will determine all the K-moduli walls for $d=5,6,7,8,9$ in this section, in spite of the lack of VGIT set-up.

Let $(X,cC)$ be a log del Pezzo pair admitting a $\mb{Q}$-Gorenstein smoothing to $(\Sigma_d,cD_d)$, where $\Sigma_d$ is the smooth del Pezzo surface of degree $d$ (for $d=8$ we have $\Sigma_{8}=\mb{P}^1\times \mb{P}^1$ and $\Sigma_{8}'=\Bl_{p}\mb{P}^2$), and $D_d\in|-K_{\Sigma_d}|$ is a smooth divisor. Then the pair $(X,cC)$ is K-(poly/semi)stable if and only if $(X,\frac{1}{2}c(2C))$ is K-(poly/semi)stable. In particular, the walls for the K-moduli of pairs $(X,cC)$ with $C\in |-K_X|$ and $c\in(0,1)$ are exactly those for the K-moduli of pairs $(X,cD)$ with $D\in |-2K_X|$ and $c\in(0,\frac{1}{2})$, which correspond to the destabilization of pairs where $D=2C$ for some $C\in|-K_X|$. Thus, for $d=5$, $8$, $9$, we obtain all the K-moduli walls 
by applying \cite{ADL24,ADL21,zha22,PSW23} directly, and noting that by Theorem \ref{2} the variety $X$ is Gorenstein.

\begin{theorem}\label{44}
	Let $0<c<1$ be a rational number, and $\ove{M}^K_{d,c}$ be the K-moduli of pairs $(X,cD)$ as above.
	\begin{enumerate}
		\item If $d=9$, then there are no walls.
		\item If $d=8$ and $\Sigma_8=\mb{P}^1\times\mb{P}^1$, then there is a unique wall $c=\frac{1}{4}$.
		\item If $d=8$ and $\Sigma_8=\Bl_p\mb{P}^2$, then there are two walls $c_1=\frac{1}{5}$ and $c_2=\frac{1}{4}$.
		\item If $d=5$, then there are six walls $c=\frac{2}{17},\frac{4}{19},\frac{2}{7},\frac{8}{23},\frac{4}{9},\frac{4}{7}$.
	\end{enumerate} 
\end{theorem}

For the cases when $d=6$ and $d=7$, the K-moduli of log pairs with boundary in $|-2K_X|$ is not available in the literature, so extra computations are needed. In the rest of this section, we will prove the following result.

\begin{theorem}\label{45}
	Let $0<c<1$ be a rational number, and $\ove{M}^K_{d,c}$ be the K-moduli of pairs $(X,cD)$ as above.
	\begin{enumerate}
		\item If $d=7$, then there are three walls $c=\frac{4}{25},\frac{2}{9},\frac{2}{5}$.
		\item If $d=6$, then there are five walls $c=\frac{2}{11},\frac{1}{4},\frac{5}{14},\frac{2}{5},\frac{1}{2}$.
	\end{enumerate}
\end{theorem}

See also Appendix \ref{app:deg5-walls} for an explicit description.

\subsection{The case for $d=7$}

Before starting the proof, we give an explanation of the idea of the proof. The easy observation is that every K-wall $c=c_i$ is detected by at least one boundary pair, in the following sense: there is a pair $(X,D)$ such that $(X,c_iD)$ is K-semistable, but $(X,cD)$ is K-unstable for any $0<c<c_i$. Such a pair need not be unique; several different surfaces or strata in the corresponding anticanonical linear series may give the same value $c_i$. This multiplicity is harmless for our purposes, since below we analyze all possible surfaces $X$ with Du Val singularities and stratify the corresponding linear series $|-K_X|$, so that coincident contributions merely produce the same wall.

By volume comparison, we know that $X$ has at worst Du Val singularities, then for each possible surface $X$ with Du Val singularities, we can analyze the stability region for divisors in $|-K_X|$. Equivalently, we give a stratification of the linear series $|-K_X|$ for any possible $X$.

There are only two different del Pezzo surfaces of degree $d=7$ with Du Val singularities. The first one, which we will denote by $X$, is smooth, and it is classically obtained by blowing up $\mb{P}^2$ along two points $p=(0:0:1)$ and $q=(0:1:0)$. Let $l\subseteq X$ be the proper transform of the line $\{x=0\}$ through $p$ and $q$. The second surface, which we denote by $X'$, has precisely one $A_1$-singularity. The surface $X$ can be realized by blowing a point in the only $(-1)$-curve of $\mathbb F_1$ and contracting the proper transform of such curve to the singularity (cf. \cite[Theorem 3.4]{Hidaka-Watanabe}). We will denote this curve by $E$.

\begin{prop}\label{24}
	If a pair $(X,cD)$ is K-semistable, then $c\geq \frac{4}{25}$; if a pair $(X',cD')$ is K-semistable, then $c\geq \frac{2}{9}$.
\end{prop}

\begin{proof}
	Suppose that $(X,cD)$ is K-semistable. Then we have that $\beta_{(X,cD)}(l)\geq 0$. Since $A_{(X,cD)}(l)\leq 1$, and  $S_{(X,cD)}(l)=\frac{25}{21}(1-c)$, then we must have $c\geq \frac{4}{25}$. Similarly, one can show that
	$$\beta_{(X', cD')}(E)=A_{(X', cD')}(E)-\frac{9}{7}(1-c)\leq 1-\frac{9}{7}(1-c)\leq 0$$
	if $c\leq \frac{2}{9}$.
\end{proof}

In fact, by the same computation as in the proof, one can obtain the following results.

\begin{corollary}\label{cor 24}
	If $D\in|-K_X|$ has $l$ as a component, then $(X,cD)$ is K-unstable for any $0<c<1$; If $D'\in|-K_{X'}|$ passes through the $A_1$-singularity, then $(X',cD')$ is K-unstable for any $0<c<1$.
\end{corollary}

Let us first deal with the smooth surface $X$. A member $C\in|-K_X|$ is the proper transform of a plane cubic curve $\ove{C}$ with $p,q\in \overline C$. Thus, we can identify $\overline C$ with a cubic polynomial $$yz(ay+bz)+xf_2(y,z)+x^2f_1(y,z)+x^3.$$
Note that $a=b=0$ if and only if $l$ is in the support of $C$.

\begin{prop}\label{23}
	Let $C\in|-K_X|$ be a curve as above. If $ab\neq0$, then $(X,cC)$ is K-semistable for $c=\frac{4}{25}$; if either $a=0$ or $b=0$ (but not both), then $(X,cC)$ is K-semistable for $c=\frac{2}{9}$, but K-unstable for $c<\frac{2}{9}$.
\end{prop}

\begin{proof}
	If $ab\neq0$, then the $\mb{G}_m$-action defined by $t\cdot (x:y:z)= (tx:y:z)$ induces an isotrivial degeneration from $(X,cC)$ to $(X,cC_0)$, where $C_0$ is the proper transform of the plane cubic $yz(ay+bz)=0$, consisting of three lines through $(1:0:0)$. These lines are distinct if and only if $ab\neq 0$.  The pair $(X,cC_0)$ is a $\mb{T}$-pair of complexity-one, so one can apply Theorem \ref{30}. Indeed, note that $C_0$ is not a toric divisor and the $\mathbb G_m$-action above is in $\Aut^0(X, C_0)$. It follows that the only $\mathbb G_m$-equivariant divisors are the lines through $(1:0:0)$. Of these, the only horizontal divisor for $(X,cC_0)$ on $X$ is $l$, and one has that $\beta_{(X,cC_0)}(l)=0$ if and only if $c=\frac{4}{25}$. One can easily check that when $c=\frac{4}{25}$, the $\beta$-invariant with respect to any vertical divisor is positive. Thus $(X,\frac{4}{25}C_0)$ is K-semistable, and so is $(X,\frac{4}{25}C)$ by openness of K-semistability.
	
	If $ab=0$, then we may assume that $a=1$ and $b=0$. In this case, we claim that the coefficient of $z^2$ in $f_2(y,z)$ is non-zero. In fact, if otherwise, then the $\beta$-invariant of $(X,cC)$ with respect to the exceptional divisor over $p=(0:0:1)$ is at most $1-c-\frac{25}{21}(1-c)<0$ for any $0<c<1$. Applying the same argument and using the $\mb{G}_m$-action $t\cdot (x:y:z)=(t^2x:ty:z)$, one concludes that $(X,cC)$ is K-semistable for $c=\frac{2}{9}$, but K-unstable for $c<\frac{2}{9}$.
\end{proof}

To show that $c=\frac{4}{25}$ is the first wall, we need to show that any other $(Y, D)$ (which in this case we know it can only be $Y=X'$) does not induce a wall before $c=\frac{4}{25}$. Since by Corollary \ref{cor 24}, when the pair $(X,cC)$ is K-semistable for some $c$, the curve $\overline C$ must not have $a=b=0$, then Proposition \ref{23} provides the explicit description of the walls when $Y$ is smooth. Furthermore, if $(Y,cD)$ is K-semistable where $Y$ is a del Pezzo surface of degree $7$ and $D\in |-K_Y|$, for a point $x\in Y$, by  \cite[Proposition 4.6]{LL19}, we have
$$7(1-c)^2 = (-K_Y-cD)^2\leq \frac{9}{4}\wh{\operatorname{vol}}(x,Y,D)\leq \frac{9}{4}\wh{\operatorname{vol}}(x,Y).$$
Furthermore, if $x$ is singular, locally $(x\in Y)\sim \mb{C}^2/G$ for some non-trivial finite group $G$ and by \cite[Theorem 2.7]{LX19} we have that
$$|G|\leq 4\cdot\frac{9}{4}\cdot\frac{1}{7(1-c)^2} <2 $$
for all $c< \frac{14-3\sqrt{14}}{14}$, where we use the fact that $\wh{\operatorname{vol}}(x,Y)=\frac{4}{|G|}$. Hence, for all $c<\frac{14-3\sqrt{14}}{14}$ we have that $Y\simeq X$ is smooth, and the first wall is given by $c = \frac{4}{25}$.

Now let us focus on the surface $X'$ with an $A_1$-singularity $P$. Similar as before, a divisor $C'\in|-K_{X'}|$ comes from a plane cubic curve $\ove{C}'$ passing through the tangent vector supported at $(0:0:1)$ along the direction of $\{x=0\}$ to blow up. Thus, we can identify $C'$ with a cubic polynomial $$z^2g_1(x,y)+zg_2(x,y)+g_3(x,y)=0.$$ If $g_1(x,y)\neq0$, then $g_1(x,y)$ is a multiple of $x$. 

\begin{lemma}\label{22}
	If $g_1(x,y)=0$, then $(X',cC')$ is K-unstable for any $0<c<1$. 
\end{lemma}

\begin{proof}
	Let $E$ be the exceptional divisor of the blowup of $X'$ at $P$, which is a $(-2)$-curve. Since $g_1(x,y)=0$, then, we have that $A_{(X',cC')}(E)\leq 1-c$. On the other hand, we have that $S_{(X',cC')}(E)=\frac{9}{7}(1-c)$, and thus $$\beta_{(X',cC')}(E)\leq -\frac{2}{7}(1-c)<0$$ for any $0<c<1$. Therefore, $(X',cC')$ is K-unstable for any $0<c<1$.
\end{proof}

Now we may assume that $g_1(x,y)=x$. Another observation is the following.

\begin{lemma}
	If the coefficient of $y^2$ in $g_2(x,y)$ and the coefficient of $y^3$ in $g_3(x,y)$ are both zero, then $(X',cC')$ is K-unstable for any $0<c<1$.
\end{lemma}

\begin{proof}
	Under the assumption, the strict transform of $\{x=0\}$, denoted by $l'$, is a component of $C'$. Computing the $\beta$-invariant of $(X',cC')$ with respect to $l'$ similarly as in Lemma \ref{22}, one concludes that $(X',cC')$ is K-unstable for any $0<c<1$.
\end{proof}

\begin{prop}
	Let $C'\in|-K_{X'}|$ be a curve such that $g_1(x,y)=x$.
	\begin{enumerate}
		\item If the coefficient of $y^2$ in $g_2(x,y)$ is non-zero, then $(X',cC')$ is K-semistable when $c=\frac{2}{9}$, and K-unstable for any $0<c<\frac{2}{9}$.
		\item If the coefficient of $y^2$ in $g_2(x,y)$ is zero, but the coefficient of $y^3$ in $g_3(x,y)$ is non-zero, then $(X',cC')$ is K-semistable when $c=\frac{2}{5}$, and K-unstable for any $0<c<\frac{2}{5}$.
	\end{enumerate}
\end{prop}

\begin{proof} The proof strategy is the same as in proof of Proposition \ref{23}.
	\begin{enumerate}
		\item Notice that the $\mb{G}_m$-action of weight $(0,1,2)$ yields an isotrivial degeneration of $(X',cC')$ to $(X',cC_0')$, where $C_0'$ comes from the cubic curve $z(zx+ay^2)=0$. Since $(X',cC')$ is a $\mb{T}$-pair of complexity-one, one can apply Theorem \ref{30} to prove that it is K-polystable when $c=\frac{2}{9}$. By  openness of K-semistability, one proves that $(X',\frac{2}{9}C')$ is K-semistable. The second statement follows from Proposition \ref{24}.
		\item This is similar to the proof of (1), but here we use the $\mb{G}_m$-action of weight $(0,2,3)$.
	\end{enumerate}
\end{proof}

\begin{corollary}
	There are three walls $c=\frac{4}{25},\frac{2}{9},\frac{2}{5}$ for the K-moduli $\ove{M}^K_{7,c}$.
\end{corollary}

\subsection{The case for $d=6$}

Recall that there is a classification of degree $6$ del Pezzo surfaces with at worst Du Val singularities (see \cite[Big Table, Section 8]{CP20}). We will describe the geometry of these surfaces as needed. Our notation is as follows: for example, we denote by $X_2$ the surface with a unique $A_2$ singularity, and by $X_{1,2}$ the surface with exactly one $A_1$ singularity and one $A_2$ singularity. There are two non-isomorphic degree $6$ del Pezzo surfaces with a unique $A_1$ singularity; we denote them by $X_1$ and $X_1'$, respectively.

\begin{prop}
	Let $0<\epsilon\ll1$ be a rational number, then $\ove{M}_{6,\epsilon}$ is isomorphic to the GIT quotient $|-K_{\Sigma_6}|\sslash \Aut(\Sigma_6)$, where $\Sigma_6$ is the smooth sextic del Pezzo surface. 
\end{prop}

\begin{proof}
	The proof follows from a similar argument as in \cite[Proof of Theorem 3.2]{zha22}.
\end{proof}

Recall that $X_1$ is anti-canonical model of the blowup of $\mb{P}^2$ at $p_1=(1:0:0)$, $p_2=(0:1:0)$ and $p_3=(1:1:0)$. A curve $C\in|-K_{X_1}|$ comes from a plane cubic $\ove{C}$ passing through $p_1,p_2,p_3$. Thus, we can identify $C$ with the defining polynomial of $\ove{C}$: $$f_3(x,y)+zf_2(x,y)+z^2f_1(x,y)+z^3=0.$$ Notice that if $f_3(x,y)\neq0$, then it is a constant multiple of $xy(x-y)$.

\begin{prop}\label{26}
	Let $C\in|-K_{X_1}|$ be a curve. 
	\begin{enumerate}
		\item If $C$ satisfies that $f_3(x,y)=0$, then $(X_1,cC)$ is K-unstable for any $0<c<1$.
		\item If $f_3(x,y)\neq 0$, then $(X_1,\frac{1}{4}C)$ is K-semistable, and $(X_1,cC)$ K-unstable for any $0<c<\frac{1}{4}$.
	\end{enumerate}
\end{prop}

\begin{proof}
	Denote by $\wt{X}_1$ the surface obtained by blowing up $\mb{P}^2$ at $p_1,p_2,p_3$, and by $\ell$ the strict transform on $\wt{X}_1$ of the line containing $p_1,p_2,p_3$. Then $\wt{X}_1$ is weak Fano, and the anticanonical model morphism is $\wt{X}_1\rightarrow X_1$. We will compute the $\beta$-invariant with respect to $\ell$. We have that $$S_{(X_1,cC)}(\ell)\ =\ \frac{1-c}{6}\left(\int_0^1 6-2t^2dt+\int_1^3 (3-t)^2dt\right) \ =\ \frac{4}{3}(1-c).$$ If $f_3(x,y)=0$, then $A_{(X_1,cC)}(\ell)\leq 1-c$, and hence $\beta_{(X_1,cC)}(\ell)<0$ for any $0<c<1$. Thus $(X_1,cC)$ is K-unstable. If $f_3(x,y)\neq0$, then $A_{(X_1,cC)}(\ell)=1$. Thus, if $(X_1,cC)$ is K-semistable, one has to have $1-\frac{4}{3}(1-c)\geq0$, and hence $c\geq \frac{1}{4}$. On the other hand, consider the pair $(X_1,cC_0)$, where $C_0$ comes from the cubic $\{xy(x-y)=0\}$. Notice that $(X_1,cC_0)$ is a $\mb{T}$-pair of complexity-one, and that $(X_1,cC)$ admits an isotrivial degeneration to $(X_1,cC_0)$, induced by the $\mb{G}_m$-action $t\cdot (x:y:z)=(x:y:tz)$. One can apply Theorem \ref{30} to prove that $(X_1,cC_0)$ is K-polystable for $c=\frac{1}{4}$. Therefore, by openness of K-semistability, one deduces that $(X_1,\frac{1}{4}C)$ is K-semistable.
\end{proof}

Now let $X_1'$ be the surface which is the anti-canonical model of the blowup of $\mb{P}^2$ along $p=(1:0:0)$ and along the tangent direction of $\{x=0\}$ at $q=(0:0:1)$. A curve $C\in|-K_{X'_1}|$ comes from a plane cubic $\ove{C}$ passing through the point $p$ and the tangent vector $v$ to blow up. Thus, we can identify $C$ with the defining polynomial of $\ove{C}$: $$z^2f_1(x,y)+zf_2(x,y)+f_3(x,y)=0.$$ If $f_1(x,y)\neq0$, then we may assume that $f_1(x,y)=x$.

\begin{prop}
	Let $C\in|-K_{X'_1}|$ be a curve and $0<c<1$ be a rational number. Then the following holds.
	\begin{enumerate}
		\item If $f_1(x,y)=0$, then $(X'_1,cC)$ is K-unstable for any $0<c<1$.
		\item Assume that $f_1(x,y)=x$ and $f_2$ has a non-zero $x^2$ term. Then $(X'_1,cC)$ is K-semistable when $c=\frac{2}{11}$, and is K-unstable for any $0<c<\frac{2}{11}$.  
		\item Assume that $f_1(x,y)=x$, $f_2$ has no $x^2$ term, and $f_3$ has a non-zero $x^2y$ term. Then $(X'_1,cC)$ is K-semistable when $c=\frac{2}{11}$, and is K-unstable for any $0<c<\frac{2}{11}$. 
		\item If $f_1(x,y)=x$, $f_2$ has no $x^2$ term, and $f_3$ has no $x^2y$ term, then $(X'_1,cC)$ is K-unstable for any $0<c<1$.
	\end{enumerate}
\end{prop}

\begin{proof} Identify $C$ with the cubic equation as above.
	\begin{enumerate}
		\item The proof of (1) is the same as that of Proposition \ref{26}(1). 
		\item The $\mb{G}_m$-action of weight $(1,0,1)$ yields an isotrivial degeneration of $(X'_{1},cC)$ to a $\mb{T}$-pair of complexity-one $(X'_{1},cC_0)$, where $C_0$ comes from the plane cubic $\{xz(z+x)=0\}$. One can apply Theorem \ref{30} and Lemma \ref{lemma:refined-complexity-one} to check that $(X'_{1},\frac{2}{11}C_0)$ is K-polystable. Indeed, the horizontal divisor on $X'_{1}$ is the strict transform $l$ of the line $\{y=0\}$. One has that $A_{(X_{1}',cC_0)}(E)=1$ and that $S_{(X_{1}',cC_0)}(E)=\frac{11}{9}(1-c)$, so $\beta_{(X'_{1},cC_0)}(E)=0$ if and only if $c=\frac{2}{11}$. One also checks that the $\beta$-invariants of all vertical divisors are positive when $c=\frac{2}{11}$. Thus $(
		X'_{1},\frac{2}{11}C_0)$ is K-polystable and $(X'_{1},\frac{2}{11}C)$ is K-semistable by  openness. The same computation also shows that $(X'_{1},cC)$ is K-unstable for any $0<c<\frac{2}{11}$.
		\item We apply the same method of proof as in (2), but we use the $\mb{G}_m$-action of weight $(2,0,1)$ to obtain an isotrivial degeneration to the pair $(X'_1,cC_1)$, where $C_1$ comes from the plane cubic $\{x(z^2+xy)=0\}$.
		\item In this case, one can check that the $\beta$-invariant with respect to the exceptional divisor over $p=(0:0:1)$ is negative for any $0<c<1$.
	\end{enumerate}
\end{proof}

The surface $X_{1,1}$ is anti-canonical model of the blowup of $\mb{P}^2$ along $p=(0:1:0)$ and along the tangent direction of $\{z=0\}$ at $q=(1:0:0)$. A curve $C\in|-K_{X_{1,1}}|$ comes from a plane cubic $\ove{C}$ passing through the point $p$ and the tangent vector $v$ to blow up. Thus, we can identify $C$ with the defining polynomial of $\ove{C}$: $$x^2f_1(y,z)+xf_2(y,z)+f_3(y,z)=0.$$ If $f_1(y,z)\neq0$, then we may assume that $f_1(x,y)=z$.

\begin{prop} Let $C\in|-K_{X_{1,1}}|$ be a curve and $0<c<1$ be a rational number.
	\begin{enumerate}
		\item If $f_1(y,z)=0$, then $(X_{1,1},cC)$ is K-unstable for any $0<c<1$.
		\item Assume that $f_1(y,z)=z$ and $f_2$ has a non-zero $y^2$ term. Then $(X_{1,1},cC)$ is K-semistable when $c=\frac{5}{14}$, and is K-unstable for any $0<c<\frac{5}{14}$.
		\item Assume that $f_1(y,z)=z$. If $f_2$ has no $y^2$ term, then $(X_{1,1},cC)$ is K-unstable for any $0<c<1$. 
	\end{enumerate}
\end{prop}

\begin{proof}
	The proof of (1) is the same as in Proposition \ref{26}(1).
	For (2), first notice that the $\mb{G}_m$-action of weight $(2,1,0)$ yields an isotrivial degeneration of $(X_{1,1},cC)$ to a $\mb{T}$-pair of complexity-one $(X_{1,1},cC_0)$, where $C_0$ comes from the plane cubic $\{x(xz+ay^2)=0\}$ for some $a\neq 0$. We can apply Theorem \ref{30} and Lemma \ref{lemma:refined-complexity-one} to check that $(X_{1,1},\frac{5}{14}C_0)$ is K-polystable. Indeed, the horizontal divisor on $X_{1,1}$ is the exceptional divisor $F$ over $(1:0:0)$ which passes through one of the $A_1$-singularities. One has that $A_{(X_{1,1},cC_0)}(F)=1$ and $S_{(X_{1,1},cC_0)}(F)=\frac{14}{9}(1-c)$, so $\beta_{(X_{1,1},cC_0)}(F)=0$ if and only if $c=\frac{5}{14}$. One also checks that the $\beta$-invariants of all vertical divisors are positive when $c=\frac{5}{14}$. Thus $(X_{1,1},\frac{5}{14}C_0)$ is K-polystable and $(X_{1,1},\frac{5}{14}C)$ is K-semistable by  openness. The same computation also shows that $(X_{1,1},cC)$ is K-unstable for any $0<c<\frac{5}{14}$.
	
	For (3), one computes the $\beta$ invariant with respect to the exceptional divisor $E$ over $X_{1,1}$ coming from the line $\{z=0\}$, and prove that $\beta_{(X_{1,1},cC)}(E)<0$ for any $0<c<1$.
\end{proof}

The surface $X_{2}$ with exactly one $A_2$-singularity is obtained by blowing up $\mb{P}^2$ along a 0-dimensional subscheme of length $3$ which is supported at $(1:0:0)$ and curvilinear with respect to a conic $\{xz-y^2=0\}$. A curve $C\in|-K_{X_{2}}|$ comes from a plane cubic $\ove{C}$ passing through the 0-dimensional scheme to blow up. Thus, we can identify $C$ with the defining polynomial of $\ove{C}$: $$x^2f_1(y,z)+xf_2(y,z)+f_3(y,z)=0.$$ If $f_1(y,z)\neq0$, then we may assume that $f_1(x,y)=z$.

\begin{prop} 
	Let $C\in|-K_{X_{2}}|$ be a curve and $0<c<1$ be a rational number.
	\begin{enumerate}
		\item If $f_1(y,z)=0$, then $(X_{2},cC)$ is K-unstable for any $0<c<1$.
		\item Assume that $f_1(y,z)=z$, then $f_2(y,z)$ has non-zero $y^2$ term, and the pair $(X_2,\frac{2}{5}C)$ is K-semistable, but $(X_2,cC)$ is K-unstable for any $0<c<\frac{2}{5}$.
	\end{enumerate}
\end{prop}

\begin{proof}
	The proof of (1) is the same as in Proposition \ref{26}(1).
	For (2), if $f_1(y,z)=z$, then the assumption that $\ove{C}$ passes through the length $3$ subscheme to blow up implies that $f_2(y,z)=-y^2+ayz+bz^2$ for some $a,b\in\mb{C}$. Therefore, the $\mb{G}_m$-action $\lambda$ of weight $(2,1,0)$ yields an isotrivial degeneration of $(X_{1,1},cC)$ to a $\mb{T}$-pair of complexity-one $(X_{2},cC_0)$, where $C_0$ comes from the plane cubic $\{x(xz-y^2)=0\}$ for some $a\neq 0$. We can apply Theorem \ref{30} and Lemma \ref{lemma:refined-complexity-one} to check that $(X_{2},\frac{2}{5}C_0)$ is K-polystable. In this case, there are no horizontal divisors on $X_{1,1}$, and one can check that the $\beta$-invariants of all vertical divisors are positive when $c=\frac{2}{5}$. The condition that $\Fut(\lambda)=0$ is equivalent to saying that $\beta_{((X_{2},cC_0))}(G)=0$, where $G$ is the exceptional divisor of the $(2,1,0)$ weighted blowup at $(0:0:1)$. We have that $$A_{(X_{2},cC_0)}(G)=3-4c,\quad\textup{and}\quad S_{(X_{2},cC_0)}(G)=\frac{7}{3}(1-c),$$ and hence $\beta_{((X_{2},cC_0))}(G)=0$ if and only if $c=\frac{2}{5}$. Thus $(X_{2},\frac{2}{5}C_0)$ is K-polystable and $(X_{2},\frac{2}{5}C)$ is K-semistable by openness. The same computation also shows that $(X_{2},cC)$ is K-unstable for any $0<c<\frac{2}{5}$.
\end{proof}

Observe that the surface $X_{1,2}$ is a toric surface with Picard rank $1$, and hence a weighted projective plane. It is isomorphic to $\mb{P}(1,2,3)_{u,v,w}$. The same proof to Proposition \ref{26}(1) shows the following lemma.

\begin{lemma}
	If a curve $C\in|-K_{X_{1,2}}|$ passes through any singularity of $X_{1,2}$, then $(X_{1,2},cC)$ is K-unstable for any $0<c<1$. 
\end{lemma}

We now assume that $C\in|-K_{X_{1,2}}|$ is a curve avoiding both of the two singularities of $X_{1,2}$. Using the coordinates of the weighted projective plane $\mb{P}(1,2,3)_{u,v,w}$, we may assume that $C$ is given by the polynomial $y^3+z^2+xf_5(x,y,z)$.

\begin{prop}
	The pair $(X_{1,2},cC)$ admits an isotrivial degeneration to the $\mb{T}$-pair of complexity-one $(X_{1,2},cC_0)$, where $C_0=\{y^3-z^2=0\}$. Moreover, the pair $(X_{1,2},cC)$ is K-semistable for $c=\frac{1}{2}$, but K-unstable for $0<c<\frac{1}{2}$.
\end{prop}

\begin{proof}
	The degeneration is induced by the $\mb{G}_m$ action of weight $(0,3,2)$ on $(u,v,w)$. Let us only verify that $(X_{1,2},cC_0)$ is K-polystable when $c=\frac{1}{2}$. In fact, the horizontal divisor on $X_{1,2}$ is $E=\{u=0\}$, and hence $\beta_{(X_{1,2},cC_0)}(E)=1-2(1-c)=1-2c$, which is zero if and only if $c=\frac{1}{2}$. Hence, by Theorem \ref{30} and Lemma \ref{lemma:refined-complexity-one} $(X_{1,2},cC_0)$ is K-polystable
\end{proof}

\begin{corollary}
	There are five walls $c=\frac{2}{11},\frac{1}{4},\frac{5}{14},\frac{2}{5},\frac{1}{2}$ for the K-moduli $\ove{M}^K_{6,c}$.
\end{corollary}

\appendix

\section{Geometric characterization of quotients for degree $2$}
\label{app:deg2}
In this section we will describe explicitly the GIT quotient $(H_4\times H_1)\sslash_t\PGL(3)$, which was detailed in Section \ref{sec:vgit of plane curves} using the algorithm in \cite{GMG18, Pap22} and the computational code \cite{Pap_code}. The description of the GIT quotient gives an explicit description of the eight walls of $\ove{M}^K_{2}(c)$ which are discussed in Theorem \ref{20}. We should note, that partial results on the explicit description of strictly polystable and strictly semistable orbits of this quotient, for each wall and chamber, were achieved in \cite[\S 4]{Laza-thesis}. However, a full classification of stable and semistable orbits for each wall and chamber is missing, which we present here, using the computational methods of \cite{GMG18, Pap22}.

The computational method produces the values of the $8$ GIT walls, and the finite set $S_{1,4}$ (Lemma \ref{one-ps for quartics}) of one-parameter subgroups \cite[Definition 3.1]{GMG18} that determine the $t$-stability of all pairs $(C,L)$ for all $t$. For convenience, given a one-parameter subgroup $\lambda=\operatorname{Diag}(r_0,\dots,r_2)$, we define its dual one as $\overline{\lambda}=\operatorname{Diag}(-r_2,\dots,-r_0)$.

\begin{lemma}\label{one-ps for quartics}
	The elements of  $S_{1,4}$ are $\lambda_k$ and $\overline{\lambda}_k$, where $\lambda_k$ is one of the following:
	\begin{equation*}
		\begin{split}
			\lambda_1 = \operatorname{Diag}(5, 2, -7) & \quad \lambda_2 = \operatorname{Diag}(4, 1, -5)\\
			\lambda_3 = \operatorname{Diag}(1, 0, -1) & \quad \lambda_4 = \operatorname{Diag}(2, -1, -1).
		\end{split}
	\end{equation*}
\end{lemma}

We will give some brief preliminaries on singularity theory, which will be used throughout this appendix.

\begin{defn}[{\cite[p.88]{Arnold1976}}]
	A class of singularities $T_2$ is \emph{adjacent} to a class $T_1$, and one writes $T_1 \leftarrow T_2$ if every germ of $f \in T_2$ can be locally deformed into a germ in $T_1$ by an arbitrary small deformation. We say that the singularity $T_2$ is \emph{worse} than $T_1$, or that $T_2$ is a degeneration of $T_1$.
\end{defn}

The degenerations of the isolated singularities that appear in a quartic curve in $\mathbb{P}^2$ are described in Figure \ref{fig:dynkin} (for details, see \cite[p.88]{Arnold1976} and \cite[\S 13]{Arnol_d_1975} and \cite[Table 3]{durfee}). The above theory considers only local deformations of singularities. When we study degenerations in the GIT quotient, we are interested in global deformations. In the particular cases of quartic curves in $\mathbb{P}^2$, since by \cite{takahashi_2013} the sum of the Milnor numbers of all ADE singularities satisfies $\sum_{i=1}^r\mu(T_r)\leq 7$, by \cite[Proposition 3.1]{HP10} any local deformation of isolated singularities is induced by a global deformation. We are now in a position to classify all stable, polystable and semistable orbits for each wall $t$. 

\begin{figure}[h!]
	\centering
	\begin{tikzcd}
		\mathbf{A}_1 & \mathbf{A}_2\arrow[l] & \mathbf{A}_3\arrow[l]& \mathbf{A}_4\arrow[l]& \mathbf{A}_5\arrow[l] &\mathbf{A}_6\arrow[l] & &\\
		& &\mathbf{D}_4 \arrow[u] &\mathbf{D}_5 \arrow[l]\arrow[u]& \mathbf{D}_6 \arrow[u]\arrow[l] & \mathbf{E}_6 \arrow[lu]\arrow[bend left=20, ll] & \mathbf{E}_7 \arrow[lu]\arrow[bend left=20, ll] \arrow[l]
	\end{tikzcd}
	\caption{Degeneration of germs of isolated singularities appearing in quartic curves in $\mathbb{P}^2$}
	\label{fig:dynkin}
\end{figure}
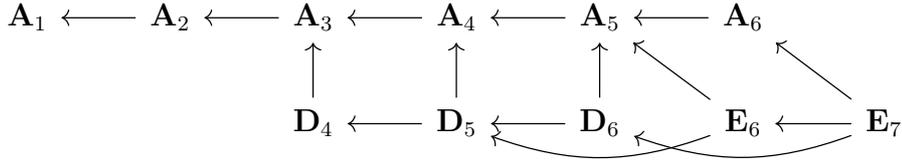

\begin{theorem}\label{main thm in p4 vgit}
	Let $(C_4, L)$ be a pair where $C_4$ is a quartic curve in $\mathbb{P}^2$ and $L$ is a line. 
	\begin{enumerate}
		\item $t\in (0, \frac{1}{2})$: The pair $(C_4,L)$ is $t$-stable if and only if $C_4$ has at worst finitely many $\mathbf{A}_4$ singularities and $L$ does not pass through the singular point of $C_4$. In particular, $C_4$ can be non-reduced.
		\item $t = \frac{1}{2}$: The pair $(C_4,L)$ is $t$-stable if and only if $C_4$ has at worst finitely many $\mathbf{A}_4$ singularities and $L$ does not pass through the singular point of $C_4$, or if $C_4$ is non-reduced and $L$ is not contained in $C_4$, or if $C_4$ has one $\mathbf{A}_2$ singularity and $L$ is not contained in $C_4$.
		\item $t\in (\frac{1}{2}, \frac{4}{5})$: The pair $(C_4,L)$ is $t$-stable if and only if $C_4$ has at worst finitely many $\mathbf{A}_4$ or $\mathbf{D}_4$ singularities and $L$ does not pass through the singular point of $C_4$.
		\item $t= \frac{4}{5}$: The pair $(C_4,L)$ is $t$-stable if and only if $C_4$ is as in the chamber $t\in (\frac{1}{2}, \frac{4}{5})$, or if $C_4$ has  one $\mathbf{A}_1$ singularity and $L$ is not contained in $C_4$.
		\item $t\in (\frac{4}{5}, 1)$: The pair $(C_4,L)$ is $t$-stable if and only if $C_4$ has at worst finitely many $\mathbf{A}_5$ or $\mathbf{D}_5$ singularities and $L$ does not pass through the singular point of $C_4$.
		\item $t= 1$: The pair $(C_4,L)$ is $t$-stable if and only if $C_4$
		is as in the chamber $t\in (\frac{4}{5}, 1)$, or if $C_4$ has  one $\mathbf{A}_1$ singularity and $L$ is not contained in $C_4$.
		\item $t \in (1, \frac{8}{7})$: The pair $(C_4,L)$ is $t$-stable if and only if $C_4$ has at worst finitely many $\mathbf{D}_6$ singularities and $L$ does not pass through the singular point of $C_4$.
		\item $t = \frac{8}{7}$: The pair $(C_4,L)$ is $t$-stable if and only if $C_4$ is as in the chamber $t\in (1, \frac{8}{7})$, or if $C_4$ is smooth and $L$ is not contained in $C_4$.
		\item $t \in (\frac{8}{7},\frac{7}{5})$: The pair $(C_4,L)$ is $t$-stable if and only if $C_4$ has at worst finitely many $\mathbf{A}_6$, $\mathbf{D}_6$ or $\mathbf{E}_6$ singularities and $L$ does not pass through the singular point of $C_4$.
		\item $t =\frac{7}{5}$: The pair $(C_4,L)$ is $t$-stable if and only if $C_4$ is as in the chamber $t \in (\frac{8}{7},\frac{7}{5})$, or if $C_4$ is smooth and $L$ is not contained in $C_4$.
		\item $t \in (\frac{7}{5},2)$: The pair $(C_4,L)$ is $t$-stable if and only if $C_4$ has at worst finitely many $\mathbf{E}_7$ singularities and $L$ does not pass through the singular point of $C_4$.
	\end{enumerate}
\end{theorem}
\begin{proof}
	Let $C_4 = \{f_4=0\}$, $L = \{h=0\}$, where $f_4$ is a homogeneous polynomial of degree $4$ and $h$ is a homogeneous polynomial of degree $1$. By \cite[Theorem 1.4]{GMG18} and \cite[Theorem 3.26, Lemma 3.25]{Pap22} the pair $(C_4, L)$ is not $t$-stable if and only if for any $g \in \operatorname{SL}(3, \mathbb{C})$ the monomials with non-zero coefficients of $(g\cdot f_4, g \cdot L)$ are  contained in a tuple of sets $N^{\ominus}_t(\lambda, x_p)$ which is maximal for every given $t$.
    These maximal sets can be found algorithmically using computational packages \cite{Pap_code, GMG-code} --- see also \cite{GMGMS} and the package \cite{Hanson-MG} for algorithms computing the non-stable, unstable and strictly polystable loci of GIT quotients by simple and semisimple groups. For a pair to have support (i.e. monomials in their equations with non-zero coefficients) is equivalent to the conditions in the statement. We verify the conditions for each $t \in (0, 2)$. 
	
	Let $t\in (0,\frac{1}{2})$, and let $(\lambda, x_p) = (\lambda_4$, $x_0)$. Then the maximal set $N^{\ominus}_{t}(\lambda, x_p)$ gives a pair $(C,L)$ defined by $
	C = \{f_4(x_1,x_2) + x_0f_3(x_1,x_2)=0\}$, $L = \{l(x_0,x_1,x_2)\}$. $C$ has a $\mathbf{D}_4$ singularity by \cite{takahashi_2013}. Hence, a $t$-stable pair cannot have $C$ with at worst a $\mathbf{D}_4$ singularity. Similarly, if we consider $(\lambda, x_p) =(\lambda_2$, $x_0)$, the maximal set $N^{\ominus}_{t}(\lambda, x_p)$ gives a pair $(C,L)$ defined by $
	C = \{x_2(f_3(x_1,x_2)+ x_2f_2(x_1,x_2))=0\}$, $L = \{l(x_0,x_1,x_2)\}$. $C$ has a $\mathbf{A}_5$ singularity by \cite{takahashi_2013}. Hence, a $t$-stable pair cannot have $C$ with at worst a $\mathbf{A}_5$ singularity. If we consider $(\lambda, x_p) =(\lambda_3$, $x_1)$, we deduce that the pair given by $C = \{f_4(x_1,x_2) + x_2x_0(f_2(x_1,x_2)+x_0x_2)=0\}$, $L = \{l(x_1,x_2)\}$, such that $C$ has a $\mathbf{A}_3$ singularity and $L$ passes through the singular point of $C$ is $t$-unstable. Hence, the pair given by $C = \{f_4(x_1,x_2) + x_2x_0(f_2(x_1,x_2)+x_0x_2)=0\}$, $L = \{l(x_0,x_1,x_2)\}$ is $t$-stable. 
	This completes the proof of $(1)$.
	
	For $t = \frac{1}{2}$, the maximal sets $N^{\ominus}_{\frac{1}{2}}(\lambda, x_p)$ are the same as those for $t\in (0,\frac{1}{2})$ but replacing the set $N^\ominus_{t\in (0,1/2)}\big((2,-1,-1), x_{0} \big)$ with the sets $N^\ominus_{1/2}\big((5,-1,-4), x_{2} \big)$, $N^\ominus_{1/2}\big((2,-1,-1), x_{0} \big)$ and $N^\ominus_{1/2}\big((1,1,-2), x_{0} \big)$, 
	which represent the monomials of the equations of any pair $(C', L')$ such that $C'$ is irreducible and has one $\mathbf{A}_2$ singularity and $L'$ is contained in $C'$, or $C'$ is irreducible and has one $\mathbf{D}_4$ singularity and $L'$ does not pass through the singular point of $C'$, or $C'$ is non-reduced and $L'$ is contained in $C'$. 
	Hence, (2) follows.
	
	Let $t\in (\frac{1}{2}, \frac{4}{5})$. The maximal $t$-non-stable sets $N^{\ominus}_t(\lambda, x_p)$ are the same as for $t = \frac{1}{2}$ but replacing the set $N^\ominus_{1/2}\big((4,1,-5), x_{0} \big)$, with the 
	sets $N^\ominus_t\big((5,2,-7),x_0 \big)$, $N^\ominus_t\big((5,-1,-4), x_0 \big)$. A pair $(C', L')$ whose defining equations have coefficients in the set $N^\ominus_t\big((1,1,-2), x_0 \big)$ require that $C'$ has (a degeneration of) a $\mathbf{D}_5$ singularity. Hence, a $t$-stable pair $(C_4, L)$ may now have $\mathbf{D}_4$ singularities but not $\mathbf{D}_5$ singularities. In particular, the pair $(C,L)$ defined by $
	C = \{f_4(x_1,x_2) + x_0f_3(x_1,x_2)=0\}$, $L = \{l(x_0,x_1,x_2)\}$, where $C$ has a $\mathbf{D}_4$ singularity and $L$ does not intersect the singular point of $C$ is $t$-stable. Therefore, 
	(3) follows.
	
	Let $t = \frac{4}{5}$. The $t$-non-stable sets $N^{\ominus}_t(\lambda, x_p)$ are the same as for $t \in (\frac{1}{2}, \frac{4}{5})$ with the addition of the set $N^\ominus_t\big((4,1,-5),x_2 \big)$, which represents the monomials of the equations of any pair $(C', L')$ such that $C'$ is irreducible with an $\mathbf{A}_1$ singularity and $L'$ is contained in $C'$.
	Therefore, $(C_4, L)$ is $\frac{4}{5}$-stable if and only if in addition to the conditions for $t$-stability when $t\in (\frac{1}{2}, \frac{4}{5})$, $L$ is not contained in $C_4$ when $C_4$ is irreducible and has (a degeneration of) a $\mathbf{A}_1$ singularity. Hence, (4) follows.

	Let $t\in (\frac{4}{5}, 1)$. The maximal $t$-non-stable sets $N^{\ominus}_t(\lambda, x_p)$ are the same as for $t = \frac{4}{5}$,  replacing set $N^\ominus_{1/2}\big((4,1,-5),x_2 \big)$ with the sets $N^\ominus_t\big((1,0,-1), x_0 \big)$ and $N^\ominus_t\big((7,-2,-5), x_0 \big)$.
	The set $N^\ominus_t\big((1,0,-1), x_0 \big)$ represents the monomials of the equations of any pair $(C', L')$ such that $C'$ has at worst a $\mathbf{D}_6$ singularity and $L'$ is a line not passing through the singular point of $C'$. Also, the set $N^\ominus_t\big((7,-2,-5), x_0 \big)$ represents the monomials of the equations of any pair $(C', L')$ such that $C'$ has at worst a $\mathbf{E}_6$ singularity and $L'$ is a line not passing through the singular point of $C'$. Therefore, $(C, L)$ is $t$-stable if and only if in addition to the conditions for $t$-stability when $t= \frac{4}{5}$, $C$ has at worst a $\mathbf{D}_5$ or $\mathbf{A}_5$ singular point and $L$ does not pass through the singular point of $C$. Hence, (5) follows.

	Let $t = 1$. The maximal $t$-non-stable sets $N^{\ominus}_t(\lambda, x_p)$ are the same as for $t \in (\frac{4}{5}, 1)$, replacing the sets $N^\ominus_{t\in (4/5,1)}\big((1,0,-1),x_1 \big)$ and $N^\ominus_{t\in (4/5,1)}\big((1,1,-2),x_2 \big)$ with the set $N^\ominus_t\big((1,0,-1), x_2 \big)$ which represents the monomials of the equations of any pair $(C', L')$ such that $L'$ is not contained at $C'$, and $C'$ has a (degeneration of) a $\mathbf{A}_1$ singularity. Furthermore, the restrictions for $t \in (\frac{4}{5}, 1)$ still apply. Therefore, a pair $(C_4, L)$ is $t$-stable if and only if satisfies the conditions in (6).
	
	Let $t\in (1, \frac{8}{7})$. The maximal $t$-non-stable sets $N^{\ominus}_t(\lambda, x_p)$ are the same as for $t = 1$,  but replacing sets $N^\ominus_{1}\big((1,1,-2),x_2 \big)$ and $N^\ominus_{1}\big((1,0,-1),x_0 \big)$ with the set $N^\ominus_t\big((1,1,-2),x_2 \big)$, which represents the monomials of the equations of any pair $(C', L')$ such that $C'$ has at worst (a degeneration of) an $\mathbf{E}_6$ singularity and $L'$ is a line not passing through this singular point. In addition, the replaced set $N^\ominus_{1}\big((1,0,-1),x_0 \big)$, represents the monomials of the equations of any pair $(C', L')$ such that $C'$ has a (degeneration of) a $\mathbf{D}_6$ singularity, and $L'$ does not pass through this singular point of $C'$. In particular, such a pair is $t$-stable and hence, (7) follows.

	Let $t=\frac{8}{7}$. The maximal $t$-non-stable sets $N^{\ominus}_t(\lambda, x_p)$ are the same as for $t \in (1, \frac{8}{7})$,  replacing the set $N^\ominus_{t\in (1,8/7)}\big((1,1,-2),x_0 \big)$ with the set $N^\ominus_t\big((5,2,-7), x_2 \big)$, which represents the monomials of the equations of any pair $(C', L')$ such that $C'$ is smooth and $L'$ is contained in $C'$. 
	Hence, a pair $(C_4, L)$ is $t$-stable if and only if satisfies the conditions in (8).
	
	Let $t\in (\frac{8}{7}, \frac{7}{5})$. The maximal $t$-non-stable sets $N^{\ominus}_t(\lambda, x_p)$ are the same as for $t=\frac{8}{7}$, replacing the set $N^\ominus_{8/7}\big((7,-2,-5), x_0 \big)$ with the sets $N^\ominus_t\big((5,-1,-4),x_0 \big)$ and $N^\ominus_t\big((2,-1,-1),x_0 \big)$, which represent the monomials of the equations of any pair $(C', L')$ such that $C'$ is irreducible and has at worst (a degeneration of) a $\mathbf{E}_7$ singularity and $L'$ is a line not passing through the singular point of $C'$, or $C'$ has non ADE singularities and $L'$ does not pass through the singular points. Therefore, $(C_4, L)$ is $t$-stable if and only if in addition to the conditions for $t$-stability when $t= \frac{8}{7}$, $C$ has at worst a $\mathbf{A}_6$, $\mathbf{D}_6$ or $\mathbf{E}_6$ singular point and $L$ is a line not passing through this singular point. Hence, 
	(9) follows.
	
	Let $t = \frac{7}{5}$. The maximal $t$-non-stable sets $N^{\ominus}_t(\lambda, x_p)$ are the same as for $t\in (\frac{8}{7}, \frac{7}{5})$, replacing the sets $N^\ominus_{t\in (8/7,7/5)}\big((5,-2,-7),x_2 \big)$ and $N^\ominus_{t\in (8/7,7/5)}\big((1,0,-1),x_2 \big)$ with the set $N^\ominus_t\big((4,1,-5),x_2 \big)$, which represents the monomials of the equations of any pair $(C', L')$ such that $C'$ is smooth and $L'$ is not contained in $C'$. Hence, a pair $(C_4, L)$ is $t$-stable if and only if satisfies the conditions in (10).

	Let $t\in (\frac{7}{5},2)$. The maximal $t$-non-stable sets $N^{\ominus}_t(\lambda, x_p)$ are the same as for $t = \frac{7}{5} $, removing the set $N^\ominus_{7/5}\big((5,-1,-4),x_0 \big)$, which represents the monomials of the equations of any pair $(C', L')$ such that $C'$ is irreducible with $\mathbf{E}_7$ singularities and $L'$ does not pass through the singular point of $C'$. Hence, a $t$-stable pair $(C_4, L)$ can have $C_4$ with at worst $\mathbf{E}_7$ singularities, with $L$ not passing through the singular point, proving (11).
	
\end{proof}

\begin{theorem}\label{main thm in vgit quartics_polystable}
	Let $t \in (0,2)$. If $t$ is a chamber, then $\overline{M(t)}$ is the compactification of the stable loci $M(t)$ by the closed $\operatorname{SL}(3)$-orbit in $\overline{M(t)}\setminus M(t)$ represented by the pair $(\tilde{C}_4,\tilde{L})$, where $\tilde{C}_4$ is the unique $\mathbb{G}_m$-invariant double conic, and $\tilde{L}$ is a line intersecting $\tilde{C}_4$ at two distinct points, or by the closed $\operatorname{SL}(3)$-orbit in $\overline{M(t)}\setminus M(t)$ represented by the pair $(\overline{C}_4,\overline{L})$, where $\overline{C}_4$ is the unique $\mathbb{G}_m$-invariant intersection of two conics on general position with $\mathbf{A}_3$ singularities, and $\overline{L}$ is a line intersecting $\overline{C}_4$ at two distinct points. If $t = t_i$, for $i = 1,2,3,4,5$, then $\overline{M(t_i)}$ is the compactification of the stable loci $M(t_i)$ by the three closed $\operatorname{SL}(3)$-orbits in $\overline{M(t)}\setminus M(t)$ represented by the uniquely defined pairs $(\tilde{C}_4,\tilde{L})$ and $(\overline{C}_4,\overline{L})$ described above, and the $\mathbb{G}_m$-invariant pairs $(C_i, L_i)$, $(C'_i, L'_i)$ uniquely defined as follows:
	\begin{enumerate}
		\item the reducible quartic $C_1$ with $1$ $\mathbf{A}_5$ and $1$ $\mathbf{A}_2$ singularities, and the line $L_1$ not passing through these singularities, and the reducible quartic curve $C'_1$ with $1$ $\mathbf{D}_4$ singularity, and the line $L'_1$ not passing through the singular point;
		\item the reducible quartic curve  $C_2$ with $1$ $\mathbf{D}_5$ singularity, and the line $L_2$ not passing through the singular point;
		\item the reducible quartic curve  $C_3$ with $1$ $\mathbf{D}_6$ singularity and $1$ $\mathbf{A}_1$ singularity, and the line $L_3$ passing through the $\mathbf{D}_6$ singular point;
		\item the irreducible quartic curve  $C_4$ with $1$ $\mathbf{E}_6$ singularity, and the line $L_4$ not passing through the singular point;
		\item the reducible quartic curve $C_5$ with $1$ $\mathbf{E}_7$ singularity, and the line $L_5$ not passing through the singular point.
	\end{enumerate}
\end{theorem}

\begin{prop}\label{c^* invariant quartics}
	Let $(C,L)$ be a pair that is invariant under a non-trivial  $\mathbb{C}^*$-action. Suppose the singularities of $C$ and the intersections with $L$ are given as in the first and third entries in one of the rows of Table  \ref{tab:C^* invariant quartics}, respectively. Then $(C,L)$ is projectively equivalent to $(\{F=0\},\{H=0\})$ for $F$ and $H$ as in the fourth and fifth entries in the same row of Table \ref{tab:C^* invariant quartics}, respectively. In particular, any such pair $(C,L)$ is unique. Conversely, if $(C,L)$ is given by equations as in the fourth and fifth entries in a given row of Table \ref{tab:C^* invariant quartics}, then $(C,L)$ has singularities and intersections as in the first and third entries in the same row of Table \ref{tab:C^* invariant quartics} and $(C,L)$ is $\mathbb{C}^*$-invariant. Furthermore, the element $\lambda\in \operatorname{SL}(3,\mathbb{C}^*)$, as defined in Lemma \ref{one-ps for quartics}, given in the entry of the corresponding row of Table \ref{tab:C^* invariant quartics} is a generator of the $\mathbb{C}^*$-action.
	
	\begin{table}[h!]
		\centering
		\begin{tabular}{|c|c|c|c|c|c |}
			\hline
			$\operatorname{Sing}(C)$  & $C$ reducible & \makecell{Intersection \\ $C\cap L$} & $F$ & $H$& $\lambda$\\
			\hline \hline
			Double conic & No & Two double points & $(x_1^2+x_0x_2)^2$& $x_1$ & $\lambda_3$\\
			\hline
			$\mathbf{A}_3$ at $P_1$, $\mathbf{A}_3$ at $P_2$& Yes & non-singular & $f_2(x_1^2,x_0x_2)$& $x_1$& $\lambda_3$ \\
			\hline
			$\mathbf{A}_5$ at $P_1$, $\mathbf{A}_2$ at $P_2$ & Yes & $P_2+ P_3$ & $x_2(x_1^3+x_0^2x_2)$& $x_0$& $\lambda_2$ \\
			\hline
			$\mathbf{D}_4$ at $P_1$, $\mathbf{A}_1$ at $P_2$ & Yes & $L = \{x_0\}$ & $x_0f_3(x_1,x_2)$& $x_0$& $\lambda_3$ \\
			\hline
			$\mathbf{D}_5$ at $P_1$, $\mathbf{A}_1$ at $P_2$ & Yes & $P_2$ & $x_1(x_1^3+x_0x_2^2)$& $x_0$& $\lambda_2$ \\
			\hline
			$\mathbf{D}_6$ at $P_1$, $\mathbf{A}_1$ at $P_2$ & Yes & $P_2+ P_3$ & $x_0x_1(x_1^2+x_0x_2)$& $x_2$& $\lambda_3$ \\
			\hline
			$\mathbf{E}_6$ at $P_1$& No & $P_3$ & $x_1^4+x_0^3x_2$& $x_2$& $\lambda_1$ \\
			\hline
			$\mathbf{E}_7$ at $P_1$ & Yes & not $P_1$ & $x_2(x_1^3+x_0x_2^2)$& $x_0$& $\ove{\lambda}_2$ \\
			\hline
		\end{tabular}
		\caption{Some degree $(4,1)$ pairs $(C,L)\subset \mathbb P^2$ invariant under $\mathbb{C}^*$-action $\lambda$.}
		\label{tab:C^* invariant quartics}
	\end{table}
\end{prop}
\begin{proof}
	The double conic is unique, given up to projective equivalence by $F$ as in row 1 of the Table. Similarly, the curve with $2$ $\mathbf{A}_3$ singularities is uniquely given by $F$ as in row $2$. These are the $2$ polystable orbits for the classical GIT of quartic curves, and hence are $\mathbb{C}^*$-invariant. 
	
	Assume that a quartic curve $C$ is reducible with an $\mathbf{A}_2$ singularity in one component and an $\mathbf{A}_5$ singularity in the other component. Since $C$ is reducible, $C = L + T$, where $L$ is a line and $T$ is a cubic. Without loss of generality, we may assume that $L = \{x_2=0\}$. Then $T$ must be a singular cubic with an $\mathbf{A}_2$ singularity. Hence, $T$ is given, up to projective equivalence, by $T = \{x_1^3+x_0^2x_2\}$m and $C$ is given by $x_2(x_1^3+x_0^2x_2)$, with an $\mathbf{A}_5$ singularity as required. 
	
	Similarly, let $C$ be a quartic curve which is reducible with an $\mathbf{D}_4$ singularity at the point $(1:0:0)$. Since $C$ is reducible, $C = L + T$, where $L$ is a line and $T$ is a singular cubic. Without loss of generality, we may assume that $L = \{x_0=0\}$. Then $T$ must be a singular cubic with an $\mathbf{D}_4$ singularity, such that $L\cap T = \varnothing$. Hence, $T$ is given, up to projective equivalence, by $T = \{f_3(x_1,x_2)\}$, i.e. three lines intersecting at one point, and $C$ is given by $x_0f_3(x_1,x_2)$, with a $\mathbf{D}_4$ singularity at $(1:0:0)$ as required. 
	
	A similar analysis shows that a curve $C$ with singularities $\mathbf{D}_5$, $\mathbf{D}_6$ and $\mathbf{E}_7$ is generated by the equations $F$ of rows $5$, $6$, $8$ of Table \ref{tab:C^* invariant quartics}. Irreducible quartics with an $\mathbf{E}_6$ singularity have equation $F = x_1^4+x_0^3x_2+\alpha x_1^2x_0^2$; it is not hard to see that this is $\mathbb{C}^*$-invariant if and only if $\alpha = 0$, as in row 7 in Table \ref{tab:C^* invariant quartics}. Furthermore, it is trivial to check that each one-parameter subgroup $\lambda$ in the corresponding row of Table \ref{tab:C^* invariant quartics} leaves $C$ invariant, and therefore $\lambda$ is a generator of the $\mathbb{C}^*$-action. Furthermore, it is straightforward to verify that the corresponding intersections $C\cap L$ are as above, and we omit the analysis.
\end{proof}


	

\begin{proof}[Proof of Theorem \ref{main thm in vgit quartics_polystable}]
	Suppose $(C,L)$—defined by polynomials $F$ and $H$—belongs to a closed strictly $t$-semistable orbit. By \cite[Theorem 1.6]{GMG18}, they are generated by monomials in $N^0_t(\lambda,x_i)$ for some $(\lambda,x_i)$ such that $N^\oplus_t(\lambda,x_i)$ is maximal with respect to the containment order of sets. Since there is a finite number of $\lambda$ to consider (those in Lemma \ref{one-ps for quartics}), this is a finite computation which can be carried out by software \cite{GMG18, GMG19, Pap_code}. For each pair $(\lambda,x_i)$, there is a change of coordinates that gives a natural bijection between $N^0_t(\lambda,x_i)$ and $N^0_t(\lambda,x_{2-i})$. Therefore, about half of the values are redundant, and we have two possible choices for each $F$ and $H$ if $t=t_2,\dots,t_5$ and three choices if $t=t_1$. 
	
	By \cite[Lemma 3.2]{GMG18} we can check that the pairs $(\tilde{C},\tilde{L})$ and $(\overline{C}, \overline{L})$ corresponding to $\tilde{F} = (x_1^2+x_0x_2)^2$, $\tilde{H} = x_1$ and $\tilde{F} = f_2(x_1,x_0x_2)$, $\tilde{H} = x_1$  are both strictly $t$-semistable. Suppose that $(\lambda,x_i)=(\lambda_3,x_1)$. Then $F = f_4(x_1,x_0)+x_2x_0(f_2(x_1,x_0)+x_0x_2)$ and 
	$H=g_1(x_1,x_2)$. We will show that the closure of $(C,L)$ contains $(\tilde{C},\tilde{L})$ and $(\overline{C}, \overline{L})$. Let $\gamma=\operatorname{Diag}(1,0,-1)$ be a one-parameter subgroup. Then 
	$$\lim_{t\to 0}\gamma(t) \cdot F = a_1x_1^4+a_2x_1^2x_0x_2+a_3x_0^2x_2^2 \text{ and } \lim_{t\to 0}\gamma(t) \cdot H = x_1.$$
	This implies, that for specific values of $a_1$, $a_2$ and $a_3$ $\lim_{t\to 0}\gamma(t) \cdot (C,L) = (\tilde{C},\tilde{L})$. For general values of the $a_i$, $\lim_{t\to 0}\gamma(t) \cdot (C,L) = (\overline{C}, \overline{L})$. Hence, the closure of the orbit of $(C,L)$ contains both $(\tilde{C},\tilde{L})$ and $(\overline{C}, \overline{L})$. We can do similar analysis for the rest of the cases and end up with $F$ and $H$ not depending on any parameters. Observe that since $(C,L)$ is strictly $t$-semistable, the stabilizer subgroup of $(C,L)$, namely, $G_{(C,L)}\subset \operatorname{SL}(3,\mathbb{C})$ is infinite (c.f. \cite[Remark 8.1 (5)]{Dolgachev_2004}). In particular, there is a $\mathbb{C}^*$-action on $(C,L)$. Proposition \ref{c^* invariant quartics} classifies the singularities of $(C,L)$ uniquely according to their equations. For each $t\in (0,2)$, the proof of Theorem  \ref{main thm in vgit quartics_polystable} follows.
\end{proof}

We also provide a description of strictly semistable orbits.

\begin{theorem}\label{strictly semistable orbits}
	Let $(C_4, L)$ be a pair where $C_4$ is a quartic curve in $\mathbb{P}^2$ and $L$ is a line. Let $(\tilde{C}, \tilde{L})$ be the pair given by equations $\tilde{C}=\{f_4(x_1,x_2)+ x_0x_2(f_2(x_1,x_2)+ x_0x_2)=0\}$ and $\tilde{L}=\{f_1(x_1,x_2)=0\}$. If $t$ is a chamber, then the unique strictly $t$-semistable pair is the pair $(\tilde{C}, \tilde{L})$, such that $\tilde{C}$ has a $\mathbf{D}_4$ singularity, and $L$ is passing through the singularity. If $t$ is a wall, then the pair $(C_4, L)$ is strictly $t$-semistable if and only if it is the pair $(\tilde{C}, \tilde{L})$ or one of the following:
	\begin{enumerate}
		\item $t = \frac{1}{2}$: The pair $(C_4,L)$ where $C_4$ is the unique reducible quartic curve with one singular point which is an $\mathbf{A}_5$ singularity, with $L$ not intersecting the $\mathbf{A}_5$ singularity, or where $C_4$ is non-reduced given by equation $C_4 = \{x_2f_3(x_0,x_1,x_2)\}$ and $C_4\cap L = C_4$. 
		\item $t= \frac{4}{5}$: The pair $(C_4,L)$ where $C_4$ is the  irreducible quartic curve with an $\mathbf{A}_1$ singularity (type I) and $C_4\cap L = 4P$, where $P$ is the singular point, or where $C_4$ is the unique irreducible quartic curve with one singular point which is a $\mathbf{D}_5$ singularity, with $L$ not intersecting the $\mathbf{D}_5$ singularity.
		\item $t= 1$: The pair $(C_4,L)$ where $C_4$ is the  irreducible quartic curve with an $\mathbf{A}_1$ singularity (type II) and $C_4\cap L = 3P+Q$, where $P$ is the singular point and $Q$ is a different point, or where $C_4$ is the unique reducible quartic curve with one singular point which is a $\mathbf{D}_6$ singularity, with $L$ not intersecting the $\mathbf{D}_6$ singularity.
		\item $t = \frac{8}{7}$: The pair $(C_4,L)$ where $C_4$ is smooth (type I) and $C_4\cap L = 4P$ or where $C_4$ is the unique irreducible quartic curve with one singular point which is a $\mathbf{E}_6$ singularity, with $L$ not intersecting the $\mathbf{E}_6$ singularity.
		\item $t =\frac{7}{5}$: The pair $(C_4,L)$ where $C_4$ is smooth (type II) and $C_4\cap L = 3P+Q$ or where $C_4$ is the unique irreducible quartic curve with one singular point which is a $\mathbf{E}_7$ singularity, with $L$ not intersecting the $\mathbf{E}_7$ singularity.
	\end{enumerate}
\end{theorem}
\begin{proof}
	Let $C_4 = \{f_4=0\}$, $L = \{h=0\}$, where $f_4$ is a homogeneous polynomial of degree $4$ and $h$ is a homogeneous polynomial of degree $1$. By \cite[Theorem 3.26, Lemma 3.25]{Pap22} the pair $(C_4, L)$ is not $t$-stable if and only if for any $g \in \operatorname{SL}(3, \mathbb{C})$ the monomials with non-zero coefficients of $(g\cdot f_4, g \cdot L)$ are not contained in a tuple of sets $N^{\ominus}_t(\lambda, x_p)$ which is maximal for every given $t$. In particular, the pair is strictly $t$-semistable if and only if it satisfies the Centroid Criterion (c.f. \cite{GMG18} or \cite{Pap22}).
	
	These maximal sets can be found algorithmically using the computational package \cite{Pap_code} as in the proof of Theorem \ref{main thm in p4 vgit}. The strictly $t$-semistable pairs are obtained from verifying via the Centroid Criterion which of these families provide strictly $t$-semistable pairs. We verify this computationally. The above classification then follows directly from \cite{takahashi_2013} (where type I and type II refers to the specific cases in this classification).
\end{proof}

\section{Explicit description of K-moduli walls for $d\geq 5$}
\label{app:deg5-walls}
For the reader's convenience, in this section we provide an explicit classification of the divisors for each K-moduli wall for degrees $d\geq 5$. The classification follows from Theorems \ref{44}, \ref{45} and \cite{ADL24,ADL21,zha22,PSW23}. For $d = 9$, there are no walls. For $d = 8$ and $\Sigma_8 \simeq \mb{P}^1\times \mb{P}^1$, there is one wall at $c = \frac{1}{4}$, where the divisor $D = 2H$, where $H$ is a smooth $(1,1)$-curve. The rest cases are summed up in the following tables.

Let $P = [1:0:0]$, and $\Sigma_8 \simeq \operatorname{Bl}_P \mb{P}^2$.

\begin{center}
	\renewcommand*{\arraystretch}{1.2}
	\begin{table}[ht]
		\centering
		\begin{tabular}{ |c  |c |c|c|}
			\hline
			wall & curve $C$ in $\mb{P}^2$  & sing. of replaced surfaces    \\ \hline 
			
			$\frac{1}{5} $  &  $x(xz+cy^2)$, $c\in \mb{C}^*$  &    smooth
			\\ \hline
			
			$\frac{1}{4} $  &  $x^2z+cy^3$, $c\in \mb{C}^*$  &   smooth
			\\ \hline     
		\end{tabular}
		\caption{K-moduli walls for pairs $(\Sigma_8,cD)$ }
		\label{Kwalls_d=8}
	\end{table}
\end{center}

Let $P = [0:0:1]$, $Q = [0:1:0]$, and $\Sigma_7 \simeq \operatorname{Bl}_{P,Q} \mb{P}^2$. Let also $X'$ be the del Pezzo surface of degree $7$ with an $A_1$-singularity, obtained by taking weighted blow up of $\mb{P}^2$ at $P$ of weight $(2,1,0)$. The K-moduli walls for Gorenstein del Pezzo surface of degree $7$ are summed below:

\begin{center}
	\renewcommand*{\arraystretch}{1.2}
	\begin{table}[ht!]
		\centering
		\begin{tabular}{ |c  |c |c|}
			\hline
			wall & curve $C$ in $\mb{P}^2$ &sing. of replaced surfaces 
			\\ \hline 
			
			$\frac{4}{25} $  &  $yz(ay+bz)+xf_2(y,z)+x^2f_1(y,z)+x^3$  
			&smooth
			\\ \hline
			
			$\frac{2}{9} $  &  $y^2z+xf_2(y,z)+x^2f_1(y,z)+x^3$  & smooth
			\\ \hline   
			$\frac{2}{9} $  & $ z^2x+zy^2+ z(ax^2+bxy)+g_3(x,y)$ & $\mathbf{A}_1$
			\\ \hline
			
			$\frac{2}{5} $  &  $z^2x+z(ax^2+bxy)+g_3(x,y)$  & $\mathbf{A}_1$
			\\ \hline   
		\end{tabular}
		\caption{K-moduli walls for Gorenstein del Pezzo surfaces of degree $7$}
		\label{Kwalls_d=7}
	\end{table}
\end{center}


			
			

Let $P_1=(1:0:0)$, $P_2=(0:1:0)$ and $P_3=(1:1:0)$. Let $\Sigma_6\simeq \operatorname{Bl}_{P_1,P_2,P_3}\mb{P}^2$. The surface $X_1'$ is the anti-canonical model of the blowup of $\mb{P}^2$ along $P=(1:0:0)$ and along the tangent direction of $\{x=0\}$ at $Q=(0:0:1)$. The surface $X_{1,1}$ is the anti-canonical model of the blowup of $\mb{P}^2$ along $P=(0:1:0)$ and along the tangent direction of $\{z=0\}$ at $Q=(1:0:0)$. The surface $X_{2}$ with one $\mathbf{A}_2$ singularity is obtained by blowing up $\mb{P}^2$ along a 0-dimensional subscheme of length $3$ which is supported at $P = (1:0:0)$ and curvilinear with respect to a conic $\{xz-y^2=0\}$. Finally, the surface $X_{1,2}$ is a toric surface with Picard rank $1$, and hence a weighted projective plane. It is isomorphic to $\mb{P}(1,2,3)_{u,v,w}$. Below, we summarize the description of K-moduli walls for Gorenstein del Pezzo surfaces of degree $6$.

\begin{center}
	\renewcommand*{\arraystretch}{1.2}
	\begin{table}[ht]
		\centering
		\begin{tabular}{ |c  |c |c|c|}
			\hline
			wall & curve $C$ in $\mb{P}^2$ or $\mb{P}(1,1,4)$ & sing. of replaced surfaces   \\ \hline 
			
			$\frac{1}{4} $  & $xy(x-y)+zf_2(x,y)+z^2f_1(x,y)+z^3=0$  & $\mathbf{A}_1$  
			\\ \hline
			$\frac{2}{11} $  & $z^2x+zx^2+zx(ax+by)+f_3(x,y)=0$  &  $\mathbf{A}_1$ 
			\\ \hline
			$\frac{5}{14} $  & $x^2z+xy^2+x(ayz+bz^2)+f_3(y,z)=0$  & $2\mathbf{A}_{1}$ 
			\\ \hline
			$\frac{2}{5} $  & $x^2z+xy^2+x(ayz+bz^2)+f_3(y,z)=0$  & $\mathbf{A}_{2}$  
			\\ \hline
			$\frac{1}{2} $  & $y^3+z^2+xf_5(x,y,z)=0$  & $\mathbf{A}_{1}+\mathbf{A}_2$ 
			\\ \hline 
		\end{tabular}
		\caption{K-moduli walls for Gorenstein del Pezzo surfaces of degree $6$}
		\label{Kwalls_d=6}
	\end{table}
\end{center}

Below, we summarize the description of K-moduli walls for Gorenstein del Pezzo surfaces of degree $5$.

\begin{center}
	\renewcommand*{\arraystretch}{1.2}
	\begin{table}[ht]
		\centering
		\begin{tabular}{ |c  |c | c| }
			\hline
			wall & corresponding cubic curves  & sing. of replaced surfaces  \\ \hline 
			
			$\frac{2}{17} $  & double line + different line  & $\mathbf{A}_1$
			\\ \hline
			
			$\frac{4}{19} $  & three concurrent distinct lines   & 2$\mathbf{A}_1$
			\\ \hline
			$\frac{2}{7} $  & three concurrent distinct lines   & $\mathbf{A}_2$
			\\ \hline
			$\frac{8}{23} $  & conic + tangent line   & $\mathbf{A}_1+\mathbf{A}_2$
			\\ \hline
			$\frac{4}{9} $  & conic + tangent line & $\mathbf{A}_3$
			\\ \hline
			$\frac{4}{7} $  & cuspidal cubic   & $\mathbf{A}_4$
			\\ \hline
			
		\end{tabular}
		\caption{K-moduli walls for Gorenstein del Pezzo surface of degree $5$}
		\label{Kwalls_d=5}
	\end{table}
\end{center}

\printbibliography

@misc{stacks-project,
    shorthand    = {Stacks},
    author       = {The {Stacks Project Authors}},
    title        = {\textit{Stacks Project}},
    howpublished = {\url{https://stacks.math.columbia.edu}},
    year         = {2025},
  }

@article {GMGMS,
    AUTHOR = {Gallardo, Patricio and Martinez-Garcia, Jesus and Moon,
              Han-Bom and Swinarski, David},
     TITLE = {Computation of {GIT} quotients of semisimple groups},
   JOURNAL = {Math. Comp.},
  FJOURNAL = {Mathematics of Computation},
      ISSN = {0025-5718,1088-6842},
      YEAR = {2026},
       DOI = {10.1090/mcom/4152},
       URL = {https://doi.org/10.1090/mcom/4152},
      NOTE = {Published electronically: March 26, 2026},
}

@misc{Hanson-MG,
      title={The CompGIT package: a computational tool for Geometric Invariant Theory quotients}, 
      author={Robert Hanson and Jesus Martinez-Garcia},
      year={2025},
      eprint={2506.19431},
      archivePrefix={arXiv},
      primaryClass={math.AG},
      url={https://arxiv.org/abs/2506.19431}, 
}

@misc{PT08,
      title={CM Stability and the Generalized Futaki Invariant I}, 
      author={Sean T. Paul and Gang Tian},
      year={2008},
      eprint={math/0605278},
      archivePrefix={arXiv},
}

@book {Laza-thesis,
    AUTHOR = {Laza, Radu},
     TITLE = {Deformations of singularities and variations of {GIT}
              quotients},
      NOTE = {Thesis (Ph.D.)--Columbia University},
 PUBLISHER = {ProQuest LLC, Ann Arbor, MI},
      YEAR = {2006},
     PAGES = {286},
      ISBN = {978-0542-64169-5},
   MRCLASS = {99-05},
  MRNUMBER = {2708609},
       URL =
              {http://gateway.proquest.com/openurl?url_ver=Z39.88-2004&rft_val_fmt=info:ofi/fmt:kev:mtx:dissertation&res_dat=xri:pqdiss&rft_dat=xri:pqdiss:3213548},
}

@article {Casagrande-Druel,
    AUTHOR = {Cheltsov, Ivan and Duarte Guerreiro, Tiago and Fujita, Kento
              and Krylov, Igor and Martinez-Garcia, Jesus},
     TITLE = {K-stability of {C}asagrande-{D}ruel varieties},
   JOURNAL = {J. Reine Angew. Math.},
  FJOURNAL = {Journal f\"ur die Reine und Angewandte Mathematik. [Crelle's
              Journal]},
    VOLUME = {818},
      YEAR = {2025},
     PAGES = {53--113},
      ISSN = {0075-4102,1435-5345},
   MRCLASS = {14J45 (14D20 14E30)},
  MRNUMBER = {4846021},
MRREVIEWER = {Artie\ Prendergast-Smith},
       DOI = {10.1515/crelle-2024-0074},
       URL = {https://doi.org/10.1515/crelle-2024-0074},
}

@article {ALF,
    AUTHOR = {Cheltsov, Ivan A. and Rubinstein, Yanir A.},
     TITLE = {Asymptotically log {F}ano varieties},
   JOURNAL = {Adv. Math.},
  FJOURNAL = {Advances in Mathematics},
    VOLUME = {285},
      YEAR = {2015},
     PAGES = {1241--1300},
      ISSN = {0001-8708,1090-2082},
   MRCLASS = {14J45 (32Q20 32Q30)},
  MRNUMBER = {3406526},
MRREVIEWER = {Dimitra\ Kosta},
       DOI = {10.1016/j.aim.2015.08.001},
       URL = {https://doi.org/10.1016/j.aim.2015.08.001},
}

@article {ALF-ample-angles,
    AUTHOR = {Cascini, Paolo and Martinez-Garcia, Jesus and Rubinstein,
              Yanir A.},
     TITLE = {On the body of ample angles of asymptotically log {F}ano
              varieties},
   JOURNAL = {Rend. Circ. Mat. Palermo (2)},
  FJOURNAL = {Rendiconti del Circolo Matematico di Palermo. Second Series},
    VOLUME = {72},
      YEAR = {2023},
    NUMBER = {2},
     PAGES = {773--790},
      ISSN = {0009-725X,1973-4409},
   MRCLASS = {14J45 (14E05 14J10 14J26)},
  MRNUMBER = {4559072},
       DOI = {10.1007/s12215-021-00712-9},
       URL = {https://doi.org/10.1007/s12215-021-00712-9},
}

@misc{Alper-book,
    author = {Alper, Jarod},
    title = {Stacks and moduli},
    note = {https://sites.math.washington.edu/~jarod/moduli.pdf, accessed 17/06/2026.},
}

@article {AGV08,
    AUTHOR = {Abramovich, Dan and Graber, Tom and Vistoli, Angelo},
     TITLE = {Gromov-{W}itten theory of {D}eligne-{M}umford stacks},
   JOURNAL = {Amer. J. Math.},
  FJOURNAL = {American Journal of Mathematics},
    VOLUME = {130},
      YEAR = {2008},
    NUMBER = {5},
     PAGES = {1337--1398},
      ISSN = {0002-9327,1080-6377},
   MRCLASS = {14N35 (14A20 53D45)},
  MRNUMBER = {2450211},
MRREVIEWER = {Johannes\ Walcher},
       DOI = {10.1353/ajm.0.0017},
       URL = {https://doi.org/10.1353/ajm.0.0017},
}

@book {CA23,
    AUTHOR = {Araujo, Carolina and Castravet, Ana-Maria and Cheltsov, Ivan
              and Fujita, Kento and Kaloghiros, Anne-Sophie and
              Martinez-Garcia, Jesus and Shramov, Constantin and S\"u\ss,
              Hendrik and Viswanathan, Nivedita},
     TITLE = {The {C}alabi problem for {F}ano threefolds},
    SERIES = {London Mathematical Society Lecture Note Series},
    VOLUME = {485},
 PUBLISHER = {Cambridge University Press, Cambridge},
      YEAR = {2023},
     PAGES = {vii+441},
      ISBN = {978-1-009-19339-9},
   MRCLASS = {14J45 (32Q15 32Q20)},
  MRNUMBER = {4590444},
}

@book {Kol23,
    AUTHOR = {Koll\'ar, J\'anos},
     TITLE = {Families of varieties of general type},
    SERIES = {Cambridge Tracts in Mathematics},
    VOLUME = {231},
      NOTE = {With the collaboration of Klaus Altmann and S\'andor J.
              Kov\'acs},
 PUBLISHER = {Cambridge University Press, Cambridge},
      YEAR = {2023},
     PAGES = {xviii+471},
      ISBN = {978-1-009-34610-8},
   MRCLASS = {14J10 (14D20 14E30 14J29)},
  MRNUMBER = {4566297},
MRREVIEWER = {Chenyang\ Xu},
}

@book {Xu25,
    AUTHOR = {Xu, Chenyang},
     TITLE = {K-stability of {F}ano varieties},
    SERIES = {New Mathematical Monographs},
    VOLUME = {50},
 PUBLISHER = {Cambridge University Press, Cambridge},
      YEAR = {2025},
     PAGES = {xi+411},
      ISBN = {978-1-009-53877-0},
   MRCLASS = {14J45 (32Q26)},
  MRNUMBER = {4893062},
}

@article {Takayama,
    AUTHOR = {Takayama, Shigeharu},
     TITLE = {Simple connectedness of weak {F}ano varieties},
   JOURNAL = {J. Algebraic Geom.},
  FJOURNAL = {Journal of Algebraic Geometry},
    VOLUME = {9},
      YEAR = {2000},
    NUMBER = {2},
     PAGES = {403--407},
      ISSN = {1056-3911,1534-7486},
   MRCLASS = {14J45 (14F45 32Q55)},
  MRNUMBER = {1735807},
MRREVIEWER = {Jaros\l aw\ A.\ Wi\'sniewski},
}

@article {zha23,
    AUTHOR = {Zhao, Junyan},
     TITLE = {Compactifications of moduli of del {P}ezzo surfaces via line
              arrangement and {K}-stability},
   JOURNAL = {Canad. J. Math.},
  FJOURNAL = {Canadian Journal of Mathematics. Journal Canadien de
              Math\'ematiques},
    VOLUME = {76},
      YEAR = {2024},
    NUMBER = {6},
     PAGES = {2115--2135},
      ISSN = {0008-414X,1496-4279},
   MRCLASS = {14D22 (14D23 14J10 14J17)},
  MRNUMBER = {4824561},
       DOI = {10.4153/S0008414X23000743},
       URL = {https://doi.org/10.4153/S0008414X23000743},
}

@article {Hidaka-Watanabe,
    AUTHOR = {Hidaka, Fumio and Watanabe, Keiichi},
     TITLE = {Normal {G}orenstein surfaces with ample anti-canonical
              divisor},
   JOURNAL = {Tokyo J. Math.},
  FJOURNAL = {Tokyo Journal of Mathematics},
    VOLUME = {4},
      YEAR = {1981},
    NUMBER = {2},
     PAGES = {319--330},
      ISSN = {0387-3870},
   MRCLASS = {14J25 (14J17)},
  MRNUMBER = {646042},
MRREVIEWER = {Mirella\ Manaresi},
       DOI = {10.3836/tjm/1270215157},
       URL = {https://doi.org/10.3836/tjm/1270215157},
}

@article{Pap22,
  title={K-moduli of log {F}ano complete intersections},
  author={Papazachariou, Theodoros Stylianos},
  journal={arXiv preprint arXiv:2212.09332},
  year={2022}
}

@book {Pap23_thesis,
    AUTHOR = {Papazachariou, Theodoros Stylianos},
     TITLE = {K-{M}oduli of {L}og {F}ano {C}omplete {I}ntersections},
      NOTE = {Thesis (Ph.D.)--University of Essex (United Kingdom)},
 PUBLISHER = {ProQuest LLC, Ann Arbor, MI},
      YEAR = {2022},
     PAGES = {248},
   MRCLASS = {99-05},
  MRNUMBER = {5025852},
       URL =
              {https://gateway.proquest.com/openurl?url_ver=Z39.88-2004&rft_val_fmt=info:ofi/fmt:kev:mtx:dissertation&res_dat=xri:pqm&rft_dat=xri:pqdiss:32370610},
}

@article{Kol18,
  title={Mumford divisors},
  author={Koll{\'a}r, J{\'a}nos},
  journal={arXiv preprint arXiv:1803.07596},
  year={2018}
}

@article {Tha96,
    AUTHOR = {Thaddeus, Michael},
     TITLE = {Geometric invariant theory and flips},
   JOURNAL = {J. Amer. Math. Soc.},
  FJOURNAL = {Journal of the American Mathematical Society},
    VOLUME = {9},
      YEAR = {1996},
    NUMBER = {3},
     PAGES = {691--723},
      ISSN = {0894-0347,1088-6834},
   MRCLASS = {14D25 (14L30)},
  MRNUMBER = {1333296},
MRREVIEWER = {I.\ Dolgachev},
       DOI = {10.1090/S0894-0347-96-00204-4},
       URL = {https://doi.org/10.1090/S0894-0347-96-00204-4},
}

@article {DH98,
    AUTHOR = {Dolgachev, Igor V. and Hu, Yi},
     TITLE = {Variation of geometric invariant theory quotients},
      NOTE = {With an appendix by Nicolas Ressayre},
   JOURNAL = {Inst. Hautes \'Etudes Sci. Publ. Math.},
  FJOURNAL = {Institut des Hautes \'Etudes Scientifiques. Publications
              Math\'ematiques},
    NUMBER = {87},
      YEAR = {1998},
     PAGES = {5--56},
      ISSN = {0073-8301,1618-1913},
   MRCLASS = {14L24},
  MRNUMBER = {1659282},
MRREVIEWER = {P.\ E.\ Newstead},
       URL = {http://www.numdam.org/item?id=PMIHES_1998__87__5_0},
}

@article{KM76,
    AUTHOR = {Knudsen, Finn Faye and Mumford, David},
     TITLE = {The projectivity of the moduli space of stable curves. {I}.
              {P}reliminaries on ``det'' and ``{D}iv''},
   JOURNAL = {Math. Scand.},
  FJOURNAL = {Mathematica Scandinavica},
    VOLUME = {39},
      YEAR = {1976},
    NUMBER = {1},
     PAGES = {19--55},
      ISSN = {0025-5521,1903-1807},
   MRCLASS = {14H10 (14C05 14F05)},
  MRNUMBER = {437541},
MRREVIEWER = {P.\ E.\ Newstead},
       DOI = {10.7146/math.scand.a-11642},
       URL = {https://doi.org/10.7146/math.scand.a-11642},
}

@article {PSW23,
    AUTHOR = {Pan, Long and Si, Fei and Wu, Haoyu},
     TITLE = {{$K$}-moduli spaces of log del {P}ezzo pairs},
   JOURNAL = {Adv. Math.},
  FJOURNAL = {Advances in Mathematics},
    VOLUME = {481},
      YEAR = {2025},
     PAGES = {Paper No. 110536, 55},
      ISSN = {0001-8708,1090-2082},
   MRCLASS = {14D23 (14E05 14J10 14J28 14J45)},
  MRNUMBER = {4957550},
MRREVIEWER = {Jingshan\ Chen},
       DOI = {10.1016/j.aim.2025.110536},
       URL = {https://doi.org/10.1016/j.aim.2025.110536},
}

@article{ADL22,
    AUTHOR = {Ascher, Kenneth and DeVleming, Kristin and Liu, Yuchen},
     TITLE = {K-stability and birational models of moduli of quartic {K}3
              surfaces},
   JOURNAL = {Invent. Math.},
  FJOURNAL = {Inventiones Mathematicae},
    VOLUME = {232},
      YEAR = {2023},
    NUMBER = {2},
     PAGES = {471--552},
      ISSN = {0020-9910,1432-1297},
   MRCLASS = {14J28 (14J10 14J45)},
  MRNUMBER = {4574660},
MRREVIEWER = {Guolei\ Zhong},
       DOI = {10.1007/s00222-022-01170-5},
       URL = {https://doi.org/10.1007/s00222-022-01170-5},
}

@article {BHLLX21,
    AUTHOR = {Blum, Harold and Halpern-Leistner, Daniel and Liu, Yuchen and
              Xu, Chenyang},
     TITLE = {On properness of {K}-moduli spaces and optimal degenerations
              of {F}ano varieties},
   JOURNAL = {Selecta Math. (N.S.)},
  FJOURNAL = {Selecta Mathematica. New Series},
    VOLUME = {27},
      YEAR = {2021},
    NUMBER = {4},
     PAGES = {Paper No. 73, 39},
      ISSN = {1022-1824,1420-9020},
   MRCLASS = {14J45 (14D23)},
  MRNUMBER = {4292783},
MRREVIEWER = {Kenta\ Hashizume},
       DOI = {10.1007/s00029-021-00694-7},
       URL = {https://doi.org/10.1007/s00029-021-00694-7},
}

@article{Laz13,
  title={{GIT} and moduli with a twist},
  author={Laza, Radu},
  journal={Handbook of moduli. Vol. II, Adv. Lect. Math.},
  volume={25},
  pages={259--297},
  year={2013},
  publisher={Int. Press, Somerville, MA}
}

@article{CP20,
    AUTHOR = {Cheltsov, Ivan and Prokhorov, Yuri},
     TITLE = {Del {P}ezzo surfaces with infinite automorphism groups},
   JOURNAL = {Algebr. Geom.},
  FJOURNAL = {Algebraic Geometry},
    VOLUME = {8},
      YEAR = {2021},
    NUMBER = {3},
     PAGES = {319--357},
      ISSN = {2313-1691,2214-2584},
   MRCLASS = {14J50 (14J17 14J26)},
  MRNUMBER = {4206439},
MRREVIEWER = {Tatiana\ M.\ Bandman},
       DOI = {10.14231/ag-2021-008},
       URL = {https://doi.org/10.14231/ag-2021-008},
}

@article {Kir85,
    AUTHOR = {Kirwan, Frances Clare},
     TITLE = {Partial desingularisations of quotients of nonsingular
              varieties and their {B}etti numbers},
   JOURNAL = {Ann. of Math. (2)},
  FJOURNAL = {Annals of Mathematics. Second Series},
    VOLUME = {122},
      YEAR = {1985},
    NUMBER = {1},
     PAGES = {41--85},
      ISSN = {0003-486X,1939-8980},
   MRCLASS = {14D25 (14L30 20G99 32M99)},
  MRNUMBER = {799252},
MRREVIEWER = {Andrew\ J.\ Sommese},
       DOI = {10.2307/1971369},
       URL = {https://doi.org/10.2307/1971369},
}

@article {Ben14,
    AUTHOR = {Benoist, Olivier},
     TITLE = {Quelques espaces de modules d'intersections compl\`etes lisses
              qui sont quasi-projectifs},
   JOURNAL = {J. Eur. Math. Soc. (JEMS)},
  FJOURNAL = {Journal of the European Mathematical Society (JEMS)},
    VOLUME = {16},
      YEAR = {2014},
    NUMBER = {8},
     PAGES = {1749--1774},
      ISSN = {1435-9855,1435-9863},
   MRCLASS = {14D22 (14L24 14M10)},
  MRNUMBER = {3262456},
MRREVIEWER = {Dan\ Abramovich},
       DOI = {10.4171/JEMS/474},
       URL = {https://doi.org/10.4171/JEMS/474},
}

@article {HP10,
    AUTHOR = {Hacking, Paul and Prokhorov, Yuri},
     TITLE = {Smoothable del {P}ezzo surfaces with quotient singularities},
   JOURNAL = {Compos. Math.},
  FJOURNAL = {Compositio Mathematica},
    VOLUME = {146},
      YEAR = {2010},
    NUMBER = {1},
     PAGES = {169--192},
      ISSN = {0010-437X,1570-5846},
   MRCLASS = {14J26 (14B05 14B07 14E30 14J30)},
  MRNUMBER = {2581246},
MRREVIEWER = {Mark\ Gross},
       DOI = {10.1112/S0010437X09004370},
       URL = {https://doi.org/10.1112/S0010437X09004370},
}

@article{zha23b,
  title={The Moduli Space of Genus Six Curves and {K}-stability: {VGIT} and the {H}assett-{K}eel Program},
  author={Zhao, Junyan},
  journal={arXiv preprint arXiv:2304.13259},
  year={2023}
}

@article{Hac04,
  title={Compact moduli of plane curves},
  author={Hacking, Paul},
  journal={Duke Math. J.},
  volume={124},
  number={2},
  pages={213--257},
  year={2004}
}

@article {Alp13,
    AUTHOR = {Alper, Jarod},
     TITLE = {Good moduli spaces for {A}rtin stacks},
   JOURNAL = {Ann. Inst. Fourier (Grenoble)},
  FJOURNAL = {Universit\'e{} de Grenoble. Annales de l'Institut Fourier},
    VOLUME = {63},
      YEAR = {2013},
    NUMBER = {6},
     PAGES = {2349--2402},
      ISSN = {0373-0956,1777-5310},
   MRCLASS = {14D23 (14L24 14L30)},
  MRNUMBER = {3237451},
MRREVIEWER = {Arvid\ Siqveland},
       DOI = {10.5802/aif.2833},
       URL = {https://doi.org/10.5802/aif.2833},
}

@article {LX14,
    AUTHOR = {Li, Chi and Xu, Chenyang},
     TITLE = {Special test configuration and {K}-stability of {F}ano
              varieties},
   JOURNAL = {Ann. of Math. (2)},
  FJOURNAL = {Annals of Mathematics. Second Series},
    VOLUME = {180},
      YEAR = {2014},
    NUMBER = {1},
     PAGES = {197--232},
      ISSN = {0003-486X,1939-8980},
   MRCLASS = {14J45 (14E30 14J10 14J80)},
  MRNUMBER = {3194814},
MRREVIEWER = {Anne-Sophie\ Kaloghiros},
       DOI = {10.4007/annals.2014.180.1.4},
       URL = {https://doi.org/10.4007/annals.2014.180.1.4},
}

@article {Kem78,
    AUTHOR = {Kempf, George R.},
     TITLE = {Instability in invariant theory},
   JOURNAL = {Ann. of Math. (2)},
  FJOURNAL = {Annals of Mathematics. Second Series},
    VOLUME = {108},
      YEAR = {1978},
    NUMBER = {2},
     PAGES = {299--316},
      ISSN = {0003-486X},
   MRCLASS = {20G05 (14L30 14M99 15A72)},
  MRNUMBER = {506989},
MRREVIEWER = {Vladimir\ L.\ Popov},
       DOI = {10.2307/1971168},
       URL = {https://doi.org/10.2307/1971168},
}

@incollection {MG25,
    AUTHOR = {Martinez-Garcia, Jesus},
     TITLE = {Two remarks on asymptotically log {F}ano pairs},
 BOOKTITLE = {{ZAG} handbook of algebraic geometry},
    EDITOR = {Cheltsov, Ivan and Martinez-Garcia, Jesus},
     PAGES = {488--492},
 PUBLISHER = {Chapman \& Hall/CRC},
   ADDRESS = {Boca Raton, FL},
      YEAR = {2025},
      ISBN = {9781032260464},
      NOTE = {arXiv:2409.18547},
}

@article {GMG21,
    AUTHOR = {Gallardo, Patricio and Martinez-Garcia, Jesus and Spotti,
              Cristiano},
     TITLE = {Applications of the moduli continuity method to log {K}-stable
              pairs},
   JOURNAL = {J. Lond. Math. Soc. (2)},
  FJOURNAL = {Journal of the London Mathematical Society. Second Series},
    VOLUME = {103},
      YEAR = {2021},
    NUMBER = {2},
     PAGES = {729--759},
      ISSN = {0024-6107,1469-7750},
   MRCLASS = {32Q20 (14D22 14J10 14J45 14L24)},
  MRNUMBER = {4230917},
MRREVIEWER = {P.\ E.\ Newstead},
       DOI = {10.1112/jlms.12390},
       URL = {https://doi.org/10.1112/jlms.12390},
}

@article {GMG19,
    AUTHOR = {Gallardo, Patricio and Martinez-Garcia, Jesus},
     TITLE = {Moduli of cubic surfaces and their anticanonical divisors},
   JOURNAL = {Rev. Mat. Complut.},
  FJOURNAL = {Revista Matem\'atica Complutense},
    VOLUME = {32},
      YEAR = {2019},
    NUMBER = {3},
     PAGES = {853--873},
      ISSN = {1139-1138,1988-2807},
   MRCLASS = {14J10 (14J26 14J45 14L24 14Q10)},
  MRNUMBER = {3995433},
MRREVIEWER = {Fyodor\ L.\ Zak},
       DOI = {10.1007/s13163-019-00298-y},
       URL = {https://doi.org/10.1007/s13163-019-00298-y},
}

@article {GMG18,
    AUTHOR = {Gallardo, Patricio and Martinez-Garcia, Jesus},
     TITLE = {Variations of geometric invariant quotients for pairs, a
              computational approach},
   JOURNAL = {Proc. Amer. Math. Soc.},
  FJOURNAL = {Proceedings of the American Mathematical Society},
    VOLUME = {146},
      YEAR = {2018},
    NUMBER = {6},
     PAGES = {2395--2408},
      ISSN = {0002-9939,1088-6826},
   MRCLASS = {14L24 (14H10 14J32 14J45 14Q10)},
  MRNUMBER = {3778143},
MRREVIEWER = {Gergely\ B\'erczi},
       DOI = {10.1090/proc/13950},
       URL = {https://doi.org/10.1090/proc/13950},
}

@article {BLX22,
    AUTHOR = {Blum, Harold and Liu, Yuchen and Xu, Chenyang},
     TITLE = {Openness of {K}-semistability for {F}ano varieties},
   JOURNAL = {Duke Math. J.},
  FJOURNAL = {Duke Mathematical Journal},
    VOLUME = {171},
      YEAR = {2022},
    NUMBER = {13},
     PAGES = {2753--2797},
      ISSN = {0012-7094,1547-7398},
   MRCLASS = {14J45 (14E30)},
  MRNUMBER = {4505846},
MRREVIEWER = {Jihun\ Park},
       DOI = {10.1215/00127094-2022-0054},
       URL = {https://doi.org/10.1215/00127094-2022-0054},
}

@article {OSS16,
    AUTHOR = {Odaka, Yuji and Spotti, Cristiano and Sun, Song},
     TITLE = {Compact moduli spaces of del {P}ezzo surfaces and
              {K}\"ahler-{E}instein metrics},
   JOURNAL = {J. Differential Geom.},
  FJOURNAL = {Journal of Differential Geometry},
    VOLUME = {102},
      YEAR = {2016},
    NUMBER = {1},
     PAGES = {127--172},
      ISSN = {0022-040X,1945-743X},
   MRCLASS = {14J15 (14J45 32Q25 53C25 53C55 58D27)},
  MRNUMBER = {3447088},
MRREVIEWER = {I.\ Dolgachev},
       URL = {http://projecteuclid.org/euclid.jdg/1452002879},
}

@book{KM98,
    AUTHOR = {Koll\'ar, J\'anos and Mori, Shigefumi},
     TITLE = {Birational geometry of algebraic varieties},
    SERIES = {Cambridge Tracts in Mathematics},
    VOLUME = {134},
 PUBLISHER = {Cambridge University Press, Cambridge},
      YEAR = {1998},
     PAGES = {viii+254},
      ISBN = {0-521-63277-3},
   MRCLASS = {14E30},
  MRNUMBER = {1658959},
MRREVIEWER = {Mark\ Gross},
       DOI = {10.1017/CBO9780511662560},
       URL = {https://doi.org/10.1017/CBO9780511662560},
}

@article{KLPZ26,
  title={The boundary of K-moduli of prime Fano threefolds of genus twelve},
  author={Kaloghiros, Anne-Sophie and Liu, Yuchen and Petracci, Andrea and Zhao, Junyan},
  journal={arXiv preprint arXiv:2603.29827},
  year={2026}
}

@article {Zha26,
    AUTHOR = {Zhao, Junyan},
     TITLE = {K-stability of {T}haddeus' moduli of stable bundle pairs on
              genus two curves},
      NOTE = {With an appendix by Benjamin Church and Junyan Zhao},
   JOURNAL = {Forum Math. Sigma},
  FJOURNAL = {Forum of Mathematics. Sigma},
    VOLUME = {14},
      YEAR = {2026},
     PAGES = {Paper No. e33, 35},
      ISSN = {2050-5094},
   MRCLASS = {14J45 (14D23 14H60 14L24)},
  MRNUMBER = {5036746},
       DOI = {10.1017/fms.2026.10177},
       URL = {https://doi.org/10.1017/fms.2026.10177},
}

@article {3fold-dim1-Kmoduli,
    AUTHOR = {Abban, Hamid and Cheltsov, Ivan and Denisova, Elena and
              Etxabarri-Alberdi, Erroxe and Kaloghiros, Anne-Sophie and
              Jiao, Dongchen and Martinez-Garcia, Jesus and Papazachariou,
              Theodoros},
     TITLE = {One-dimensional components in the {K}-moduli of smooth {F}ano
              3-folds},
   JOURNAL = {J. Algebraic Geom.},
  FJOURNAL = {Journal of Algebraic Geometry},
    VOLUME = {34},
      YEAR = {2025},
    NUMBER = {3},
     PAGES = {489--534},
      ISSN = {1056-3911,1534-7486},
   MRCLASS = {14J45 (14J30)},
  MRNUMBER = {4895152},
MRREVIEWER = {Francesco\ Bottacin},
       DOI = {10.1090/jag/839},
       URL = {https://doi.org/10.1090/jag/839},
}

@incollection {Fuj90,
    AUTHOR = {Fujita, Takao},
     TITLE = {On singular del {P}ezzo varieties},
 BOOKTITLE = {Algebraic geometry ({L}'{A}quila, 1988)},
    SERIES = {Lecture Notes in Math.},
    VOLUME = {1417},
     PAGES = {117--128},
 PUBLISHER = {Springer, Berlin},
      YEAR = {1990},
      ISBN = {3-540-52217-4},
   MRCLASS = {14J40 (14B05 14J35)},
  MRNUMBER = {1040555},
MRREVIEWER = {L.\ B\u adescu},
       DOI = {10.1007/BFb0083337},
       URL = {https://doi.org/10.1007/BFb0083337},
}

@book{MFK94,
    AUTHOR = {Mumford, D. and Fogarty, J. and Kirwan, F.},
     TITLE = {Geometric invariant theory},
    SERIES = {Ergebnisse der Mathematik und ihrer Grenzgebiete (2) [Results
              in Mathematics and Related Areas (2)]},
    VOLUME = {34},
   EDITION = {Third},
 PUBLISHER = {Springer-Verlag, Berlin},
      YEAR = {1994},
     PAGES = {xiv+292},
      ISBN = {3-540-56963-4},
   MRCLASS = {14D25 (58E05 58F05)},
  MRNUMBER = {1304906},
MRREVIEWER = {Yi\ Hu},
}

@incollection{MM20,
  title={Stability and {E}instein-{K}{\"a}hler metric of a quartic del {P}ezzo surface},
  author={Mabuchi, Toshiki and Mukai, Shigeru},
  booktitle={Einstein metrics and Yang-Mills connections},
  pages={133--160},
  year={1993},
  publisher={CRC Press}
}

@misc{GMG-code,
    title={Variations of GIT quotients package v0. 6.13},
    author={Gallardo, Patricio, and Martinez-Garcia, Jesus},
    year={2017},
    publisher={University of Bath},
    title = {Variations of GIT quotients package v0.6.13},
    howpublished = {\url{https://doi.org/10.15125/BATH-00458}},
    url = {https://doi.org/10.15125/BATH-00458},
    DOI={10.15125/BATH-00458},
}

@article {LL19,
    AUTHOR = {Li, Chi and Liu, Yuchen},
     TITLE = {K\"ahler-{E}instein metrics and volume minimization},
   JOURNAL = {Adv. Math.},
  FJOURNAL = {Advances in Mathematics},
    VOLUME = {341},
      YEAR = {2019},
     PAGES = {440--492},
      ISSN = {0001-8708,1090-2082},
   MRCLASS = {32Q20 (53C25 53C55)},
  MRNUMBER = {3872852},
MRREVIEWER = {Aijin\ Lin},
       DOI = {10.1016/j.aim.2018.10.038},
       URL = {https://doi.org/10.1016/j.aim.2018.10.038},
}

@article {CP21,
    AUTHOR = {Codogni, Giulio and Patakfalvi, Zsolt},
     TITLE = {Positivity of the {CM} line bundle for families of
              {$K$}-stable klt {F}ano varieties},
   JOURNAL = {Invent. Math.},
  FJOURNAL = {Inventiones Mathematicae},
    VOLUME = {223},
      YEAR = {2021},
    NUMBER = {3},
     PAGES = {811--894},
      ISSN = {0020-9910,1432-1297},
   MRCLASS = {14J45 (14C20)},
  MRNUMBER = {4213768},
MRREVIEWER = {James\ McKernan},
       DOI = {10.1007/s00222-020-00999-y},
       URL = {https://doi.org/10.1007/s00222-020-00999-y},
}

@article {LX19,
    AUTHOR = {Liu, Yuchen and Xu, Chenyang},
     TITLE = {K-stability of cubic threefolds},
   JOURNAL = {Duke Math. J.},
  FJOURNAL = {Duke Mathematical Journal},
    VOLUME = {168},
      YEAR = {2019},
    NUMBER = {11},
     PAGES = {2029--2073},
      ISSN = {0012-7094,1547-7398},
   MRCLASS = {14L24 (14E30 14J30 32Q20)},
  MRNUMBER = {3992032},
MRREVIEWER = {Yuji\ Odaka},
       DOI = {10.1215/00127094-2019-0006},
       URL = {https://doi.org/10.1215/00127094-2019-0006},
}

@article {Xu20,
    AUTHOR = {Xu, Chenyang},
     TITLE = {A minimizing valuation is quasi-monomial},
   JOURNAL = {Ann. of Math. (2)},
  FJOURNAL = {Annals of Mathematics. Second Series},
    VOLUME = {191},
      YEAR = {2020},
    NUMBER = {3},
     PAGES = {1003--1030},
      ISSN = {0003-486X,1939-8980},
   MRCLASS = {14E30 (14J17 14J45)},
  MRNUMBER = {4088355},
MRREVIEWER = {Yuchen\ Liu},
       DOI = {10.4007/annals.2020.191.3.6},
       URL = {https://doi.org/10.4007/annals.2020.191.3.6},
}

@article {Zho23,
    AUTHOR = {Zhou, Chuyu},
     TITLE = {On wall-crossing for {K}-stability},
   JOURNAL = {Adv. Math.},
  FJOURNAL = {Advances in Mathematics},
    VOLUME = {413},
      YEAR = {2023},
     PAGES = {Paper No. 108857, 26},
      ISSN = {0001-8708,1090-2082},
   MRCLASS = {14D23 (14E30 14J45)},
  MRNUMBER = {4533746},
MRREVIEWER = {Nathan\ Grieve},
       DOI = {10.1016/j.aim.2022.108857},
       URL = {https://doi.org/10.1016/j.aim.2022.108857},
}

@article {ADL24,
    AUTHOR = {Ascher, Kenneth and DeVleming, Kristin and Liu, Yuchen},
     TITLE = {Wall crossing for {K}-moduli spaces of plane curves},
   JOURNAL = {Proc. Lond. Math. Soc. (3)},
  FJOURNAL = {Proceedings of the London Mathematical Society. Third Series},
    VOLUME = {128},
      YEAR = {2024},
    NUMBER = {6},
     PAGES = {Paper No. e12615, 113},
      ISSN = {0024-6115,1460-244X},
   MRCLASS = {14H10 (14H50 14J10 14J45 14L24)},
  MRNUMBER = {4757702},
MRREVIEWER = {Luca\ Schaffler},
       DOI = {10.1112/plms.12615},
       URL = {https://doi.org/10.1112/plms.12615},
}

@article{ADL21,
    AUTHOR = {Ascher, Kenneth and DeVleming, Kristin and Liu, Yuchen},
     TITLE = {K-moduli of curves on a quadric surface and {K}3 surfaces},
   JOURNAL = {J. Inst. Math. Jussieu},
  FJOURNAL = {Journal of the Institute of Mathematics of Jussieu. JIMJ.
              Journal de l'Institut de Math\'{e}matiques de Jussieu},
    VOLUME = {22},
      YEAR = {2023},
    NUMBER = {3},
     PAGES = {1251--1291},
      ISSN = {1474-7480,1475-3030},
   MRCLASS = {14J10 (14D23 14J28)},
  MRNUMBER = {4574172},
       DOI = {10.1017/S1474748021000384},
       URL = {https://doi.org/10.1017/S1474748021000384},
}

@article{Nak07,
  title={Classification of log del {P}ezzo surfaces of index two},
  author={Nakayama, Noboru},
  journal={J. Math. Sci. Univ. Tokyo},
  volume={14},
  number={3},
  pages={293--498},
  year={2007}
}

@article{Fuj19,
    AUTHOR = {Fujita, Kento},
     TITLE = {A valuative criterion for uniform {K}-stability of {$\mathbb
              Q$}-{F}ano varieties},
   JOURNAL = {J. Reine Angew. Math.},
  FJOURNAL = {Journal f\"ur die Reine und Angewandte Mathematik. [Crelle's
              Journal]},
    VOLUME = {751},
      YEAR = {2019},
     PAGES = {309--338},
      ISSN = {0075-4102,1435-5345},
   MRCLASS = {14E05 (14J45 32Q26)},
  MRNUMBER = {3956698},
MRREVIEWER = {Ziquan\ Zhuang},
       DOI = {10.1515/crelle-2016-0055},
       URL = {https://doi.org/10.1515/crelle-2016-0055},
}

@article{SSY,
    AUTHOR = {Spotti, Cristiano and Sun, Song and Yao, Chengjian},
     TITLE = {Existence and deformations of {K}\"{a}hler-{E}instein metrics
              on smoothable {$\mathbb{Q}$}-{F}ano varieties},
   JOURNAL = {Duke Math. J.},
  FJOURNAL = {Duke Mathematical Journal},
    VOLUME = {165},
      YEAR = {2016},
    NUMBER = {16},
     PAGES = {3043--3083},
      ISSN = {0012-7094,1547-7398},
   MRCLASS = {53C55 (14J10 14J45 32Q20 53C25)},
  MRNUMBER = {3566198},
MRREVIEWER = {G.\ K.\ Sankaran},
       DOI = {10.1215/00127094-3645330},
       URL = {https://doi.org/10.1215/00127094-3645330},
}

@article {Li17,
    AUTHOR = {Li, Chi},
     TITLE = {K-semistability is equivariant volume minimization},
   JOURNAL = {Duke Math. J.},
  FJOURNAL = {Duke Mathematical Journal},
    VOLUME = {166},
      YEAR = {2017},
    NUMBER = {16},
     PAGES = {3147--3218},
      ISSN = {0012-7094,1547-7398},
   MRCLASS = {14B05 (13A18 14J45 52A27 53C25 53C55)},
  MRNUMBER = {3715806},
MRREVIEWER = {Ruadha\'i\ Dervan},
       DOI = {10.1215/00127094-2017-0026},
       URL = {https://doi.org/10.1215/00127094-2017-0026},
}

@article {XZ21,
    AUTHOR = {Xu, Chenyang and Zhuang, Ziquan},
     TITLE = {Uniqueness of the minimizer of the normalized volume function},
   JOURNAL = {Camb. J. Math.},
  FJOURNAL = {Cambridge Journal of Mathematics},
    VOLUME = {9},
      YEAR = {2021},
    NUMBER = {1},
     PAGES = {149--176},
      ISSN = {2168-0930,2168-0949},
   MRCLASS = {14B05 (13A18 14E30)},
  MRNUMBER = {4325260},
MRREVIEWER = {Yuchen\ Liu},
       DOI = {10.4310/CJM.2021.v9.n1.a2},
       URL = {https://doi.org/10.4310/CJM.2021.v9.n1.a2},
}

@article {BX19,
    AUTHOR = {Blum, Harold and Xu, Chenyang},
     TITLE = {Uniqueness of {K}-polystable degenerations of {F}ano
              varieties},
   JOURNAL = {Ann. of Math. (2)},
  FJOURNAL = {Annals of Mathematics. Second Series},
    VOLUME = {190},
      YEAR = {2019},
    NUMBER = {2},
     PAGES = {609--656},
      ISSN = {0003-486X,1939-8980},
   MRCLASS = {14J45 (14D20 14E30)},
  MRNUMBER = {3997130},
MRREVIEWER = {James\ McKernan},
       DOI = {10.4007/annals.2019.190.2.4},
       URL = {https://doi.org/10.4007/annals.2019.190.2.4},
}

@article {Liu18,
    AUTHOR = {Liu, Yuchen},
     TITLE = {The volume of singular {K}\"ahler-{E}instein {F}ano varieties},
   JOURNAL = {Compos. Math.},
  FJOURNAL = {Compositio Mathematica},
    VOLUME = {154},
      YEAR = {2018},
    NUMBER = {6},
     PAGES = {1131--1158},
      ISSN = {0010-437X,1570-5846},
   MRCLASS = {14J45 (14B05 14L24 32Q15)},
  MRNUMBER = {3797604},
MRREVIEWER = {Giulio\ Codogni},
       DOI = {10.1112/S0010437X18007042},
       URL = {https://doi.org/10.1112/S0010437X18007042},
}

@article {Fuj18,
    AUTHOR = {Fujita, Kento},
     TITLE = {Optimal bounds for the volumes of {K}\"ahler-{E}instein {F}ano
              manifolds},
   JOURNAL = {Amer. J. Math.},
  FJOURNAL = {American Journal of Mathematics},
    VOLUME = {140},
      YEAR = {2018},
    NUMBER = {2},
     PAGES = {391--414},
      ISSN = {0002-9327,1080-6377},
   MRCLASS = {32Q25 (14J45 53C25 53C55)},
  MRNUMBER = {3783213},
MRREVIEWER = {Cristiano\ Spotti},
       DOI = {10.1353/ajm.2018.0009},
       URL = {https://doi.org/10.1353/ajm.2018.0009},
}

@article {Der16,
    AUTHOR = {Dervan, Ruadha\'i},
     TITLE = {On {K}-stability of finite covers},
   JOURNAL = {Bull. Lond. Math. Soc.},
  FJOURNAL = {Bulletin of the London Mathematical Society},
    VOLUME = {48},
      YEAR = {2016},
    NUMBER = {4},
     PAGES = {717--728},
      ISSN = {0024-6093,1469-2120},
   MRCLASS = {32Q26 (32Q20)},
  MRNUMBER = {3532146},
MRREVIEWER = {Giulio\ Codogni},
       DOI = {10.1112/blms/bdw029},
       URL = {https://doi.org/10.1112/blms/bdw029},
}

@article {KSB88,
    AUTHOR = {Koll\'ar, J. and Shepherd-Barron, N. I.},
     TITLE = {Threefolds and deformations of surface singularities},
   JOURNAL = {Invent. Math.},
  FJOURNAL = {Inventiones Mathematicae},
    VOLUME = {91},
      YEAR = {1988},
    NUMBER = {2},
     PAGES = {299--338},
      ISSN = {0020-9910,1432-1297},
   MRCLASS = {14J10 (14D20 14J30 32G10 32G13)},
  MRNUMBER = {922803},
MRREVIEWER = {Yujiro\ Kawamata},
       DOI = {10.1007/BF01389370},
       URL = {https://doi.org/10.1007/BF01389370},
}

@article {Laz09,
    AUTHOR = {Laza, Radu},
     TITLE = {Deformations of singularities and variation of {GIT}
              quotients},
   JOURNAL = {Trans. Amer. Math. Soc.},
  FJOURNAL = {Transactions of the American Mathematical Society},
    VOLUME = {361},
      YEAR = {2009},
    NUMBER = {4},
     PAGES = {2109--2161},
      ISSN = {0002-9947,1088-6850},
   MRCLASS = {14B07 (14J17 14L24)},
  MRNUMBER = {2465831},
MRREVIEWER = {Arvid\ Siqveland},
       DOI = {10.1090/S0002-9947-08-04660-6},
       URL = {https://doi.org/10.1090/S0002-9947-08-04660-6},
}

@article{Xu21,
  title={K-stability of {F}ano varieties: an algebro-geometric approach},
  author={Xu, Chenyang},
  journal={EMS Surv. Math. Sci.},
  volume={8},
  number={1},
  pages={265-354},
  year={2021},
}

@article{Jia20,
  title={Boundedness of $\mb{Q}$-{F}ano varieties with degrees and alpha-invariants bounded from below},
  author={Chen Jiang},
  journal={Ann. Sci. Éc. Norm. Supér (4)},
  volume={53},
  number={5},
  pages={1235--1248},
  year={2020},
}

@article {LXZ22,
    AUTHOR = {Liu, Yuchen and Xu, Chenyang and Zhuang, Ziquan},
     TITLE = {Finite generation for valuations computing stability
              thresholds and applications to {K}-stability},
   JOURNAL = {Ann. of Math. (2)},
  FJOURNAL = {Annals of Mathematics. Second Series},
    VOLUME = {196},
      YEAR = {2022},
    NUMBER = {2},
     PAGES = {507--566},
      ISSN = {0003-486X,1939-8980},
   MRCLASS = {14J45 (14D20 14E30 32Q20)},
  MRNUMBER = {4445441},
MRREVIEWER = {James\ McKernan},
       DOI = {10.4007/annals.2022.196.2.2},
       URL = {https://doi.org/10.4007/annals.2022.196.2.2},
}

@article {LWX21,
    AUTHOR = {Li, Chi and Wang, Xiaowei and Xu, Chenyang},
     TITLE = {Algebraicity of the metric tangent cones and equivariant
              {K}-stability},
   JOURNAL = {J. Amer. Math. Soc.},
  FJOURNAL = {Journal of the American Mathematical Society},
    VOLUME = {34},
      YEAR = {2021},
    NUMBER = {4},
     PAGES = {1175--1214},
      ISSN = {0894-0347,1088-6834},
   MRCLASS = {14B07 (14E30 14J17 14J45 53C55)},
  MRNUMBER = {4301561},
MRREVIEWER = {Carlos\ Galindo},
       DOI = {10.1090/jams/974},
       URL = {https://doi.org/10.1090/jams/974},
}

@article {ABHLX20,
    AUTHOR = {Alper, Jarod and Blum, Harold and Halpern-Leistner, Daniel and
              Xu, Chenyang},
     TITLE = {Reductivity of the automorphism group of {$K$}-polystable
              {F}ano varieties},
   JOURNAL = {Invent. Math.},
  FJOURNAL = {Inventiones Mathematicae},
    VOLUME = {222},
      YEAR = {2020},
    NUMBER = {3},
     PAGES = {995--1032},
      ISSN = {0020-9910,1432-1297},
   MRCLASS = {14J45 (14D23 14E30 14J10)},
  MRNUMBER = {4169054},
MRREVIEWER = {James\ McKernan},
       DOI = {10.1007/s00222-020-00987-2},
       URL = {https://doi.org/10.1007/s00222-020-00987-2},
}

@book {Kol96,
    AUTHOR = {Koll\'{a}r, J\'{a}nos},
     TITLE = {Rational curves on algebraic varieties},
    SERIES = {Ergebnisse der Mathematik und ihrer Grenzgebiete. 3. Folge. A
              Series of Modern Surveys in Mathematics [Results in
              Mathematics and Related Areas. 3rd Series. A Series of Modern
              Surveys in Mathematics]},
    VOLUME = {32},
 PUBLISHER = {Springer-Verlag, Berlin},
      YEAR = {1996},
     PAGES = {viii+320},
      ISBN = {3-540-60168-6},
   MRCLASS = {14-02 (14C05 14E05 14F17 14J45)},
  MRNUMBER = {1440180},
MRREVIEWER = {Yuri\ G.\ Prokhorov},
       DOI = {10.1007/978-3-662-03276-3},
       URL = {https://doi.org/10.1007/978-3-662-03276-3},
}

@article {XZ20,
    AUTHOR = {Xu, Chenyang and Zhuang, Ziquan},
     TITLE = {On positivity of the {CM} line bundle on {K}-moduli spaces},
   JOURNAL = {Ann. of Math. (2)},
  FJOURNAL = {Annals of Mathematics. Second Series},
    VOLUME = {192},
      YEAR = {2020},
    NUMBER = {3},
     PAGES = {1005--1068},
      ISSN = {0003-486X,1939-8980},
   MRCLASS = {14J45 (14D20 14E30)},
  MRNUMBER = {4172625},
MRREVIEWER = {Kenta\ Hashizume},
       DOI = {10.4007/annals.2020.192.3.7},
       URL = {https://doi.org/10.4007/annals.2020.192.3.7},
}

@book {MG13,
    AUTHOR = {Martinez Garcia, Jesus},
     TITLE = {Dynamic alpha-invariants of del {P}ezzo surfaces with
              boundary},
      NOTE = {Thesis (Ph.D.)--The University of Edinburgh (United Kingdom)},
      YEAR = {2013},
     PAGES = {1},
   MRCLASS = {99-05},
  MRNUMBER = {3389398},
       URL =
              {http://gateway.proquest.com/openurl?url_ver=Z39.88-2004&rft_val_fmt=info:ofi/fmt:kev:mtx:dissertation&res_dat=xri:pqm&rft_dat=xri:pqdiss:U637064},
}

@article {CMG16,
    AUTHOR = {Cheltsov, Ivan and Martinez-Garcia, Jesus},
     TITLE = {Dynamic alpha-invariants of del {P}ezzo surfaces},
   JOURNAL = {Int. Math. Res. Not. IMRN},
  FJOURNAL = {International Mathematics Research Notices. IMRN},
      YEAR = {2016},
    NUMBER = {10},
     PAGES = {2994--3028},
      ISSN = {1073-7928,1687-0247},
   MRCLASS = {14J26 (32Q20 53C55)},
  MRNUMBER = {3551828},
MRREVIEWER = {Sergiy\ Koshkin},
       DOI = {10.1093/imrn/rnv229},
       URL = {https://doi.org/10.1093/imrn/rnv229},
}

@article {GMGZ18,
    AUTHOR = {Gallardo, Patricio and Martinez-Garcia, Jesus and Zhang,
              Zheng},
     TITLE = {Compactifications of the moduli space of plane quartics and
              two lines},
   JOURNAL = {Eur. J. Math.},
  FJOURNAL = {European Journal of Mathematics},
    VOLUME = {4},
      YEAR = {2018},
    NUMBER = {3},
     PAGES = {1000--1034},
      ISSN = {2199-675X,2199-6768},
   MRCLASS = {14L24 (14J17 14J28 14Q10 32G20)},
  MRNUMBER = {3851127},
MRREVIEWER = {Ronan\ Terpereau},
       DOI = {10.1007/s40879-018-0248-7},
       URL = {https://doi.org/10.1007/s40879-018-0248-7},
}

@article {zha22,
    AUTHOR = {Zhao, Junyan},
     TITLE = {Moduli of genus six curves and {K}-stability},
   JOURNAL = {Trans. Amer. Math. Soc. Ser. B},
  FJOURNAL = {Transactions of the American Mathematical Society. Series B},
    VOLUME = {11},
      YEAR = {2024},
     PAGES = {863--900},
      ISSN = {2330-0000},
   MRCLASS = {14H10 (14D23 14J10 14J17 14L24)},
  MRNUMBER = {4741264},
MRREVIEWER = {H.\ Lange},
       DOI = {10.1090/btran/195},
       URL = {https://doi.org/10.1090/btran/195},
}

@article {LZ25,
    AUTHOR = {Liu, Yuchen and Zhao, Junyan},
     TITLE = {K-moduli of {F}ano threefolds and genus four curves},
   JOURNAL = {J. Reine Angew. Math.},
  FJOURNAL = {Journal f\"ur die Reine und Angewandte Mathematik. [Crelle's
              Journal]},
    VOLUME = {824},
      YEAR = {2025},
     PAGES = {1--38},
      ISSN = {0075-4102,1435-5345},
   MRCLASS = {14J45 (14D20 14L24)},
  MRNUMBER = {4926940},
MRREVIEWER = {James\ McKernan},
       DOI = {10.1515/crelle-2025-0016},
       URL = {https://doi.org/10.1515/crelle-2025-0016},
}

@misc{Pap_code, 
title={Variations of {GIT} quotients for complete intersections package},
DOI={10.5526/ERDR-00000155}, 
publisher={University of Essex}, 
author={Theodoros Stylianos 
Papazachariou}, 
year={2022}, 
howpublished = {\url{https://doi.org/10.5526/ERDR-00000155}}}

@book{Dolgachev_2004,
    AUTHOR = {Dolgachev, Igor},
     TITLE = {Lectures on invariant theory},
    SERIES = {London Mathematical Society Lecture Note Series},
    VOLUME = {296},
 PUBLISHER = {Cambridge University Press, Cambridge},
      YEAR = {2003},
     PAGES = {xvi+220},
      ISBN = {0-521-52548-9},
   MRCLASS = {14L24 (13A50 14L30 15A72)},
  MRNUMBER = {2004511},
MRREVIEWER = {Michel\ Brion},
       DOI = {10.1017/CBO9780511615436},
       URL = {https://doi.org/10.1017/CBO9780511615436},
}

@inproceedings{takahashi_2013,
  title={Geometric Properties of Plane Quartics (Computer Algebra : The Algorithms, Implementations and the Next Generation)},
  author={Takahashi, Tadashi},
  booktitle={RIMS Kokyuroku},
  pages={163--165},
  year={2013},
  number={1843} 
}

@article {Arnold1976,
    AUTHOR = {Arnold, V. I.},
     TITLE = {Local normal forms of functions},
   JOURNAL = {Invent. Math.},
  FJOURNAL = {Inventiones Mathematicae},
    VOLUME = {35},
      YEAR = {1976},
     PAGES = {87--109},
      ISSN = {0020-9910,1432-1297},
   MRCLASS = {58C25 (57D70)},
  MRNUMBER = {467795},
MRREVIEWER = {J.\ S.\ Joel},
       DOI = {10.1007/BF01390134},
       URL = {https://doi.org/10.1007/BF01390134},
}

@book {Kol13,
    AUTHOR = {Koll\'ar, J\'anos},
     TITLE = {Singularities of the minimal model program},
    SERIES = {Cambridge Tracts in Mathematics},
    VOLUME = {200},
      NOTE = {With a collaboration of S\'andor Kov\'acs},
 PUBLISHER = {Cambridge University Press, Cambridge},
      YEAR = {2013},
     PAGES = {x+370},
      ISBN = {978-1-107-03534-8},
   MRCLASS = {14E30 (14B05)},
  MRNUMBER = {3057950},
MRREVIEWER = {Tommaso\ De Fernex},
       DOI = {10.1017/CBO9781139547895},
       URL = {https://doi.org/10.1017/CBO9781139547895},
}

@incollection {AH11,
    AUTHOR = {Abramovich, Dan and Hassett, Brendan},
     TITLE = {Stable varieties with a twist},
 BOOKTITLE = {Classification of algebraic varieties},
    SERIES = {EMS Ser. Congr. Rep.},
     PAGES = {1--38},
 PUBLISHER = {Eur. Math. Soc., Z\"urich},
      YEAR = {2011},
      ISBN = {978-3-03719-007-4},
   MRCLASS = {14D23 (14D22 14E30 14J10)},
  MRNUMBER = {2779465},
MRREVIEWER = {Arvid\ Perego},
       DOI = {10.4171/007-1/1},
       URL = {https://doi.org/10.4171/007-1/1},
}

@article{Pan25,
  title={Moduli of multi-uniformized stacks and {S}eifert $\mathbb{G}^d_m$-bundles},
  author={Pan, Zhengkai},
  journal={arXiv preprint arXiv:2506.22573},
  year={2025}
}

@article {Arnol_d_1975,
    AUTHOR = {Arnold, V. I.},
     TITLE = {Critical points of smooth functions, and their normal forms},
   JOURNAL = {Uspehi Mat. Nauk},
  FJOURNAL = {Akademija Nauk SSSR i Moskovskoe Matemati\v ceskoe Ob\v s\v
              cestvo. Uspehi Matemati\v ceskih Nauk},
    VOLUME = {30},
      YEAR = {1975},
    NUMBER = {5(185)},
     PAGES = {3--65},
      ISSN = {0042-1316},
   MRCLASS = {58C25 (32C40 57D70)},
  MRNUMBER = {420689},
MRREVIEWER = {D.\ Siersma},
}

@article {durfee,
    AUTHOR = {Durfee, Alan H.},
     TITLE = {Fifteen characterizations of rational double points and simple
              critical points},
   JOURNAL = {Enseign. Math. (2)},
  FJOURNAL = {L'Enseignement Math\'ematique. Revue Internationale. 2e
              S\'erie},
    VOLUME = {25},
      YEAR = {1979},
    NUMBER = {1-2},
     PAGES = {131--163},
      ISSN = {0013-8584},
   MRCLASS = {14B05 (32B30 32C40 57Q45 58C27)},
  MRNUMBER = {543555},
MRREVIEWER = {P.\ Orlik},
}

\end{document}